\documentclass[11pt]{article}
\usepackage{fullpage}
\usepackage{amsmath,amssymb,amsthm}
\usepackage[english]{babel}
\usepackage[left=3cm,right=3cm,top=3cm,bottom=4cm]{geometry}

\usepackage{enumitem}
\usepackage{tikz,tikz-cd}
\usetikzlibrary{positioning}
\usepackage{braket}
\usepackage{mathtools,stmaryrd}

\SetSymbolFont{stmry}{bold}{U}{stmry}{m}{n}
\usepackage{lmodern,bm,bbm,mathrsfs}
\usepackage{manfnt}

\usepackage[scr=esstix]{mathalfa}

\usepackage[textsize=tiny]{todonotes}

\usepackage[backref=page, colorlinks=true, linkcolor=blue, citecolor=red, menucolor=green]{hyperref}
\usepackage{setspace}

\theoremstyle{plain}
\newtheorem*{theorem}{Theorem}
\newtheorem*{proposition}{Proposition}
\newtheorem*{lemma}{Lemma}
\newtheorem*{corollary}{Corollary}

\theoremstyle{definition}
\newtheorem*{definition}{Definition}
\newtheorem*{example}{Example}
\theoremstyle{remark}
\newtheorem*{remark}{Remark}

\numberwithin{equation}{section}

\renewcommand{\hat}[1]{\widehat{#1}}
\newcommand{\cat}[1]{\mathscr{#1}}

\renewcommand{\vec}[1]{\bm{#1}}
\renewcommand{\tilde}[1]{\widetilde{#1}}
\renewcommand{\bar}[1]{\overline{#1}}

\newcommand{\bC}{\mathbb{C}}
\newcommand{\bE}{\mathbb{E}}
\newcommand{\bF}{\mathbb{F}}
\newcommand{\bk}{\mathbbm{k}}
\newcommand{\bK}{\mathbb{K}}
\newcommand{\bfK}{\mathbf{K}}
\newcommand{\bL}{\mathbb{L}}
\newcommand{\bM}{\mathbb{M}}
\newcommand{\bP}{\mathbb{P}}
\newcommand{\bQ}{\mathbb{Q}}

\newcommand{\bZ}{\mathbb{Z}}
\newcommand{\cA}{\mathcal{A}}

\newcommand{\cE}{\mathcal{E}}
\newcommand{\cF}{\mathcal{F}}
\newcommand{\cG}{\mathcal{G}}
\newcommand{\cK}{\mathcal{K}}
\newcommand{\cL}{\mathcal{L}}
\newcommand{\cN}{\mathcal{N}}
\newcommand{\cO}{\mathcal{O}}

\newcommand{\cV}{\mathcal{V}}

\newcommand{\fF}{\mathfrak{F}}
\newcommand{\fM}{\mathfrak{M}}
\newcommand{\fN}{\mathfrak{N}}
\newcommand{\fX}{\mathfrak{X}}
\newcommand{\fY}{\mathfrak{Y}}

\newcommand{\sA}{\mathsf{A}}
\newcommand{\se}{\mathsf{e}}
\newcommand{\sG}{\mathsf{G}}
\newcommand{\sT}{\mathsf{T}}
\newcommand{\sZ}{\mathsf{Z}}
\newcommand{\sz}{\mathsf{z}}

\newcommand{\pl}{\mathrm{pl}}

\newcommand{\sst}{\mathrm{sst}}
\newcommand{\st}{\mathrm{st}}
\newcommand{\vir}{\mathrm{vir}}
\newcommand{\loc}{\mathrm{loc}}
\newcommand{\fix}{\mathsf{fix}}
\newcommand{\di}{\partial}
\newcommand{\vac}{\mathbf{1}}
\DeclareMathOperator{\Fr}{Fr}
\DeclareMathOperator{\cFr}{\mathcal{F}{\it r}}
\DeclareMathOperator{\fr}{fr}

\DeclareMathOperator{\End}{End}

\DeclareMathOperator{\diag}{diag}
\DeclareMathOperator{\GL}{\mathsf{GL}}
\DeclareMathOperator{\PGL}{\mathsf{PGL}}
\DeclareMathOperator{\rank}{rank}
\DeclareMathOperator{\ad}{ad}
\DeclareMathOperator{\ch}{ch}
\DeclareMathOperator{\Res}{Res}
\DeclareMathOperator{\id}{id}
\DeclareMathOperator{\cone}{cone}

\DeclareMathOperator{\Hom}{Hom}

\DeclareMathOperator{\Pic}{Pic}

\DeclareMathOperator{\tot}{tot}

\DeclareMathOperator{\Ext}{Ext}
\DeclareMathOperator{\im}{im}
\DeclareMathOperator{\Char}{\mathsf{char}}
\DeclareMathOperator{\supp}{supp}
\DeclareMathOperator{\Frac}{Frac}
\DeclareMathOperator{\sing}{sing}
\DeclareMathOperator{\reg}{reg}
\DeclareMathOperator{\Spec}{Spec}
\DeclarePairedDelimiter{\inner}{\langle}{\rangle}
\DeclarePairedDelimiterX{\lseries}[1]{(}{)}{\mkern-2mu\delimsize(#1\delimsize)\mkern-2mu}
\DeclarePairedDelimiterX{\pseries}[1]{[}{]}{\mkern-2mu\delimsize[#1\delimsize]\mkern-2mu}
\DeclarePairedDelimiterX{\NO}[1]{{\vcentcolon}}{{\vcentcolon}\,}{{#1}}

\tikzset{%
  vertex/.style={shape=circle,fill=black,minimum size=6pt,inner sep=0},
  framing/.style={shape=rectangle,fill=black,minimum size=6pt,inner sep=0},
  baseline={([yshift=-0.8ex]current bounding box.center)}
}

\title{Equivariant K-theoretic enumerative invariants and wall-crossing formulae in abelian categories}
\author{Henry Liu}
\date{\today}

\begin{document}

\maketitle

\begin{abstract}
  We provide a general framework for wall-crossing of equivariant
  K-theoretic enumerative invariants of appropriate moduli stacks
  $\fM$, by lifting Joyce's homological universal wall-crossing
  \cite{Joyce2021} to K-theory and to include equivariance. The
  primary new result is that appropriate K-homology groups of $\fM$
  are equivariant multiplicative vertex algebras.
\end{abstract}

\renewcommand{\baselinestretch}{1}\normalsize
\tableofcontents
\renewcommand{\baselinestretch}{1.25}\normalsize

\section{Introduction}

Given a $\bC$-linear abelian category $\cA$, enumerative geometry may
be phrased very generally as the problem of counting objects in
$\cat{A}$. In this paper, we primarily imagine $\cA$ to be the
category of coherent sheaves on a smooth projective $\bC$-scheme $X$
of dimension $\le 2$. To make a connection with geometry, assume $\cA$
has an associated algebraic moduli stack $\fM = \bigsqcup_\alpha
\fM_\alpha$. For a (weak) stability condition $\tau$ on $\cA$, the
$\tau$-stable (resp. $\tau$-semistable) objects form substacks (resp.
open substacks) $\fM_\alpha^{\st}(\tau) \subset
\fM_\alpha^{\sst}(\tau) \subset \fM_\alpha^{\pl}$ of a
$\bC^\times$-rigidifcation of $\fM_\alpha$, for each $\alpha$. If
$\alpha$ has no strictly $\tau$-semistable objects, i.e.
$\fM^\st_\alpha(\tau) = \fM^{\sst}_\alpha(\tau)$, and
$\fM_\alpha^\pl$ has a perfect obstruction theory, then following
\cite{Behrend1997} there is a {\it virtual fundamental class}
\begin{equation} \label{eq:chow-invariants}
  [\fM_\alpha^{\sst}(\tau)]^{\vir} \in A_*(\fM_\alpha^{\sst}(\tau))_\bQ.
\end{equation}
Pairing it with tautological classes and integrating produces
important enumerative quantities invariant under deformations of $X$,
though the virtual class itself is more fundamental than any set of
such invariants.

This setup already encompasses entire worlds within modern enumerative
geometry, particularly Donaldson-type theories; see
\cite{Mochizuki2009, Maulik2006, Cao2022} for a sample. Notably,
however, it excludes Gromov--Witten theory, whose underlying category
of stable maps is not abelian.

Wall-crossing studies relationships between the classes
\eqref{eq:chow-invariants} for different $\tau$; the name comes from
the simple case where the space of all stability conditions is
separated by codimension-$1$ {\it walls} into chambers in which the
stable=semistable locus and its virtual class are constant. Many
interesting correspondences between different enumerative setups, e.g.
DT/PT \cite{Toda2010}, can be viewed as wall-crossing formulae. To
study wall-crossing, Joyce \cite{Joyce2021} constructs classes
\begin{equation} \label{eq:homology-semistable-invariants}
  [\fM^{\sst}_\alpha(\tau)]^{\vir} \in H_*(\fM^\pl_\alpha)_\bQ
\end{equation}
for {\it all} $\alpha$ and $\tau$ inductively, using auxiliary
categories $\tilde{\cat{A}}^{\Fr}$ of pairs in the style of
Joyce--Song \cite{Joyce2012}. These classes agree with
\eqref{eq:chow-invariants} whenever comparable. There are then
universal, though somehat unwieldy, formulae describing their behavior
under variation of $\tau$, involving a vertex algebra and a Lie
algebra structure on the homology groups of $\fM$ and $\fM^\pl$
respectively.

Importantly, the {\it homology} groups where these classes live have
very different functoriality properties than, say, Chow groups. While
Borel--Moore-type homology theories have only proper pushforwards and
flat pullbacks, Joyce's construction crucially relies on the arbitrary
pushforwards present in homology.

The goal of this paper is to port Joyce's constructions into (a dual
of) {\it equivariant algebraic K-theory}. K-theory is a multiplicative
refinement of ordinary cohomology, and equivariance allows for
powerful tools such as localization with respect to a torus $\sT$
acting on $\fM$, so e.g. $X$ may now be quasi-projective; see
\cite{Okounkov2017} for a salient introduction. To accomplish this
goal, we introduce two new ingredients which may also be of
independent interest.
\begin{itemize}
\item We provide axioms defining a suitable but rather abstract notion
  of a {\it commutative operational equivariant K-homology theory}
  $\bfK^\sT_\circ(-)$ for algebraic stacks $\fM$ with $\sT$-action;
  such theories behave like the dual of the K-group $K^\circ_\sT(\fM)$
  of $\sT$-equivariant perfect complexes on $\fM$, with a finiteness
  condition, and serves as K-theoretic analogues of $H_*(\fM)$.
  Besides having arbitrary pushforwards, they are also natural homes
  for ``universal'' K-theoretic enumerative invariants, namely
  functions $\sZ_\alpha(\tau)$ which take a perfect complex on $\fM$
  (e.g. a tautological bundle), tensors it with the virtual structure
  sheaf of $\fM^{\sst}_\alpha(\tau)$ whenever it exists, and pushes
  forward to the base. For practical use, especially in wall-crossing,
  we provide a ``concrete'' instantiation $K_\circ^\sT(-)$ such a
  K-homology theory.

\item We define and put the structure of an {\it equivariant
  multiplicative vertex algebra} on $\bfK^\sT_\circ(\fM)$, and
  generalize Borcherd's construction \cite[\S 4]{Borcherds1986} to
  induce a Lie bracket on a certain quotient $\bfK^\sT_\circ(\fM)^\pl$,
  related but not necessarily equal to $\bfK^\sT_\circ(\fM^\pl)$. This
  is our analogue of the ordinary vertex algebra structure on
  $H_*(\fM)_\bQ$ and induced Lie bracket on $H_*(\fM^\pl)_\bQ$
  \cite[\S 4.2]{Joyce2021}. Equivariance and multiplicativity means
  that the vertex operation now has poles over the torus $\sT$. In
  particular, the OPE of two fields $A(z)$ and $B(w)$ has poles at
  $z/w = t_i$ for a finite number of equivariant weights $t_i$ which
  may depend on $A$ and $B$, in contrast to the ordinary case of poles
  only at $z - w = 0$.
\end{itemize}
With these two ingredients in hand, Joyce's construction of
\eqref{eq:homology-semistable-invariants} proceeds almost verbatim to
give K-theoretic classes
\[ \sZ_\alpha(\tau) \in K^\sT_\circ(\fM^\pl_\alpha)_{\loc,\bQ} \]
for all $\alpha$ and $\tau$. These classes satisfy analogues of all
the wall-crossing formulae in \cite{Joyce2021}. Since the input from
geometry and the choice of cohomology theory is more or less the same
for the construction of these classes as for the proof of their
wall-crossing formulae, this paper focuses only on their construction.

Many variations on this construction are possible. In particular, when
$\dim X = 3$ and $\fM$ has a {\it symmetric} obstruction theory, all
constructions in this paper continue to hold after an appropriate
symmetrization. Special cases already have powerful applications,
including a construction \cite{Liu2023} of refined semistable
Vafa--Witten invariants satisfying the main conjecture of
\cite{Thomas2020}, and a proof of the 3-fold K-theoretic DT/PT vertex
correspondence \cite{Kuhn2023}.

\subsection{Outline of the paper}

Section~\ref{sec:setup} sets the stage. In \S\ref{sec:k-theory}, we
review equivariant K-groups $K_\sT(\fM)$ and $K^\circ_\sT(\fM)$,
K-theoretic Euler and Chern classes of equivariant (virtual) vector
bundles, and (perfect) obstruction theories and their smooth
pullbacks. In \S\ref{sec:operational-K-homology}, we define
``abstract'' operational equivariant K-homology groups
$\bfK^\sT_\circ(-)$ of algebraic stacks, along with a ``concrete''
instantiation $K^\sT_\circ(-)$ of them which is suitable for
applications to enumerative geometry, in particular wall-crossing. In
\S\ref{sec:operational-k-homology-BG}, we compute $\bfK^\sT_\circ(-)$
of the classifying stack $[*/\bC^\times]$, both for use in
multiplicative vertex algebras and also for comparison with the
K-homology of its topological realization $BU(1) = \bC\bP^\infty$.

Section~\ref{sec:equivariant-mVOA} presents the equivariant and
multiplicative analogue of ordinary vertex algebras. In
\S\ref{sec:series-and-residues} and
\S\ref{sec:equivariant-mVOA-general-theory}, we give the general
definition and properties, in complete analogy with the theory of
ordinary vertex algebras, including OPEs and an induced Lie bracket
$[-, -]$ on a suitable quotient. In \S\ref{sec:mVOA-on-K-homology}, we
construct a equivariant multiplicative vertex algebra structure on
$\bfK^\sT_\circ(\fM)$ for Artin stacks $\fM$ which we call {\it graded
  monoidal}. This induces the Lie bracket on $\bfK^\sT_\circ(\fM)^\pl$
used in wall-crossing formulae. In \S\ref{sec:mVOA-on-K-theory}, we
observe that the same construction on $K_\sT(\fM)$ is well-defined in
special cases and makes it into a {\it holomorphic} equivariant
multiplicative vertex algebra. Under an extra assumption, these
correspond to graded algebras and should be compared with K-theoretic
Hall algebras.

Section~\ref{sec:wall-crossing} is mostly an exposition of the
construction \cite[Theorem 5.7]{Joyce2021} of semistable invariants,
but rewritten (in slightly less generality, for simplicity) in our
equivariant K-theoretic framework. Several components of Joyce's big
machine have been black-boxed and what remains is only the
construction of an auxiliary category/stack of pairs, in
\S\ref{sec:moduli-stacks}, and the key geometric argument using
localization on a master space for these pairs, in
\S\ref{sec:semistable-invariants}. Hopefully this makes for an
accessible introduction to \cite{Joyce2021}.

Appendix~\ref{sec:residue-maps} is a discussion and characterization
of the K-theoretic residue map which mediates the passage from the
equivariant multiplicative vertex algebra $\bfK^\sT_\circ(\fM)$ to the
Lie algebra $\bfK^\sT_\circ(\fM)^\pl$.

\subsection{Acknowledgements}

A substantial portion of this project is either inspired by or based
on D. Joyce's work \cite{Joyce2021}, and benefitted greatly from
discussions with him as well as A. Bojko, C.-J. Bu, I. Karpov, N.
Kuhn, M. Moreira, A. Okounkov, F. Thimm, and M. Upmeier. This research
was supported by the Simons Collaboration on Special Holonomy in
Geometry, Analysis and Physics, and the World Premier International
Research Center Initiative (WPI), MEXT, Japan.

\section{The K-homology group}
\label{sec:setup}

\subsection{Review of K-theory}
\label{sec:k-theory}

\subsubsection{}

We work over $\bC$. Let $\sG$ be an algebraic group (i.e. group scheme
of finite type) and let $\cat{Art}_\sG$ be the strict $2$-category of
Artin stacks with $\sG$-action \cite{Romagny2005}. We will rarely use
the $2$-categorical structure. Equalities of morphisms will take place
in the homotopy category of $\cat{Art}_\sG$. When $\sG$ is trivial, we
omit the subscript $\sG$ from all relevant objects.

\subsubsection{}

\begin{definition}
  Given $\fX \in \cat{Art}_\sG$, let $\cat{Coh}_\sG(\fX)$ (resp.
  $\cat{Perf}_\sG(\fX)$) be the abelian (resp. triangulated) category
  of $\sG$-equivariant coherent sheaves (resp. perfect complexes) on
  $\fX$. Define the {\it K-theory groups}
  \begin{align*}
    K_\sG(\fX) &\coloneqq K_0(\cat{Coh}_{\sG}(\fX)) \\
    K_\sG^\circ(\fX) &\coloneqq K_0(\cat{Perf}_{\sG}(\fX))
  \end{align*}
  Tautologically, $K_\sG(\fX) = K([\fX/\sG])$ and similarly for
  $K_\sG^\circ$. Both $K^\circ_\sG(\fX)$ and $K_\sG(\fX)$ are modules
  for the representation ring $\bk_\sG \coloneqq K_\sG(*)$. When $\sG$
  is a (split) torus $\sT \coloneqq (\bC^\times)^r$,
  \[ \bk_\sT \cong \bZ[t^\mu : \mu \in \Char(\sT)] \]
  where $\Char(\sT) \coloneqq \Hom(\sT, \bC^\times)$ is the character
  lattice. We refer to the $t^\mu$ as {\it weights} of $\sT$.
\end{definition}

\subsubsection{}

In the literature, our $K(-)$ is often called {\it G-theory} and our
$K^\circ(-)$ is often called {\it (Thomason--Trobaugh) K-theory}. One
may also consider the exact category $\cat{Vect}_\sG(\fX)$ of
$\sG$-equivariant vector bundles on $\fX$, whose K-group we call {\it
  Quillen K-theory}. The inclusion $\cat{Vect}_{\sG} \subset
\cat{Perf}_{\sG}$ induces a map
\[ K_0(\cat{Vect}_\sG(\fX)) \to K^\circ_\sG(\fX). \]
This is an isomorphism if $\fX$ has the {\it $\sG$-equivariant
  resolution property}, namely if every $\sG$-equivariant coherent
sheaf on $\fX$ is a quotient of a $\sG$-equivariant vector bundle on
$\fX$ \cite{Totaro2004}.

While many of our constructions take place in Quillen K-theory, e.g.
Definition~\ref{def:k-theoretic-chern-class}, Thomason--Trobaugh
K-theory is better-behaved for stacks; see \cite[\S 2]{Khan2022}. So
we choose to use Thomason--Trobaugh K-theory in this paper.

\subsubsection{}

We consider $K_\sG^\circ(-)$ only for Artin stacks $\fX$ locally of
finite type. We list some of its properties.
\begin{itemize}
\item The (derived) tensor product induces $\otimes\colon
  K_\sG^\circ(\fX) \otimes K_\sG^\circ(\fX) \to K_\sG^\circ(\fX)$
  making $K_\sG^\circ(\fX)$ into a $\bk_\sG$-algebra.
\item A $\sG$-equivariant morphism $f\colon \fX \to \fY$ induces a
  pullback $f^*\colon K_\sG^\circ(\fY) \to K_\sG^\circ(\fX)$. It is
  compatible with tensor product:
  \begin{equation} \label{eq:k-theory-pullback-tensor-product}
    f^*(\cE \otimes \cE') = f^*\cE \otimes f^*\cE', \qquad \cE, \cE' \in K_\sT^\circ(\fY).
  \end{equation}
  In particular, $f^*$ is a morphism of $\bk_\sG$-modules.
\end{itemize}
We do not consider pushforwards, in general.

\subsubsection{}

We consider $K_\sG(-)$ mostly only for schemes $X$ of finite type. We
list some of its properties.
\begin{itemize}
\item The (derived) tensor product induces $\otimes\colon
  K_\sG^\circ(X) \otimes K_\sG(X) \to K_\sG(X)$
  making $K_\sG(X)$ into a $K_\sG^\circ(X)$-module.
\item A proper $\sG$-equivariant morphism $f\colon X
  \to Y$ induces a derived pushforward $f_*\colon K_\sG(X) \to
  K_\sG(Y)$. It satisfies a {\it projection formula}
  \[ f_*(f^*\cE \otimes \cF) = \cE \otimes f_*(\cF), \qquad \cE \in K_\sG^\circ(Y), \; \cF \in K_\sG(X). \]
  In particular, $f_*$ is a morphism of $\bk_\sG$-modules.
\end{itemize}

\subsubsection{}

\begin{definition}
  Let
  \[ I^\circ_\sG(\fX) \subset K_0(\cat{Vect}_\sG(\fX)) \]
  be the {\it augmentation ideal} of rank-$0$ elements. If $\fX = *$,
  we just write $I^\circ_\sG$.

  Since $K_\sG(\fX)$ is also a $K_0(\cat{Vect}_\sG(\fX))$-module by
  tensor product, $I^\circ_\sG(\fX)$ acts on $K_\sG(\fX)$. But
  typically, when working with $I^\circ_\sG(\fX)$, we will only be
  interested in the case where $\fX$ has the $\sG$-equivariant
  resolution property, so that $I^\circ_\sG(\fX) \subset
  K^\circ_\sG(\fX)$.
\end{definition}

\subsubsection{}

\begin{lemma} \label{lem:augmentation-ideal-vanishing-for-schemes}
  Let $X$ be a finite-type scheme. Then
  \begin{equation} \label{eq:augmentation-ideal-vanishing-for-schemes}
    I^\circ(X)^{\otimes N}K(X) = 0, \qquad \forall N \gg 0.
  \end{equation}
\end{lemma}

\begin{proof}
  We construct such an $N \in \bZ$. Note that $I^\circ(X)$ is
  generated by elements $\cO_X^{\oplus \rank \cE} - \cE$ where $\cE$
  is a vector bundle. Let $\cF$ be a coherent sheaf on $X$. Pick an
  open $U \subset X$ where
  \[ \cE|_U \cong \cO_U^{\oplus \rank \cE} \]
  is trivial. (Here, we really use that $X$ is a scheme; for instance,
  no such $U$ exists for a non-trivial line bundle on
  $[*/\bC^\times]$.) By the long exact sequence
  \[ \cdots \to K(X \setminus U) \xrightarrow{i_*} K(X) \xrightarrow{j^*} K(U) \to 0 \]
  associated to the inclusions $X \setminus U \xhookrightarrow{i} X
  \xhookleftarrow{j} U$, we see that $\supp (\cO_X^{\oplus \rank \cE}
  - \cE) \otimes \cF \subset X \setminus U$. Repeating with the
  coherent sheaf $(\cO_X^{\oplus \rank \cE} - \cE) \otimes \cF$ on $X
  \setminus U$ in place of the coherent sheaf $\cF$ on $X$, it follows
  that
  \[ \dim \supp (\cO_X^{\oplus \rank \cE} - \cE)^{\otimes n} \otimes \cF < \dim \supp \cF \]
  for some $n \in \bZ$ which depends only on the number of irreducible
  components in $X$. Since $\dim X < \infty$, we are done by induction
  on $\dim X$.
\end{proof}

\subsubsection{}

\begin{definition} \label{def:k-theoretic-chern-class}
  Given a $\sG$-equivariant vector bundle $\cE$ on $\fX$, let
  \[ \wedge_{-s}^\bullet(\cE) \coloneqq \sum_i (-s)^i \wedge^i(\cE) = \prod_\cL (1 - s\cL) \in K_0(\cat{Vect}_\sG(\fX))[s^\pm] \]
  be its exterior algebra graded by a formal variable $s$; the product
  ranges over (K-theoretic) Chern roots $\cL$ of $\cE$, by the
  splitting principle. Extend $\wedge_{-s}^\bullet(-)$ to
  $K_0(\cat{Vect}_\sG(\fX))$:
  \[ \wedge_{-s}^\bullet(\cE_1 - \cE_2) \coloneqq \frac{\wedge_{-s}^\bullet(\cE_1)}{\wedge_{-s}^\bullet(\cE_2)} \in K_0(\cat{Vect}_\sG(\fX))\lseries*{(1-s)^{-1}}, \]
  using the splitting principle and the formula
  \begin{equation} \label{eq:k-theoretic-inverse-chern-root}
    \frac{1}{\wedge_{-s}^\bullet(\cL)} \coloneqq \frac{1}{(1 - s) - s(\cL - 1)} \coloneqq \sum_{k \ge 0} \frac{s^k}{(1 - s)^{k+1}} (\cL - 1)^k,
  \end{equation}
  for Chern roots $\cL$ of $\cE_2$, to define the inverse of
  $\wedge_{-s}^\bullet(\cE_2)$. For $\cE \in
  K_0(\cat{Vect}_\sG(\fX))$, define
  \[ c_i^K(\cE) \coloneqq \text{coefficient of } (1 - s)^{\rank \cE-i} \text{ in } \wedge_{-s}^\bullet(\cE^\vee) \]
  and $c_{\text{rank}}^K(\cE) \coloneqq c_{\rank \cE}^K(\cE)$. These
  are K-theoretic analogues of {\it Conner--Floyd--Chern classes}
  \cite{Conner1966}, up to a factor of $\det\cE$. 
\end{definition}

\subsubsection{}

\begin{lemma} \label{lem:k-theoretic-chern-class}
  Let $\fX \in \cat{Art}_\sG$ and $\cE \in K_0(\cat{Vect}_\sG(\fX))$.
  \begin{enumerate}
  \item If $\cE$ is a vector bundle, then $c_{\text{rank}}^K(\cE) =
    \wedge_{-1}^\bullet(\cE^\vee) = \se(\cE)$.
  \item $c_i^K(\cE + \cO_\fX) = c_i^K(\cE)$ for all $i \in \bZ$;
  \item $c_{\text{rank}}^K(\cE + \cE') = c_{\text{rank}}^K(\cE)
    c_{\text{rank}}^K(\cE')$ for any $\cE' \in
    K_0(\cat{Vect}_\sG(\fX))$.
  \end{enumerate}
\end{lemma}

\begin{proof}
  Straightforward.
\end{proof}
  
Comparing K-theory to cohomology, for $\cE$ which is a vector bundle,
$c_{\text{rank}}^K$ is the analogue of the top (Connor--Floyd--)Chern
class, and $\se$ is the analogue of the Euler class. For general
$\cE$, these are two different classes, and the same is true of
$c_{\text{rank}}^K$ and $\se$.

\subsubsection{}

\begin{definition} \label{def:k-theoretic-euler-class}
  Let $\sT$ be a torus acting trivially on $\fX$. Then every
  $\sT$-equivariant vector bundle on $\fX$ decomposes as
  \[ \cE = \bigoplus_{\mu} t^\mu \otimes \cE_\mu, \qquad \cE_\mu \in \cat{Vect}(\fX) \]
  for a finite set of $\sT$-weights $\mu$. Its {\it K-theoretic Euler
    class} is the operator
  \[ \se(\cE) \otimes \coloneqq \bigotimes_\mu \wedge_{-t^{-\mu}}^\bullet(\cE_\mu^\vee) \otimes\colon K_\sT(\fX) \to K_\sT(\fX) \]
  Extend this multiplicatively to the subgroup of
  $K_0(\cat{Vect}_\sT(\fX))$ consisting of (virtual) bundles $\cE$
  with no trivial $\sT$-weight subbundles:
  \begin{equation} \label{eq:k-theoretic-euler-class}
    \se(\cE) \otimes \in \End_{\bk_\sT}(K_\sT(\fX))\pseries*{(1 - t^\mu)^{-1} : 0 \neq \mu \in \Char(\sT)}.
  \end{equation}
  For a given $\cE$, only the $\sT$-weights $t^\mu$ occuring in $\cE$
  will appear in $\se(\cE) \otimes$, so in particular only finitely
  many $(1 - t^\mu)^{-1}$ need to be adjoined. Coefficients in
  \eqref{eq:k-theoretic-euler-class} are operators of multiplication
  by various K-theoretic Connor--Floyd--Chern classes of the
  $\cE_\mu$.

  Note that $\se(-) \otimes$ is different from
  $\wedge_{-s}^\bullet(-)^\vee \otimes$, which may not even be
  well-defined at $s=1$.
\end{definition}

\subsubsection{}

\begin{lemma} \label{lem:k-theory-inverse-euler-class}
  Given $\cE$, there exists a function $M(N)$ such that $\lim_{N \to
    \infty} M(N) = \infty$ and
  \[ c_N^K(\cE) \in I^\circ(\fX)^{\otimes M(N)}. \]
  Consequently, if $\fX = X$ is a finite-type scheme, then
  \eqref{eq:k-theoretic-euler-class} truncates to a homomorphism
  \[ \se(\cE) \otimes\colon K_\sT(X) \to K_\sT(X)_{\loc}. \]
\end{lemma}

Recall that if $M$ is a $\bk_\sT$-module, then $M_{\loc}\coloneqq M
\otimes_{\bk_\sT} \bk_{\sT,\loc}$ where
\[ \bk_{\sT,\loc} \coloneqq \bk_\sT[(1 - t^\mu)^{-1} : 0 \neq \mu \in \Char(\sT)]. \]

\begin{proof}
  It suffices to check this claim for a vector bundle $\cE$ and the
  expansion of $1/\wedge_{-s}^\bullet(\cE)$ as a Laurent series in $(1
  - s)^{-1}$ using the splitting principle and
  \eqref{eq:k-theoretic-inverse-chern-root}. Observe that
  \begin{align*}
    \wedge_{-s}^\bullet(\cE)
    &= \prod_\cL \left((1 - \cL) + \cL (1 - s)\right) \\
    &\in \det(\cE) (1 - s)^{\rank \cE} \cdot (1 + I^\circ(X)[(1 - s)^{-1}])
  \end{align*}
  since the coefficient of $(1-s)^{-k}$ is a degree-$k$ symmetric
  polynomial in the variables $1 - \cL$. Therefore the coefficients of
  its inverse, as a series in $(1-s)^{-1}$, have the desired property.

  The remaining claim follows from the definition of $\se(\cE)
  \otimes$ and
  Lemma~\ref{lem:augmentation-ideal-vanishing-for-schemes}.
\end{proof}
  
\subsubsection{}

\begin{remark} \label{rem:k-theory-inverse-euler-class}
  Lemma~\ref{lem:k-theory-inverse-euler-class} may be extended to
  separated finite-type Deligne--Mumford stacks $X$, as follows. The
  difficulty is that, in \eqref{eq:k-theoretic-inverse-chern-root},
  multiplication by $1 - \cL$ may not be nilpotent; this is already
  clear for $X = [*/G]$ where $G$ is a finite group. However, it is
  true that $\cL^{\otimes m}$ descends to the coarse moduli space for
  some $m \in \bZ$, and
  Lemma~\ref{lem:augmentation-ideal-vanishing-for-schemes} continues
  to hold for arbitrary algebraic spaces, see e.g. \cite[Lemma
    2.4]{Edidin2000}. Consequently
  \eqref{eq:k-theoretic-inverse-chern-root} may be replaced by
  \begin{equation} \label{eq:k-theoretic-inverse-chern-root-DM}
    \frac{1}{1 - s\cL} = \sum_{k=0}^{m-1} \frac{(s\cL)^{\otimes k}}{1 - (s \cL)^{\otimes m}}
  \end{equation}
  and multiplication by each term in this finite sum is well-defined
  on $K_\sT(X)_{\loc}$ using \eqref{eq:k-theoretic-inverse-chern-root}
  because multiplication by $1 - \cL^{\otimes m}$ is now nilpotent.
\end{remark}

\subsubsection{}

\begin{definition} \label{def:obstruction-theories}
  Recall that a {\it $\sG$-equivariant obstruction theory} on $\fX \in
  \cat{Art}_\sG$ is an object $\bE_{\fX}$ and a morphism
  $\varphi\colon \bE_{\fX} \to \bL_{\fX}$, where $\bL_\fX$ is the
  cotangent complex \cite{Illusie1971}, in the derived category of
  $\sG$-equivariant sheaves on $\fX$ with quasi-coherent cohomology,
  such that $h^1(\varphi)$, $h^0(\varphi)$ are isomorphisms and
  $h^{-1}(\varphi)$ is surjective. If further $\bE_{\fX}$ is perfect
  of amplitude $[-1,1]$ then $\varphi$ is a {\it perfect} obstruction
  theory.

  Given a smooth $\sG$-equivariant morphism $f\colon \fX \to \fY$ of
  Artin stacks, two $\sG$-equivariant obstruction theories
  $\varphi\colon \bE_\fX \to \bL_{\fX}$ and $\phi\colon \bE_{\fY} \to
  \bL_{\fY}$ are {\it compatible under $f$} if they fit into a diagram
  \begin{equation} \label{eq:obstruction-theories-compatibility}
    \begin{tikzcd}
      \bL_f[-1] \ar{r}{\Xi} \ar[equal]{d} & f^*\bE_{\fY} \ar{r} \ar{d}{f^*\phi} & \bE_{\fX} \ar{r} \ar{d}{\varphi} & \bL_f \ar[equals]{d} \\
      \bL_f[-1] \ar{r} & f^*\bL_{\fY} \ar{r} & \bL_{\fX} \ar{r} & \bL_f
    \end{tikzcd}
  \end{equation}
  where both rows are exact triangles (the bottom one for the relative
  cotangent complex $\bL_f$).

  Note that $\varphi\colon \cone(\Xi) \to \bL_\fX$ can be constructed
  given only $\phi$ and $\Xi$ which make the leftmost square commute,
  in which case we say $\varphi$ is a {\it smooth pullback} of $\phi$
  along $f$. In general, if $\phi$ is constructed in a sufficiently
  functorial manner, then a natural $\Xi$ exists; see e.g.
  \cite[Appendix B]{Graber1999} \cite[Lemma 5.18]{Kuhn2021}.
\end{definition}

\subsubsection{}

If $X$ is a proper scheme (or Deligne--Mumford stack, using
Remark~\ref{rem:k-theory-inverse-euler-class}) acted on by a torus
$\sT$, and $X$ has a $\sT$-equivariant perfect obstruction theory
\cite{Behrend1997}, one can do equivariant enumerative geometry using:
\begin{itemize}
\item the virtual structure sheaf $\cO_X^\vir \in K_\sT(X)$
  \cite{Lee2004, Ciocan-Fontanine2009};
\item the Euler characteristic $\chi(X, -)\coloneqq \sum_i (-1)^i
  H^i(X, -)\colon K_\sT(X) \to \bk_\sT$;
\item the virtual torus localization theorem \cite{Graber1999, Qu2018}
  \begin{equation} \label{eq:virtual-cycle-localization}
    \cO_X^\vir = \iota_*\frac{\cO_{X^\sT}^\vir}{\se(\cN_\iota^\vir)} \in K_\sT(X)_{\loc}
  \end{equation}
  where $\iota\colon X^\sT \hookrightarrow X$ is the $\sT$-fixed
  locus, $\cN_\iota^\vir$ is the virtual normal bundle (see
  Remark~\ref{rem:localization-resolution-property} below), and the
  inverse is well-defined by
  Lemma~\ref{lem:k-theory-inverse-euler-class}.
\end{itemize}
If $X$ is not proper, then $\chi(X, -)$ is not well-defined in
general. But if $X^\sT$ is proper, then we define $\chi(X, \cO_X^\vir)
\in \bk_{\sT,\loc}$ by localization, i.e. using the right hand side of
\eqref{eq:virtual-cycle-localization}.

\subsubsection{}

\begin{remark} \label{rem:localization-resolution-property}
  In the literature for virtual localization as in
  \eqref{eq:virtual-cycle-localization}, it is common to require that
  $\cN_\iota^\vir$ has a (necessarily two-term) resolution by vector
  bundles, namely that it lies in $K_0(\cat{Vect}_\sT(X^\sT))$ instead
  of $K_\sT^\circ(X^\sT)$, so that its K-theoretic Euler class
  $\se(\cN_\iota^\vir)$ is well-defined. The easiest way to satisfy
  this requirement is to ensure $X^{\sT}$ has the resolution property,
  so that these two K-groups are equal. One may define K-theoretic
  Euler classes of perfect complexes in general and prove that
  \eqref{eq:virtual-cycle-localization} still holds, as explained by
  \cite{Aranha2024} for Chow homology. Then no extra requirement on
  $X^{\sT}$ or $\cN_\iota^\vir$ is necessary.
\end{remark}

\subsubsection{}

\begin{lemma}[Virtual projective bundle formula] \label{lem:projective-bundle-formula}
  Let $\pi\colon X \to Y$ be a $\sT$-equivariant morphism of
  finite-type schemes, with $\sT$-equivariant perfect obstruction
  theories compatible under $\pi$. If $\pi$ is a $\bP^{N-1}$-bundle,
  then
  \[ \pi_*\left(\cO_X^\vir \otimes \wedge_{-1}^\bullet(\bL_\pi)\right) = N \cdot \cO_Y^\vir \in K_\sT(Y). \]
\end{lemma}

\begin{proof}
  By virtual pullback \cite{Manolache2012, Qu2018} or otherwise,
  $\cO_X^\vir = \pi^*\cO_Y^\vir$. By projection formula, it suffices
  to compute $\pi_* \wedge_{-1}^\bullet(\bL_\pi)$. This is some
  combination of relative cohomology bundles $R^i\pi_*\Omega^j_\pi$,
  which are all canonically trivialized by powers of the hyperplane
  class. So it is enough to compute on fibers:
  \[ \chi\left(\bP^{N-1}, \wedge_{-1}^\bullet(\Omega_{\bP^{N-1}})\right) = \sum_{i,j=0}^N (-1)^{i+j} H^{i,j}(\bP^{N-1}) = N \in \bk_\sT. \]
  Note that the $H^{i,j}$ are valued in trivial $\sT$-representations
  by Hodge theory.
\end{proof}

\subsection{Operational K-homology}
\label{sec:operational-K-homology}

\subsubsection{}
\label{sec:k-homology-support-category}

Let $\sT$ be a torus. Our K-homology groups, to be defined below,
depend on the choice of a full $2$-subcategory
\[ \cat{C} \subset \cat{Art}_\sT \]
whose elements have trivial $\sT$-action and the resolution property,
such that $\cat{C}$ is closed under fiber products over $\Spec \bC$
and contains all projective schemes. (This last condition is so that
Proposition~\ref{prop:operational-k-homology-BGm} always holds.) For
instance, $\cat{C}$ could consist of all quasi-projective schemes
\cite{Totaro2004}.

The resolution property is only really used to define K-theoretic
Euler classes of perfect complexes and can be omitted with more care
(see Remark~\ref{rem:localization-resolution-property}), in which case
$\cat{C}$ could simply consist of all proper schemes. In particular,
in the vertex algebra construction of \S\ref{sec:mVOA-on-K-homology},
the resolution property plays a role only in
Lemma~\ref{sec:monoidal-stack-vertex-product-well-defined}.

With even more care, one can probably replace ``schemes'' with
``algebraic spaces'' as well.

\subsubsection{}

\begin{definition}[``Abstract'' K-homology] \label{def:operational-k-homology}
  The {\it $\sT$-equivariant operational K-homology group} of $\fX \in
  \cat{Art}_\sT$, of {\it elements with $\cat{C}$-theoretic support},
  is the $\bk_\sT$-module $\bK_\circ^\sT(\fX) = \bK_\circ^\sT(\fX;
  \cat{C})$ which consists of all collections
  \[ \phi\coloneqq \{K^\circ_\sT(\fX \times S) \xrightarrow{\phi_S} K^\circ_\sT(S)\}_{S \in \cat{Art}_\sT} \]
  of homomorphisms of $K^\circ_\sT(S)$-modules, for all $S \in
  \cat{Art}_\sT$ which we call the {\it base}, which obey the
  following axioms.
  \begin{enumerate}
  \item (Naturality) For any $\sT$-equivariant morphism $h\colon S \to
    S'$ of Artin stacks, there is a commutative square
    \[ \begin{tikzcd}
        K^\circ_{\sT}(\fX \times S') \ar{r}{(\id \times h)^*} \ar{d}{\phi_{S'}} & K^\circ_{\sT}(\fX \times S) \ar{d}{\phi_S} \\
        K^\circ_{\sT}(S') \ar{r}{h^*} & K^\circ_{\sT}(S).
      \end{tikzcd} \]
  \item (Equivariant localization) There exists $\fF_\phi \in \cat{C}$
    and a map $\fix_\phi\colon \fF_\phi \to \fX$ in $\cat{Art}_\sT$
    such that $\phi_S$ factors as
    \[ \begin{tikzcd}
        K_{\sT}^\circ(\fX \times S) \ar{rr}{\phi_S} \ar{dr}[swap]{(\fix_\phi \times \id)^*} && K_{\sT}^\circ(S) \\
        & K_{\sT}^\circ(\fF_\phi \times S) \ar{ur}[swap]{\phi_S^{\sT}}
      \end{tikzcd} \]
    for homomorphisms $\{\phi_S^{\sT}\}_S$ of
    $K_{\sT}^\circ(S)$-modules which themselves satisfy all other
    axioms, i.e. forming an element $\phi^\sT \in \bK_\circ^\sT(\fF_\phi;
    \cat{C})$.
  \item (Finiteness) for any base $S \in \cat{C}$,
    \begin{equation} \label{eq:operational-k-homology-finiteness-condition}
      \phi_S^\sT\left(I^\circ(\fF_\phi \times S)^{\otimes N}\right) = 0, \qquad \forall N \gg 0.
    \end{equation}
  \end{enumerate}
  The sum $\phi + \psi$ of two elements $\phi, \psi \in
  \bK_\circ^\sT(\fX)$ still satisfies the equivariant localization and
  finiteness axioms by setting $\fF_{\phi + \psi} \coloneqq \fF_\phi
  \sqcup \fF_\psi$.
  
  Similarly, define the {\it localized} groups
  $\bK_\circ^\sT(\fX)_{\loc}$ by replacing all groups
  $K_{\sT}^\circ(-)$ with the localized groups
  $K_{\sT}^\circ(-)_{\loc}$.

  For short, from here on we refer to the groups $\bK_\circ^\sT(-)$
  simply as {\it K-homology}.
\end{definition}

\subsubsection{}

When $\fX = X$ is a scheme with $\sT$-action and proper $\sT$-fixed
locus, the definition and nomenclature for $\bK^\sT_\circ(X)$ should
be compared to the {\it operational K-theory} of $X^\sT$
\cite{Anderson2015}, which is a bivariant theory in the sense of
\cite{Fulton1981} modeled on $K(X^\sT)$. Our notion of K-homology
arises from an analogous bivariant theory modeled on $K^\circ(X^\sT)$,
applied to the map $X^\sT \to *$.

We refer to $\bK_\circ^\sT$ as ``abstract'' K-homology because its
definition imposed only the minimal set of axioms on its elements
$\phi = \{\phi_S\}_S$ so that the vertex algebra construction in
\S\ref{sec:mVOA-on-K-homology} works. So, unlike for operational
K-theory, $\bK_\circ^\sT$ is potentially enormous and not usually
explicitly computable. For applications where an explicit description
of the K-homology group is required, we construct a ``concrete''
K-homology group $K_\circ^\sT$
(Definition~\ref{def:concrete-k-homology}) later.

\subsubsection{}
\label{sec:k-homology-shorthand}

To prevent notational clutter, and for clarity, we generally write
formulas involving elements $\phi$ in a K-homology group
$\bK_\circ^\sT(\fX)$ in terms of just the functional $\phi_S$ for $S =
*$. It will always be clear how to extend the formula to $\phi_S$ for
general $S$. For instance, in \eqref{eq:k-homology-pushforward} below,
the definition $(f_*\phi)(\cE) \coloneqq \phi(f^*\cE)$ means that
\[ f_*\phi \coloneqq \{(f_*\phi)_S\colon K^\circ_\sT(\fY \times S) \to K^\circ_\sT(S)\} \in \bK^\sT_\circ(\fY) \]
is defined by $(f_*\phi)_S(\cE) \coloneqq \phi_S((f \times \id)^*\cE)$
for every base $S$.

\subsubsection{}

Some properties of $K^\circ_\sT(\fX)$ are inherited by
$\bK_\circ^\sT(\fX)$ as follows.
\begin{itemize}
\item A $\sT$-equivariant morphism $f\colon \fX \to \fY$ induces a
  pushforward
  \begin{equation} \label{eq:k-homology-pushforward}
    f_*\colon \bK^\sT_\circ(\fX) \to \bK^\sT_\circ(\fY), \qquad (f_*\phi)(\cE) \coloneqq \phi(f^*\cE).
  \end{equation}
  It is straightforward to check that $f_*\phi$ satisfies all
  the K-homology axioms, using the obvious maps. 

\item Tensor product on $K^\circ_\sT(\fX)$ induces a cap product
  \begin{equation} \label{eq:k-homology-cap-product}
    \cap\colon \bK^\sT_\circ(\fX) \times K^\circ_\sT(\fX) \to \bK^\sT_\circ(\fX), \qquad (\phi \cap \cF)(\cE) \coloneqq \phi(\cE \otimes \cF).
  \end{equation}
  This is well-defined since $I^\circ(\fF_\phi)$ is an ideal. 

\item If $\bK^\sT_\circ(\fX \times \fX) \cong
  \bK^\sT_\circ(\fX)^{\otimes 2}$ (a very restrictive condition, in
  contrast to homology), then the diagonal map $\Delta\colon \fX \to
  \fX \times \fX$ induces a coproduct
  \begin{equation} \label{eq:operational-K-homology-coproduct}
    \Delta_*\colon \bK^\sT_\circ(\fX) \to \bK^\sT_\circ(\fX) \times \bK^\sT_\circ(\fX).
  \end{equation}
\end{itemize}
Hence $K^\sT_\circ(\fX)$ behaves like a K-theoretic version of
homology.

\subsubsection{}

\begin{lemma}[Push-pull]
  For any $\sT$-equivariant $f\colon \fX \to \fY$, and any
  $\phi \in \bK^\sT_\circ(\fX)$, $\cE \in K_\sT^\circ(\fY)$,
  \[ f_*(\phi) \cap \cE = f_*\left[\phi \cap f^*(\cE)\right]. \]
\end{lemma}

\begin{proof}
  Applying both sides to $\cE' \in K^\circ_\sT(\fY)$, this reduces to
  the equality \eqref{eq:k-theory-pullback-tensor-product} in
  $K_\sT^\circ(\fX)$.
\end{proof}

\subsubsection{}

\begin{lemma}
  Let $\phi \in \bK^\sT_\circ(\fX)$ and $\psi \in \bK^\sT_\circ(\fX')$.
  The bottom left composition in
  \begin{equation} \label{eq:operational-k-homology-boxtimes}
    \begin{tikzcd}
      K^\circ_\sT(\fX \times \fX' \times S) \ar{r}{\psi_{\fX \times S}} \ar{d}[swap]{\phi_{\fX' \times S}} & K^\circ_\sT(\fX \times S) \ar{d}{\phi_S} \\
      K^\circ_\sT(\fX' \times S) \ar{r}{\psi_S} & K^\circ_\sT(S)
    \end{tikzcd}
  \end{equation}
  yields a well-defined element
  \begin{equation} \label{eq:k-homology-boxtimes}
    \phi \boxtimes \psi \in \bK^\sT_\circ(\fX \times \fX'). 
  \end{equation}
\end{lemma}

\begin{proof}
  Clearly the composition of $K_\sT^\circ(S)$-linear homomorphisms is
  still a $K_\sT^\circ(S)$-linear homomorphism. We check the
  K-homology axioms.
  \begin{enumerate}
  \item (Naturality) This is clear.
  \item (Equivariant localization) Since $\cat{C}$ is closed under
    fiber products over $\Spec \bC$, take $\fF_{\phi \boxtimes \psi}
    \coloneqq \fF_\phi \times \fF_\psi$ and
    \[ \fix_{\phi \boxtimes \psi} \coloneqq \fix_\phi \times \fix_\psi\colon \fF_{\phi \boxtimes \psi} \to \fX \times \fX', \]
    so that $(\phi \boxtimes \psi)^{\sT} \coloneqq \phi^{\sT}
    \boxtimes \psi^{\sT}$. In particular, $(\phi \boxtimes \psi)^\sT_S
    = \psi_S^\sT \circ \phi_{\fF_\psi \times S}^\sT$.
  \item (Finiteness) Since $\phi_{\fF_\psi \times S}^\sT$ already
    annihilates sufficiently high powers of $I^\circ(\fF_\psi \times
    S)$, so does the composition $(\phi \boxtimes \psi)^\sT_S$.
    \qedhere
  \end{enumerate}
\end{proof}

\subsubsection{}

\begin{definition} \label{def:k-homology-theory}
  View $\bK_\circ^\sT(-)$
  (Definition~\ref{def:operational-k-homology}) as a covariant functor
  from $\cat{Art}_\sT$ into $\bk_\sT$-modules via
  \eqref{eq:k-homology-pushforward}. A {\it $\sT$-equivariant
    operational K-homology theory} $\bfK_\circ^\sT(-) =
  \bfK_\circ^\sT(-; \cat{C})$ is a sub-functor
  \[ \bfK_\circ^\sT(-) \subset \bK_\circ^\sT(-) \]
  which is closed under the cap product \eqref{eq:k-homology-cap-product}
  and the external tensor product \eqref{eq:k-homology-boxtimes}.
  Similarly, use $\bK_\circ^\sT(-)_{\loc}$ to define the notion of
  {\it localized} K-homology theory.

  We say $\bfK_\circ^\sT(-)$ is a {\it commutative} K-homology theory
  (see \cite[\S 2.2]{Fulton1981}) if for any $\fX, \fX' \in
  \cat{Art}_\sT$ and $\phi \in \bfK_\circ^\sT(\fX)$ and $\psi \in
  \bfK_\circ^\sT(\fX')$,
  \[ \phi \boxtimes \psi = \psi \boxtimes \phi. \]
  In other words, the square
  \eqref{eq:operational-k-homology-boxtimes} commutes for any base $S
  \in \cat{Art}_\sT$.
\end{definition}

\subsubsection{}

\begin{definition}[``Concrete'' K-homology] \label{def:concrete-k-homology}
  Given $\fX \in \cat{Art}_\sT$, define $K_\circ^\sT(\fX)_{\loc}$ to
  be the set of all triples $\phi = (Z_\phi, \fix_\phi, \cF_\phi)$
  where:
  \begin{itemize}
  \item $Z_\phi \in \cat{C}$ and is a proper scheme, e.g. $Z_\phi$ is
    a projective scheme (see \S\ref{sec:k-homology-support-category});
  \item $\fix_\phi\colon Z_\phi \to \fX$ is a $\sT$-equivariant
    morphism for the trivial $\sT$-action on $Z_\phi$;
  \item $\cF_\phi \in K_\sT(Z_\phi)_{\loc}$ is a K-theory element.
  \end{itemize}
  Equip it with the group operation where $\fix_{\phi+\psi}$ is the
  obvious map from $Z_{\phi+\psi} \coloneqq Z_\phi \sqcup Z_\psi$ and
  $\cF_{\phi+\psi} \coloneqq \cF_\phi + \cF_\psi$. For the remainder
  of this paper, $K_\circ^\sT(\fX)_{\loc}$ will denote its image
  inside $\bK_\circ^\sT(\fX)_{\loc}$, defined by the following
  Proposition~\ref{prop:actual-k-homology}.
\end{definition}

\subsubsection{}

\begin{proposition} \label{prop:actual-k-homology}
  $K_\circ^\sT(-)_{\loc}$ defines a localized commutative K-homology
  theory (Definition~\ref{def:k-homology-theory}). Furthermore
  $K_\circ^\sT([*/\bC^\times])_{\loc} =
  \bK_\circ^\sT([*/\bC^\times])_{\loc}$.
\end{proposition}

\begin{proof}
  Let $\phi = (Z_\phi, \fix_\phi, \cF_\phi) \in
  K_\circ^\sT(\fX)_{\loc}$. Given a base $S \in \cat{Art}_\sT$, define
  \begin{align*}
    \phi_S\colon K^\circ_\sT(\fX \times S)_{\loc} &\to K^\circ_\sT(S)_{\loc} \\
    \cE &\mapsto (\pi_S)_*(\pi_Z^*\cF_\phi \otimes (\fix_\phi \times \id)^*(\cE))
  \end{align*}
  where $\pi_Z$ and $\pi_S$ are the projections from $Z_\phi \times S$
  to the $Z_\phi$ and $S$ factors respectively. Indeed, this defines
  an element $\phi_S(\cE)$ of $K^\circ_\sT(S)_{\loc}$ because $\cE$ is
  perfect, $\pi_Z^*\cF_\phi$ is flat over $S$, and the pushforward is
  proper. Also, $\phi_S$ is $K_\sT^\circ(S)_{\loc}$-linear by the
  projection formula. We verify the K-homology axioms for
  $\{\phi_S\}_S$.
  \begin{enumerate}
  \item (Naturality) Consider the commutative diagram
    \[ \begin{tikzcd}
      & \fX \times S \ar{rr}{\id \times h} && \fX \times S' \\
      Z_\phi \times S \ar{ur}{\fix_\phi \times \id} \ar{rr}{\id \times h} \ar{dr}[swap]{\pi_S} && Z_\phi \times S' \ar{ur}[swap]{\fix_\phi \times \id} \ar{dr}{\pi_{S'}} \\
      & S \ar{rr}{h} && S'.
    \end{tikzcd} \]
    The lower square is cartesian and also Tor-independent because the
    projection $\pi_{S'}$ is flat. The desired statement follows from
    base change.

  \item (Equivariant localization) Take $\fF_\phi = Z_\phi$ with the
    morphism $\fix_\phi$, so that
    \[ \phi_S^\sT(\cE) \coloneqq (\pi_S)_*(\pi_Z^*\cF_\phi \otimes \cE). \]

  \item (Finiteness) Let $S$ be any scheme, so $Z_\phi \times S$ is
    still a scheme. By
    Lemma~\ref{lem:augmentation-ideal-vanishing-for-schemes},
    \[ \pi_Z^*\cF_\phi \otimes I^\circ(Z_\phi \times S)^{\otimes N} = 0 \in K_\sT(Z_\phi \times S) \]
    for all $N \gg 0$. Hence $\phi_S^\sT(I^\circ(Z_\phi \times
    S)^{\otimes N}) = 0$.

  \item (Commutativity) Consider the commutative diagram
    \[ \begin{tikzcd}
      && Z_\phi \times Z_\psi \times S \ar{dr}{\pi_{Z_\psi \times S}} \ar{dl}[swap]{\id \times \fix_\psi \times \id} \\
      & Z_\phi \times \fX' \times S \ar{dl}[swap]{\fix_\phi \times \id \times \id} \ar{dr}{\pi_{\fX' \times S}} && Z_\psi \times S \ar{dl}[swap]{\fix_\psi \times \id} \ar{dr}{\pi_S} \\
      \fX \times \fX' \times S && \fX' \times S && S
    \end{tikzcd} \]
    where, again, the square is cartesian and Tor-independent. Base
    change shows that
    \begin{equation} \label{eq:k-homology-tensor-product}
      (\phi \boxtimes \psi)_S(\cE) = (\pi_S)_*\left(\pi_{Z_\phi}^*(\cF_\phi) \otimes \pi_{Z_\psi}^*(\cF_\psi) \otimes (\fix_\phi \times \fix_\psi \times \id)^*(\cE)\right)
    \end{equation}
    where $\pi_{Z_\phi}$, $\pi_{Z_\psi}$ and $\pi_S$ are the
    projections from $Z_\phi \times Z_\psi \times S$. Interchanging
    $\phi$ and $\psi$, the right hand side remains unchanged.
  \end{enumerate}
  So indeed $K_\circ^\sT(-)_{\loc} \subset \bK_\circ^\sT(-)_{\loc}$.
  It is a sub-functor by defining $f_*\phi \coloneqq (Z_\phi, f \circ
  \fix_\phi, \cF_\phi)$. It is closed under cap product by defining
  $\phi \cap \cF \coloneqq (Z_\phi, \fix_\phi, \cF_\phi \otimes
  \fix_\phi^*\cF)$. It is closed under external tensor product by
  \eqref{eq:k-homology-tensor-product}, i.e. by defining $\phi
  \boxtimes \psi \coloneqq (Z_\phi \times Z_\psi, \fix_\phi \times
  \fix_\psi, \cF_\phi \boxtimes \cF_\psi)$. Finally, the proof of
  Proposition~\ref{prop:operational-k-homology-BGm} shows that
  $K_\circ^\sT([*/\bC^\times])_{\loc} =
  \bK_\circ^\sT([*/\bC^\times])_{\loc}$.
\end{proof}

\subsubsection{}
\label{sec:universal-invariants-shorthand}

It is clear from the proof of Proposition~\ref{prop:actual-k-homology}
that we view an element $\phi = (Z_\phi, \fix_\phi, \cF_\phi) \in
K_\circ^\sT(\fX)_{\loc}$ as an operator
\[ \phi = \chi\left(Z_\phi, \cF_\phi \otimes \fix_\phi^*(-)\right)\colon K_\sT^\circ(\fX)_{\loc} \to \bk_{\sT,\loc} \]
which behaves well with respect to base change. Hence we adopt this
shorthand notation for elements in $K_\circ^\sT(-)$. For instance,
\eqref{eq:k-homology-tensor-product} says
\begin{equation} \label{eq:universal-invariants-tensor-product}
  \phi \boxtimes \psi = \chi\left(Z_\phi \times Z_\psi, \left(\cF_\phi \boxtimes \cF_\psi\right) \otimes \left(\fix_\phi \times \fix_\psi\right)^*(-)\right).
\end{equation}

\subsection{Example: \texorpdfstring{$[*/G]$}{[*/G]}}
\label{sec:operational-k-homology-BG}

\subsubsection{}

Let $G$ be a linear algebraic group and $\fX \coloneqq [*/G]$ with
trivial $\sT$-action. In this subsection, let $\bK_\circ^\sT(\fX) =
\bK_\circ^\sT(\fX; \cat{C})$ denote the ``abstract'' K-homology group
of Definition~\ref{def:operational-k-homology} where we take $\cat{C}$
to be the category of projective schemes; this is the smallest
possible and therefore the most restrictive choice. The first goal of
this subsection is to more concretely describe $\bK_\circ^\sT(\fX)$.
The second goal is a careful study of the case $G = \bC^\times$, which
will be relevant in \S\ref{sec:mVOA-on-K-homology}.

Note that Artin stacks of the form $\fX = [*/G]$ are special because,
for any $S \in \cat{Art}_\sT$, they satisfy the K\"unneth property
\begin{equation} \label{eq:kunneth}
  K^\circ_\sT(\fX \times S) = K^\circ_\sT(\fX) \otimes_{\bk_\sT} K^\circ_\sT(S).
\end{equation}

\subsubsection{}

\begin{theorem} \label{thm:operational-k-homology-BG}
  Let $I \subset R(G)$ be the augmentation ideal of rank-$0$ virtual
  representations. Then
  \begin{equation} \label{eq:operational-k-homology-BG}
    \bK_\circ^\sT(\fX) \subset \Hom_{\mathrm{cts}}\left(R(G)^\wedge_I, \bZ\right) \otimes_{\bZ} \bk_\sT
  \end{equation}
  where $\Hom_{\mathrm{cts}}$ means continuous $\bZ$-module
  homomorphisms, for the $I$-adic topology on $R(G)$ and the discrete
  topology on $\bZ$.
\end{theorem}

By the definition of $I$-adic topology and continuity, the completion
$(-)^\wedge_I$ in \eqref{eq:operational-k-homology-BG} may be removed
with no effect. However, the completion is the more geometrically
natural object: $R(G)^\wedge_I \cong K_{\cat{Top}}^0(BG)$, the zeroth
topological K-group of the classifying space $BG$ \cite{Atiyah1969}.

The $I$-adic topology on $R(G)$ agrees with the topology induced by
many other filtrations \cite{Karpenko2021}.

\subsubsection{}

\begin{example} \label{ex:k-homology-BGm}
  Let $G = \bC^\times$. Then $I = \inner{1-s} \subset \bZ[s^\pm] =
  R(\bC^\times)$, and
  \[ R(\bC^\times)^\wedge_I = \bZ\pseries{1-s}. \]
  Indeed, $B\bC^\times = BU(1) = \bC\bP^\infty$ and
  $K_{\cat{Top}}^0(B\bC^\times) = \bZ\pseries{1-s}$. Let $\{\xi^k\}_k$
  be the dual basis to $\{(1 - s)^\ell\}_\ell$. Explicitly,
  \[ \xi^k(f) = \frac{1}{k!} \frac{\di^k f}{\di (1-s)^k} \bigg|_{s=0} = \frac{(-1)^k}{k!} \frac{\di^k f}{\di s^k} \bigg|_{s=1}. \]
  It follows that
  \begin{equation} \label{eq:k-homology-BGm}
    \Hom_{\mathrm{cts}}(R(\bC^\times)^\wedge_I, \bZ) = \bZ[\xi], \qquad \xi^k(s^n) = (-1)^k \binom{n}{k}.
  \end{equation}
\end{example}

\subsubsection{}

\begin{proof}[Proof of Theorem~\ref{thm:operational-k-homology-BG}.]
  Consider the naturality axiom. The K\"unneth property
  \eqref{eq:kunneth}, along with Hom-tensor adjunction, implies
  \[ \Hom_{K^\circ_\sT(S)}(K^\circ_\sT(\fX \times S), K^\circ_\sT(S)) = \Hom_{\bk_\sT}(K^\circ_\sT(\fX), K^\circ_\sT(S)). \]
  Given an element $\phi = \{\phi_S\} \in \bK_\circ^\sT(\fX)$, each
  $\phi_S$ is therefore equivalent to a $\bk_\sT$-linear map
  \[ \bar\phi_S\coloneqq \phi_S \circ (\id_\fX \times h)^*\colon K^\circ_\sT(\fX) \to K^\circ_\sT(S) \]
  where $h\colon S \to *$ is the structure morphism. By naturality,
  all $\bar\phi_S$ may therefore be identified with the one for $S =
  *$, which is an element of
  \[ \Hom_{\bk_\sT}(K^\circ_\sT(\fX), \bk_\sT) = \Hom_\bZ(K^\circ(\fX), \bk_\sT) = \Hom_\bZ(K^\circ(\fX), \bZ) \otimes_{\bZ} \bk_\sT. \]
  The first equality follows from the triviality of the $\sT$-action
  on $\fX$, along with Hom-tensor adjunction, and the second equality
  follows from $\bk_\sT$ being a free $\bZ$-module. Hence, without
  loss of generality, we may consider $\bar\phi \in
  \Hom_{\bZ}(K^\circ(\fX), \bZ)$, with $\phi_S = \bar\phi \otimes
  \id_{K^\circ_\sT(S)}$.

  Clearly $K^\circ(\fX) = K^\circ_\sG(*) = K_\sG(*) = R(G)$ is the
  representation ring of $G$, and $I^\circ(\fX) \subset R(G)$ is the
  augmentation ideal $I$.

  Consider the finiteness axiom. Since both $\fX$ and $S \in \cat{C}$
  have trivial $\sT$-action, $I^\circ(\fX \times S)$ is well-defined
  and $(\fix_\phi \times \id)^*I^\circ(\fX \times S) \subset
  I^\circ(\fF_\phi \times S)$. The finiteness axiom therefore implies
  \begin{equation} \label{eq:X-finiteness-condition}
    \phi_S\left(I^\circ(\fX \times S)^{\otimes N}\right) = 0 \qquad \forall N \gg 0.
  \end{equation}
  In particular this holds for $S = *$ and $\bar\phi$. Put the
  $I^\circ(\fX)$-adic topology on $K^\circ(\fX)$ and the discrete
  topology on $\bZ$. We claim \eqref{eq:X-finiteness-condition} is the
  condition that $\bar\phi$ is continuous. Namely, continuity means
  the pre-image of the open set $\{k\} \subset \bZ$ must still be
  open, i.e.
  \[ w \in \bar\phi^{-1}(k) \implies w + I^\circ(\fX)^{\otimes N} \in \bar\phi^{-1}(k) \;\; \forall N \gg 0. \]
  But $\bar\phi$ is linear, so $\bar\phi(w + I^\circ(\fX)^{\otimes N})
  = \bar\phi(w)$ iff \eqref{eq:X-finiteness-condition} holds for $S = *$.
\end{proof}

\subsubsection{}
\label{sec:operational-k-homology-BGm}

\begin{proposition} \label{prop:operational-k-homology-BGm}
  For $G = \bC^\times$, \eqref{eq:operational-k-homology-BG} is an
  equality:
  \[ \bK_\circ^\sT([*/\bC^\times]) = \Hom_{\mathrm{cts}}(R(\bC^\times)^\wedge_I, \bZ) \otimes_{\bZ} \bk_\sT \stackrel{\eqref{eq:k-homology-BGm}}{=} \bk_\sT[\xi]. \]
\end{proposition}

\begin{proof}
  Apply Theorem~\ref{thm:operational-k-homology-BG} to obtain that
  \[ \bK_\circ^\sT([*/\bC^\times]) \subset \Hom_{\mathrm{cts}}(R(\bC^\times)^\wedge_I, \bZ) \otimes_\bZ \bk_\sT. \]
  We must show the reverse inclusion by checking the localization and
  finiteness axioms for
  \[ \xi^k_S \coloneqq \xi^k \otimes \id_{K^\circ_\sT(S)} \colon K([*/\bC^\times]) \otimes_\bZ K^\circ_\sT(S) \to K^\circ_\sT(S) \]
  for all $k \ge 0$ and $S \in \cat{C}$. Write $K([*/\bC^\times]) =
  \bZ[s^\pm]$ and recall that $\xi^k(s^n) = (-1)^k \binom{n}{k}$.
  Define
  \[ \fix_{\xi^k}\colon \bP_{\bC}^k = (\bC^{k+1} \setminus \{0\}) / \bC^\times \hookrightarrow [\bC^{k+1}/\bC^\times] \to [*/\bC^\times] \]
  to be the natural open embedding followed by projection, set
  $\fF_{\xi^k} \coloneqq \bP_{\bC}^k$, and let
  \[ (\xi^k_S)^\sT \coloneqq (-1)^k \cdot (\pi_S)_*(\cO(-k) \otimes -)\colon K_\sT^\circ(\fF_{\xi^k} \times S) \to K_\sT^\circ(S). \]
  Since $\fix_k^* s^n = \cO_{\bP_{\bC}^k}(n)$ and $\chi(\bP^k, \cO(n))
  = \binom{n+k}{k}$, indeed
  \[ \xi^k_S = (\xi^k_S)^\sT \circ (\fix_{\xi^k} \times \id)^* \]
  as desired. Finally, the proof of
  Proposition~\ref{prop:actual-k-homology} shows that
  $\{(\xi^k_S)^{\sT}\}_S$ is indeed a collection of homomorphisms of
  $K_\sT^\circ(S)$-modules which satisfy the naturality and finiteness
  axioms.
\end{proof}

\subsubsection{}

\begin{remark}
  The appearance of $\bP_{\bC}^k$ in the proof of
  Proposition~\ref{prop:operational-k-homology-BGm} is an instance of
  the following more general phenomenon. For any connected solvable
  $G$, there exists \cite[Theorems 2.1, 2.2]{Edidin2000} an inverse
  system of pairs $(U, V)$ such that $V$ is a vector space, $U \subset
  V$ has high(er and higher) codimension, $G$ acts freely on $U
  \subset V$, and
  \[ R(G)^\wedge_I = \varprojlim_{(U,V)} K(U/G) \]
  Hence every continuous homomorphism $R(G)^\wedge_I \to \bZ$ factors
  through $K(U/G)$ for some $U$ such that $U/G$ is a scheme. For $G =
  \bC^\times$, we used the pairs $(\bC^k \setminus \{0\}, \bC^k)$.

  It is possible that such ideas may be applied to generalize the
  argument of Proposition~\ref{prop:operational-k-homology-BGm}, to
  show that \eqref{eq:operational-k-homology-BG} is an equality for
  more general $G$.
\end{remark}

\subsubsection{}

\begin{remark}
  In contrast to $K^\circ([X/G]) = K^\circ_G(X)$, it is easy to check,
  using Proposition~\ref{prop:operational-k-homology-BGm} and the
  proof of Theorem~\ref{thm:operational-k-homology-BG}, that
  \[ \bZ[\xi] = \bK_\circ([*/\bC^\times]) \neq \bK_\circ^{\bC^\times}(*) = \bk_{\bC^\times}. \]
  This is another (lack of a) property shared by equivariant
  K-homology and reasonable definitions of equivariant homology of
  Artin stacks; see e.g. \cite[\S 2.2]{Joyce2021}, where $H_*^G(X)
  \neq H_*([X/G])$ in general.
\end{remark}

\subsubsection{}

The multiplication map $\Omega\colon [*/G] \times [*/G] \to [*/G]$ makes
$[*/G]$ (resp. $BG$) into a group object in Artin stacks (resp.
H-spaces). Thus $K^\circ([*/G]) = R(G)$ is a Hopf algebra and
therefore its continuous dual is too. For completeness, when $G =
\bC^\times$, we explicitly compute its product/coproduct, continuing
with the notation of Example~\ref{ex:k-homology-BGm}.
\begin{itemize}
\item Recall that $\xi^k = \frac{(-1)^k}{k!} \di_s^k\big|_{s=0}$. The
  Leibniz rule $\xi^k(fg) = \sum_{i+j=k} \xi^i(f) \xi^j(g)$ induces
  the coproduct (cf. \eqref{eq:operational-K-homology-coproduct})
  \[ \Delta(\xi^k) = \sum_{i+j=k} \xi^i \otimes \xi^j. \]

\item Let $\deg_s\colon R(\bC^\times) = \bZ[s^\pm] \to \bZ$ be the
  degree-in-$s$ homomorphism. Viewing $\xi^k = (-1)^k
  \binom{\deg_s}{k}$ as polynomials in $\deg_s$, the coproduct
  $\Omega^*\colon s \mapsto s \boxtimes s$ on $R(\bC^\times)$ induces
  the ordinary product of polynomials in $\deg_s$. In other words,
  \begin{equation} \label{eq:operational-k-homology-BGm-as-numerical-polynomial}
    \bZ[\xi] \subset \bQ[\deg_s]
  \end{equation}
  is the subalgebra of numerical polynomials in $\deg_s$. We denote
  this product on $\bZ[\xi]$ by $\star$. It is distinct from the
  standard product, and a nice combinatorial exercise is that
  \[ \xi^a \star \xi^b = \sum_{k=0}^{a+b} (-1)^k \binom{a+b-k}{a} \binom{a}{k} \xi^{a+b-k}. \]
\end{itemize}

\subsubsection{}

\begin{proposition}
  Let $K^\circ(\bC\bP^\infty) \coloneqq \varprojlim_n K^\circ(\bP_{\bC}^n)$ and
  \[ \ch\colon K^\circ(\bC\bP^\infty) = \bZ[[1-s]] \xrightarrow{s \mapsto e^x} \bQ[[x]] = H^*(\bC\bP^\infty)_\bQ \]
  be the Chern character induced by the usual Chern characters
  $\ch\colon K(\bP_{\bC}^n) \to H^*(\bC\bP^n)_\bQ$. Let $\zeta \in
  H_2(\bC\bP^n)$ be the class dual to $x$. Then, passing to duals,
  $\ch$ induces an inclusion
  \[ \ch_*\colon \bK_\circ([*/\bC^\times]) \subset H_*(\bC\bP^\infty)_\bQ = \bQ[\zeta] \]
  given by $\xi^k \mapsto (-1)^k \binom{\zeta}{k}$.
\end{proposition}

This gives an interpretation of
\eqref{eq:operational-k-homology-BGm-as-numerical-polynomial} as a
homology Chern character.

Note that $\ch \otimes \bQ$ is no longer an
isomorphism, as $\bQ \otimes_{\bZ} \bZ[[x]] \neq \bQ[[x]]$.

\begin{proof}
  We claim that the correct product on $\zeta$ is given by
  \[ \zeta^m(x^n) = \delta_{mn} n!. \]
  This follows from the computation
  \[ \Omega_*(\zeta^a \boxtimes \zeta^b)(x^k) = (\zeta^a \boxtimes \zeta^b)(\Omega^*(x)^k) = \sum_{i+j=k} \binom{k}{i} \zeta^a(x^i) \zeta^b(x^j) = \delta_{k,a+b} k! = \zeta^{a+b}(x^k). \]
  It remains to check that $\xi^k$ and $(-1)^k \binom{\zeta}{k}$ both
  act the same way on $s^n \mapsto e^{nx}$ for $n \ge 0$. Since
  $\zeta^k(e^{nx}) = n^k$, it follows that
  \[ (-1)^k \binom{\zeta}{k}(e^{nx}) = (-1)^k \binom{n}{k} = \xi^k(s^n). \qedhere \]
\end{proof}

\subsubsection{}

\begin{remark}
  The Hopf algebra $(\bK_\circ([*/\bC^\times]), \star, \Delta)$ is
  exactly the $\lambda$-divided power Hopf algebra of
  \cite{Andrews2003} for $\lambda=-1$ after a mild change of basis.
  Our construction ascribes geometric meaning to their purely
  algebraic definition and provides an alternate proof of their main
  theorem. In general $\lambda$ is the constant in the $1$-dimensional
  formal group law $F_\lambda(x, y) \coloneqq x + y + \lambda xy$.
  There is a well-understood correspondence between formal group laws,
  their associated generalized cohomology theories $E(-)$, and Hopf
  algebra structures on an appropriate dual of $E(\bC\bP^\infty)$
  \cite[Part II]{Adams1974}.

  Degenerating to $\lambda=0$ produces $H_*(\bC\bP^\infty)$, which is
  an ordinary divided power Hopf algebra generated by the divided
  powers $\zeta^{[k]} \coloneqq \zeta^k/k!$, e.g.
  \[ \zeta^{[a]} \zeta^{[b]} = \binom{a+b}{a} \zeta^{[a+b]}. \]
  The map $\ch_*$ is a Hopf algebra morphism, and becomes an
  isomorphism only after base change to $\bQ$. Indeed, as $F_{-1}$ and
  $F_0$ are non-isomorphic group laws, the $\lambda=-1$ and
  $\lambda=0$ divided power Hopf algebras are not isomorphic over
  $\bZ$.
\end{remark}

\section{Equivariant multiplicative vertex algebras}
\label{sec:equivariant-mVOA}

\subsection{Series and residues}
\label{sec:series-and-residues}

\subsubsection{}
\label{sec:equivariant-series-rings}

Let $\sT$ be a split algebraic torus, with character lattice denoted
$\Char(\sT)$ and representation ring $\bk_\sT = \bZ[t^\mu : \mu \in
  \Char(\sT)]$. For formal variables $z_1, \ldots, z_m$ and a
$\bk_\sT$-module $M$, define
\begin{equation} \label{eq:equivariant-polynomial-ring}
  M\left[(1 - z_1 \cdots z_m)^{-1}\right]_\sT \coloneqq M\left[(1 - t^\mu z_1^{i_1} \cdots z_m^{i_m})^{-1} : \begin{array}{c} \mu \in \Char(\sT) \\ i_1,\ldots,i_m \in \bZ \setminus \{0\}\end{array}\right]
\end{equation}
Similarly define $M[(1-z_1\cdots z_m)^\pm]_\sT$. Let
\[ M\left[(1-z_1 \cdots z_m)^{-1}\right]_\sT^{(1)} \subset M\left[(1-z_1 \cdots z_m)^{-1}\right]_\sT \]
denote the submodule given by \eqref{eq:equivariant-polynomial-ring}
but with $\bZ \setminus \{0\}$ replaced by $\{1, -1\}$. Similarly
define $M[(1-z_1\cdots z_m)^\pm]_\sT^{(1)}$. In the case of a single
variable $z$, let
\[ M\lseries*{1 - z}_{\sT} \coloneqq M\pseries*{1-z}\left[(1 - z)^{-1}\right]_\sT. \]
and similarly define $M\lseries{1-z}_{\sT}^{(1)}$.

If $M$ is a ring, then all these modules are also rings under the
usual multiplication of polynomials and power series.

\subsubsection{}

In what follows, we give some elementary constructions on ordinary
polynomial/series rings over $\bZ = \bk_{\{1\}}$, with the
understanding that they apply equally well to the $\bk_\sT$-modules
$M$ of \S\ref{sec:equivariant-series-rings}, for non-trivial $\sT$,
after an appropriate base change. For instance, we have the obvious
embeddings
\[ \bZ\pseries*{1-z} \hookrightarrow \bZ\lseries*{1-z} \hookrightarrow \bZ\lseries*{1-z}_{\{1\}}, \]
and furthermore we often implicitly use the embedding
\[ \iota_z\colon \bZ\left[z^\pm\right] \hookrightarrow \bZ\pseries*{1-z}, \qquad z^k \mapsto (1 - (1-z))^k. \]
Note that $1 - (1-z) \in \bZ\pseries{1-z}$ is invertible, so this is
well-defined for all $k \in \bZ$.

\subsubsection{}

\begin{lemma}
  For non-zero $i, j \in \bZ$, the unique ring homomorphism
  $\Lambda_{i,j}$ such that
  \[ \begin{tikzcd}
    & \bZ\left[z^{\pm ij}\right] \ar[hookrightarrow]{dr}{\iota_{z^i}} \ar[hookrightarrow]{dl}[swap]{\iota_{z^{ij}}} \\
    \bZ\pseries*{1-z^{ij}} \ar{rr}{\Lambda_{i,j}} && \bZ\pseries*{1-z^i}
  \end{tikzcd} \]
  commutes is injective and is an isomorphism if and only if $j \in
  \{1, -1\}$.
\end{lemma}

We often implicitly apply $\Lambda_{i,j}$, particularly the
isomorphism $\Lambda_{-1,-1}$, in
\S\ref{sec:equivariant-mVOA-general-theory}.

\begin{proof}
  Commutativity implies
  \[ \Lambda_{i,j}\colon 1 - z^{ij} \mapsto 1 - (1 - (1 - z^i))^j = (1 - z^i)(j + \cdots) \]
  where $\cdots$ denotes terms of positive degree in $1 - z^i$. The
  image is clearly non-zero. As for invertibility, $(j + \cdots)$ is a
  unit if and only if $j \in \bZ$ is invertible.
\end{proof}

\subsubsection{}

\begin{definition}
  For two formal variables $z$ and $w$, the {\it (multiplicative)
    expansion}
  \[ \iota_{z,w}\colon \bZ\left[(1 - zw)^\pm\right] \hookrightarrow \bZ\left[(1-w)^\pm\right]\pseries*{1 - z} \]
  is the injective ring homomorphism given by
  \begin{align*}
    \iota_{z,w}(1 - zw)^k
    &\coloneqq \left((1-z) + (1-w) - (1-z)(1-w)\right)^k \\
    &= (1-w)^k \left(1 - (1-z)(1 - (1-w)^{-1})\right)^k.
  \end{align*}
  Note that $1 - (1-z)(1 - (1-w)^{-1})$ is invertible, so this is
  well-defined for all $k \in \bZ$. The name is because, analytically,
  $\iota_{z,w}$ can be viewed as series expansion in the domain $|1-z|
  < |1-w^{-1}|$.

  This is the multiplicative analogue of the additive expansion
  $\bZ[(u - v)^{-1}] \hookrightarrow \bZ[v^{-1}]\pseries{u}$ given by
  series expansion in the domain $|u| < |v|$, i.e. in non-negative
  powers of $u/v$.
\end{definition}

\subsubsection{}

\begin{definition} \label{def:equivariant-expansion}
  The {\it (multiplicative) equivariant expansion}
  \[ \iota_{z,w}^\sT\colon \bk_\sT\left[(1 - zw)^\pm\right]_\sT \to \bk_\sT\left[(1-w)^\pm\right]_\sT\pseries*{1-z}. \]
  is the $\bk_\sT$-algebra homomorphism given by applying the
  expansion
  \[ \iota_{z^i,t^\mu w^j}\colon \bk_\sT\left[(1 - t^\mu z^i w^j)^\pm\right] \to \bk_\sT\left[(1 - t^\mu w^j)^\pm\right]\pseries*{1 - z^i} \]
  to the monomial $(1 - t^\mu z^i w^j)^k$, and using the embedding
  $\Lambda_{1,i}\colon \bZ\pseries{1-z^i} \hookrightarrow \bZ\pseries{1-z}$.
\end{definition}

\subsubsection{}

\begin{definition} \label{def:k-theoretic-residue-map}
  Given a rational function $f \in \bk_\sT[(1-z)^\pm]_\sT$, let $f_+ \in
  \bk_\sT\lseries{z}$ and $f_- \in \bk_\sT\lseries{z^{-1}}$ be its formal
  series expansion around $z=0$ and $z=\infty$ respectively. The {\it
    (equivariant) K-theoretic residue map} is the $\bk_\sT$-module
  homomorphism
  \begin{align*}
    \rho_K\colon \bk_\sT[(1-z)^\pm]_\sT &\to \bk_\sT \\
    f &\mapsto z^0 \text{ term in } (f_- - f_+).
  \end{align*}
  See Appendix~\ref{sec:residue-maps} for a discussion of residue maps
  in general and a characterization of this one. By construction,
  $\rho_K\colon \bk_\sT[z^\pm] \mapsto 0$, and it has the {\it
    $\sT$-invariance} property that
  \begin{equation} \label{eq:residue-map-T-invariance}
    \rho_K\left(f\big|_{z \mapsto tz}\right) = \rho_K(f),
  \end{equation}
  for any $f$ and $\sT$-weight $t$, because such a substitution does
  not affect the $z^0$ term.
  
  When multiple variables are present, write $\rho_{K,z}$ to
  mean $\rho_K$ applied in the variable $z$. 
\end{definition}

\subsubsection{}

\begin{lemma} \label{lem:residue-map-on-series}
  $\rho_K$ extends to a $\bk_\sT$-module homomorphism
  \[ \rho_K\colon \bk_\sT\lseries*{1-z}_\sT^{(1)} \to \bk_\sT^\wedge \]
  where $\bk_\sT^\wedge$ is the completion of $\bk_\sT$ at the
  augmentation ideal $\inner{1 - t^\mu : \mu \in \Char(\sT)}$.
\end{lemma}

\begin{proof}
  It suffices to show, by direct calculation, that if $N \ge n$ then
  \begin{equation} \label{eq:residue-map-on-filtration}
    \rho_K\frac{(1 - z)^N}{\prod_{i=1}^n (1 - z^{k_i} t^{\mu_i})} \in I^{N-n+1}
  \end{equation}
  where $k_i \in \{1, -1\}$ and $\mu_i \in \Char(\sT)$. Since all
  $t^{\mu_i}$ are invertible, it is equivalent to show that $\rho_K f_a
  \in I^{N-n+1}$, where $a = \#\{i : k_i = 1\}$ and
  \begin{equation} \label{eq:residue-map-on-filtration-1}
    f_a \coloneqq \frac{z^{-a} (1 - z)^N}{\prod_{i=1}^n (1 - z^{-1} t^{\mu_i})}.
  \end{equation}
  Multiplying the top and bottom by $z^n$, clearly the $z^0$ term of
  $(f_a)_+$ vanishes unless $a = n$, in which case it is $(-1)^n
  \prod_{i=1}^n t^{-\mu_i}$. On the other hand,
  \begin{align*}
    (f_a)_-
    &= (-1)^N z^{N-a} (1 - z^{-1})^N \sum_{k \ge 0} (-1)^k z^{-k} (1 - z^{-1})^{-k-n} h_k(1 - t^\mu) \\
    &= (-1)^N z^{N-a} \sum_{k \ge 0} \sum_{l \ge 0} (-1)^{k+l} \binom{N-k-n}{l} z^{-k-l} h_k(1 - t^\mu)
  \end{align*}
  where $h_k(1 - t^\mu) \in I^k$ is the $k$-th complete homogeneous
  polynomial in $\{1 - t^{\mu_i}\}_{i=1}^n$. Hence
  \[ z^0 \text{ term in } (f_a)_- = (-1)^a \sum_{k=0}^{N-a} \binom{N-n-k}{N-a-k} h_k(1 - t^\mu). \]
  \begin{itemize}
  \item If $0 \le a < n$, then this is $0 \bmod{I^{N-n+1}}$ because
    $\binom{N-n-k}{N-a-k} = 0$ for $k \le N-n$.
  \item If $a = n$, then this cancels with the $z^0$ term of $(f_a)_+$
    because
    \[ (-1)^n \sum_{k=0}^{N-n} h_k(1 - t^\mu) \equiv (-1)^n \sum_{k=0}^\infty h_k(1 - t^\mu) = (-1)^n \prod_{i=1}^n t^{-\mu_i} \bmod{I^{N-n+1}}. \qedhere \]
  \end{itemize}
\end{proof}

\subsubsection{}

\begin{remark}
  We chose to give an elementary but unenlightening proof of
  Lemma~\ref{lem:residue-map-on-series}. Here is a sketch of an
  alternate proof. By Remark~\ref{rem:residue-map-K-theory-comparison}
  and Cauchy's residue theorem, $\rho_K$ can be viewed as a contour
  integral. Then, for functions $f_a$ as in
  \eqref{eq:residue-map-on-filtration-1},
  \[ \rho_K(f_a) = \chi\left(\bP^{n-1}, \cO(-a) \otimes (1 - \cO(1))^N\right)\Big|_{a_i \mapsto t^{\mu_i}} \]
  by Jeffrey--Kirwan integration \cite[Theorem 8.1]{Jeffrey1995}
  \cite[Appendix A]{Aganagic2018}, where $\chi$ denotes equivariant
  Euler characteristic with respect to the maximal torus $\sA
  \coloneqq \{\diag(a_1, \ldots, a_n)\} \subset\GL(n)$. We claim that
  $\chi(\bP^{n-1}, -)$ is $I^\circ_\sA(-)$-adically continuous, i.e.
  \[ \chi(\bP^{n-1}, I^\circ_\sA(\bP^{n-1})^{\otimes N}) \subset I_\sA^\circ(*)^{\otimes M(N)} \]
  for some function $M$ such that $\lim_{N \to \infty} M(N) = \infty$,
  which is all that we need. A high-powered way to prove this is to
  note that the $I_\sA^\circ(\bP^{n-1})$- and $I_\sA^\circ$-adic
  topologies on $\bP^{n-1}$ coincide \cite[Theorem
    6.1(a)]{Edidin2000}, and since $\chi(\bP^{n-1}, -)$ is a
  $\bk_\sA$-module homomorphism, it is clearly continuous for the
  $I^\circ_\sA$-adic topology.
\end{remark}

\subsubsection{}

\begin{remark}
  Note that $\rho_K$ does {\it not} further extend to a homomorphism
  $\bk_\sT\lseries{1-z}_\sT \to \bk_\sT^\wedge$, because
  \eqref{eq:residue-map-on-filtration} fails without the restriction
  that all $k_i \in \{1, -1\}$. The calculation
  \[ \rho_K\frac{(1 - z)^n}{1 - z^2} = \rho_K\frac{(1-z)^{n-1}}{1+z} = -2^{n-1} \notin I \]
  is a good example of the basic underlying issue, which only occurs
  when there are poles in $z$ at non-trivial roots of unity.
\end{remark}

\subsubsection{}

\begin{lemma} \label{lem:residue-map-of-diagonal-expansion}
  Let $f \in \bk_\sT[(1-z)^\pm]_{\sT}$. Then
  \begin{align*}
    \rho_{K,z}\left(\iota_{z,w}^\sT f(z/w)\right) &= 0 \\
    \rho_{K,z}\left(\iota_{w,z}^\sT f(z/w)\right) &= \rho_{K,z} f(z).
  \end{align*}
\end{lemma}

\begin{proof}
  The first equality is clear since $\iota_{z,w}^{\sT} f(z/w)$ as a
  function of $z$ involves only elements of
  $\bk_\sT\pseries{1-z} \subset \ker\rho_{K,z}$. We focus
  on the more interesting second equality.

  The rational function $f$ has finitely many poles $\{z^{k_i}
  t^{\mu_i} = 1\}_i$ where $k_i \in \bZ \setminus \{0\}$ and $\mu_i
  \in \Char(\sT)$. Pass to an $N$-fold cover of $\sT$, i.e. to
  \[ \bk_\sT^{1/N} \coloneqq \bZ[1^{1/N}][t^{\mu/N} : \mu \in \Char(\sT)] \supset \bk_\sT, \]
  with $N \gg 0$ such that all $t^{\mu_i/k_i}$ exist in
  $\bk_\sT^{1/N}$. Let $K_\sT \coloneqq \Frac(\bk_\sT^{1/N})$ be its
  fraction field. Clearly we may compute $\rho_{K,z}$ and all
  equivalent expansions over $K_\sT$ instead of $\bk_\sT$. Since
  $K_\sT[z]$ is a Euclidean domain, there is a partial fraction
  decomposition
  \[ f(z) = f_{\text{reg}}(z) + \sum_{i=1}^m \sum_{j=1}^{e_i} \frac{s_{ij}}{(1 - s_i z)^j} \]
  for some integers $m$ and $e_1, \ldots, e_m$, elements $s_i, s_{ij}
  \in K_\sT$, and a polynomial $f_{\text{reg}} \in K_\sT[z^\pm]$. It
  therefore suffices to prove the lemma for $f(z) = (1 - sz)^{-n}$ for
  any integer $n > 0$ and $s \in K_\sT$. We do this by explicit
  calculation. It is easy to check that
  \[ \rho_{K,z} (1 - sz)^{-n} = -1, \qquad \forall n > 0. \]
  On the other hand, using the binomial theorem,
  \begin{align*}
    \rho_{K,z} \left(\iota_{w,z}^\sT (1 - sz/w)^{-n}\right)
    &= \rho_{K,z} \sum_{k \ge 0} \binom{-n}{k} (1 - w^{-1})^k (w^{-1})^{-n-k} (1 - sz)^{-n-k} \\
    &= -\sum_{k \ge 0} \binom{-n}{k} (1 - w^{-1})^k (w^{-1})^{-n-k} \\
    &= -((1 - w^{-1}) + w^{-1})^{-n} = -1. \qedhere
  \end{align*}
\end{proof}

\subsection{General theory}
\label{sec:equivariant-mVOA-general-theory}

\subsubsection{}

\begin{definition}
  A {\it $\sT$-equivariant multiplicative vertex algebra} is the data
  of:
  \begin{enumerate}
  \item a $\bk_\sT$-module $V$ of {\it states} with a distinguished
    vacuum vector $\vac \in V$;
  \item a multiplicative {\it translation operator} $D(z)
    \in \End(V)\pseries*{1 - z}$, i.e. $D(z)D(w) = D(zw)$;
  \item a {\it vertex product} $Y(-, z)\colon V \otimes V \to
    V\lseries{1-z}_\sT$.
  \end{enumerate}
  This data must satisfy the following axioms:
  \begin{enumerate}
  \item (vacuum) $Y(\vac, z) = \id$ and $Y(a, z)\vac \in V\pseries{1 -
    z}$ with $Y(a,1)\vac = a$;
  \item (skew symmetry) $Y(a, z)b = D(z) Y(b, z^{-1}) a$;
  \item (weak associativity) $Y(Y(a,z) b, w) \equiv Y(a, zw) Y(b, w)$,
    where $\equiv$ means that when applied to any $c \in V$, both
    sides are equivariant expansions of the same element in
    \begin{equation} \label{eq:mvoa-ope-ring}
      V\pseries*{1 - z, 1 - w}\left[(1 - z)^{-1}, (1 - w)^{-1}, (1 - zw)^{-1}\right]_{\sT}.
    \end{equation}
  \end{enumerate}
\end{definition}

For clarity, we refer to the original notion of vertex algebra, see
e.g. \cite[Definition 1.3.1]{Frenkel2004}, as an {\it ordinary vertex
  algebra}. In many aspects, it is a non-equivariant, additive version
of our $\sT$-equivariant multiplicative vertex algebras.

\subsubsection{}

To be precise regarding weak associativity, first observe that, for
$a,b,c \in V$,
\begin{align*}
  Y(Y(a,z)b, w)c &\in V\lseries*{1-w}_\sT\lseries*{1-z}_\sT \\
  Y(a,zw)Y(b,w)c &\in V\lseries*{1-zw}_\sT\lseries*{1-w}_\sT
\end{align*}
so they are not immediately comparable. Weak associativity means to
compare them using the equivariant expansions
\begin{align*}
  \iota_{z,w}^\sT\colon &V\pseries*{1-z, 1-w}\left[(1-z)^{-1}, (1-w)^{-1}, (1-zw)^{-1}\right]_\sT \\
  &\hookrightarrow V\pseries*{1-w}[(1-w)^\pm]_\sT\pseries*{1-z}[(1-z)^{-1}]_\sT = V\lseries*{1-w}_\sT\lseries*{1-z}_\sT, \\
  \iota_{w^{-1},zw}^\sT\colon &V\pseries*{1-zw, 1-w}\left[(1-z)^{-1}, (1-w)^{-1}, (1-zw)^{-1}\right]_\sT \\
  &\hookrightarrow V\pseries*{1-zw}[(1-zw)^\pm]_\sT\pseries*{1-w^{-1}}[(1-w)^{-1}]_\sT = V\lseries*{1-zw}_\sT\lseries*{1-w}_\sT,
\end{align*}
where we identify $\bZ\pseries{1-zw, 1-w} \cong \bZ\pseries{1-z, 1-w}$
via the invertible map $1-zw \mapsto (1-z) + (1-w) + (1-z)(1-w)$.

This is completely analogous to what happens for ordinary vertex
algebras, where the relevant expansions are the ring embeddings
\[ \bZ\lseries{u}\lseries{v} \hookleftarrow \bZ\left[(u-v)^{-1}\right] \hookrightarrow \bZ\lseries{v}\lseries{u}. \]

\subsubsection{}

\begin{definition}
  Let $(V, \vac, D, Y)$ be an equivariant multiplicative vertex algebra.
  \begin{enumerate}
  \item It is {\it holomorphic} if $Y(-, z)$ takes values in the
    sub-module $V\pseries{1-z} \subset V\lseries{1-z}_\sT$.
    Consequently the $\equiv$ in weak associativity becomes an
    equality of series.
  \item It is {\it reduced} if $Y(-, z)$ takes values in the
    sub-module $V\lseries{1-z}_\sT^{(1)} \subset V\lseries{1-z}_\sT$,
    and the $\equiv$ in weak associativity means both sides are
    equivariant expansions of the same element in
    \[ V\pseries*{1 - z, 1 - w}\left[(1 - z)^{-1}, (1 - w)^{-1}, (1 - zw)^{-1}\right]_{\sT}^{(1)}. \]
  \end{enumerate}
\end{definition}

\subsubsection{}

\begin{remark}
  Following \cite{Li2011}, associated to any formal $1$-dimensional
  group law $F = F(u, v)$ is the notion of a {\it vertex $F$-algebra}.
  When $F_a(u, v) \coloneqq u + v$ is the additive formal group law,
  it is exactly an ordinary vertex algebra. When $F_m(u, v) \coloneqq
  u + v + uv$ is the multiplicative formal group law, it is the
  special case of our multiplicative vertex algebras where:
  \begin{enumerate}
  \item $\sT = \{1\}$ is trivial, i.e. there is no equivariance;
  \item poles occur only at $z = 1$ (corresponding to $u = 0$), i.e.
    the vertex algebra is reduced.
  \end{enumerate}
  The latter restriction is unnatural from a geometric perspective,
  where $z$ should be viewed as the K-theoretic weight of some
  (hidden) $\bC^\times$-action, and poles in $z$ arise from
  $\bC^\times$-equivariant localization. Such poles generally exist at
  non-trivial roots of unity as well.

  All formal $1$-dimensional group laws are equivalent over $\bQ$, so
  the notion of vertex $F$-algebras is independent of $F$ after an
  appropriate base change \cite[Theorem 3.7]{Li2011}. For the additive
  and multiplicative group laws, this is roughly equivalent to the
  existence of Riemann--Roch isomorphisms between $K(X)$ and $A_*(X)$
  over $\bQ$. In contrast, when there is non-trivial
  $\sT$-equivariance, typically $K_\sT(X)$ is no longer isomorphic to
  $A_*^\sT(X)$ over $\bQ$ \cite{Edidin2000}.
\end{remark}

\subsubsection{}

\begin{lemma}
  For $a, b \in V$:
  \begin{enumerate}
  \item\label{it:mVOA-translation} (translation) $D(w) Y(a,z) \equiv
    Y(a, zw) D(w)$;
  \item\label{it:mVOA-locality} (locality) $Y(a,z)Y(b,w) \equiv
    Y(b,w)Y(a,z)$.
  \end{enumerate}
\end{lemma}

Later, in Lemmas~\ref{lem:mVOA-translation-expansion} and
\ref{lem:mVOA-OPE}, we refine this lemma by being more specific about
the meaning of $\equiv$, by analyzing the poles on both sides.

\begin{proof}
  This is purely formal. The skew symmetry axiom with $b = \vac$, and
  the vacuum axiom, gives $Y(a,z)\vac = D(z)a$. Then weak
  associativity applied to $\vac$ gives
  \[ Y(a,zw) D(w) b = Y(a,zw) Y(b,w)\vac \equiv Y(Y(a,z)b, w)\vac = D(w) Y(a,z)b. \]
  Similarly, weak associativity with $b = \vac$ gives $Y(D(z)a, w)
  \equiv Y(a, zw)$, also called {\it translation covariance}. Using it
  and weak associativity and skew symmetry,
  \begin{align*}
    Y(a,z) Y(b,w) \equiv Y(Y(a, z/w)b, w)
    &= Y(D(z/w) Y(b, w/z)a, w) \\
    &\equiv Y(Y(b, w/z)a, z) \equiv Y(b,w) Y(a,z). \qedhere
  \end{align*}
\end{proof}


\subsubsection{}

\begin{corollary}[Cousin property, cf. {\cite[Corollary 3.2.3]{Frenkel2004}}] \label{cor:mVOA-cousin}
  For $a, b, c \in V$, the three series
  \begin{align*}
    Y(a,z)Y(b,w)c &\in V\lseries*{1-z}_{\sT}\lseries*{1-w}_{\sT}, \\
    Y(b,w)Y(a,z)c &\in V\lseries*{1-w}_{\sT}\lseries*{1-z}_{\sT}, \\
    Y(Y(a, z/w)b, w)c &\in V\lseries*{1-w}_{\sT}\lseries*{1-z/w}_{\sT},
  \end{align*}
  are expansions of the same element, using $\iota_{w,z}^\sT$,
  $\iota_{z,w}^\sT$, and $\iota_{z/w,w}^\sT$ respectively.
\end{corollary}

\begin{proof}
  Follows from locality (Lemma~\ref{it:mVOA-locality}) and the weak
  associativity axiom.
\end{proof}

\subsubsection{}

\begin{remark}
  For ordinary vertex algebras, it appears standard to impose the
  stronger axiom that the translation operator is $D(v) = \exp(vT)$
  for some derivation $T$ such that $[T, Y(a,u)] = \di_u Y(a,u)$.
  Exponentiating both sides yields
  \[ D(v) Y(a,u) = \exp(v \di_u) Y(a, u) D(v) = Y(a, u+v) D(v), \]
  which is the additive version of our translation axiom. We choose
  not to impose the analogous stronger axiom that $D(w) = w^H$ for
  some grading operator $H$ such that $[H, Y(a,z)] = z\di_z Y(a, z)$.
  (One can check that this is indeed the multiplicative analogue,
  since if $z = \exp u$ and $w = \exp v$ then $\exp(v \di_u) = w^{z
    \di_z}$, and clearly $w^{z \di_z} f(z) = f(zw)$ is a
  multiplicative translation.)
  
  On the other hand, locality is equivalent to the statement that,
  given $a,b,c \in V$, there exists a finite set $\{t_i\}_i$ of
  $\sT$-weights (which, to be clear, may depend on $a, b, c$) such
  that
  \begin{equation} \label{eq:mVOA-locality-as-equality}
    \prod_i (1 - t_i z/w) [Y(a,z),Y(b,w)]c = 0
  \end{equation}
  where $[-,-]$ is the commutator, and this is completely analogous to
  the case \cite[\S 1.2.4]{Frenkel2004} of ordinary vertex algebras.
\end{remark}

\subsubsection{}

\begin{lemma}[Translation] \label{lem:mVOA-translation-expansion}
  For $a, b \in V$,
  \[ \iota_{w^{-1}, zw}^\sT D(w) Y(a, z) b = Y(a, zw) D(w) b. \]
\end{lemma}

\begin{proof}
  This is a refinement of Lemma~\ref{it:mVOA-translation}, which says
  that
  \begin{align}
    \iota_{w^{-1}, zw}^\sT F_{a,b} &= Y(a, zw) D(w) b \in V\lseries*{1-zw}_\sT \pseries*{1-w} \label{eq:mVOA-translation-explicit-1} \\
    \iota_{z, w}^\sT F_{a,b} &= D(w) Y(a, z) b \in V\pseries*{1-w} \lseries*{1-z}_\sT \label{eq:mVOA-translation-explicit-2}
  \end{align}
  for some element $F_{a,b} \in V\pseries{1-z, 1-w}[(1-z)^{-1},
    (1-w)^{-1}, (1-zw)^{-1}]_\sT$. We claim that
  \[ F_{a,b} \in V\pseries*{1-z, 1-w}[(1 - z)^{-1}]_\sT, \]
  so $\iota_{z,w}^\sT$ is the identity map, and therefore
  \eqref{eq:mVOA-translation-explicit-1} and
  \eqref{eq:mVOA-translation-explicit-2} combine to give the desired
  result.
  \begin{itemize}
  \item (No poles in $[(1 - w)^{-1}]_\sT$) Since $\iota_{w^{-1},
    zw}^\sT$ is the identity on $(1-w^k t^\mu)^{-1}$, but
    \eqref{eq:mVOA-translation-explicit-1} has no such poles in $w$,
    neither does $F_{a,b}$.
  \item (No poles in $[(1 - zw)^{-1}]_\sT$) Since $\iota_{z, w}^\sT$
    sends $(1-z^iw^j t^\mu)^{-1}$ (for $i, j \neq 0$) to an expression
    which has non-trivial poles of the form $(1 - w^k t^\mu)^{-1}$,
    but \eqref{eq:mVOA-translation-explicit-2} has no such poles in
    $w$, the original $F_{a,b}$ must have no poles of the form
    $(1-z^iw^j t^\mu)^{-1}$. \qedhere
  \end{itemize}
\end{proof}

\subsubsection{}

\begin{lemma} \label{lem:zw-decomposition}
  There is a decomposition of $\bk_\sT$-modules
  \[ \bk_\sT[(1 - z/w)^{-1}]_\sT^{(1)} = \bk_\sT \oplus I_{z,w} \]
  where $I_{z,w}$ consists of terms of degree $>0$ in the generating
  set. It induces a decomposition
  \[ \bk_\sT[(1 - z)^\pm, (1 - w)^\pm, (1 - z/w)^\pm]_\sT^{(1)} \cong I_{z,w}[(1 - w)^\pm]_\sT^{(1)} \oplus \bk_\sT[(1 - z)^\pm, (1 - w)^\pm]_\sT^{(1)}. \]
\end{lemma}

This decomposition, and Lemma~\ref{lem:mVOA-OPE} below, hold without
the reducedness assumption, i.e. the superscripts $(1)$, if one passes
from $\bk_\sT$ to the fraction field of an appropriate multiple cover,
like in the proof of
Lemma~\ref{lem:residue-map-of-diagonal-expansion}. We will not use
this.

\begin{proof}
  This is polynomial division with remainder: recursively apply the
  explicit formulas
  \begin{align*}
    \frac{1 - z t^\mu}{1 - (z/w)t^\nu} &= \frac{1 - w t^{\mu-\nu}}{1 - (z/w)t^\nu} + wt^{\mu-\nu}, \\
    \frac{(1 - z t^\mu)^{-1}}{1 - (z/w) t^\nu} &= \frac{(1 - wt^{\mu+\nu})^{-1}}{1 - (z/w) t^\nu} + w t^\mu (1 - zt^\mu)^{-1}(1 - wt^{\mu+\nu})^{-1}. \qedhere
  \end{align*}
\end{proof}

\subsubsection{}

\begin{definition}
  Given an element $f \in \bk_\sT\pseries{1-z, 1-w}[(1 - z)^{-1}, (1 -
    w)^{-1}, (1 - z/w)^{-1}]_\sT^{(1)}$, let
  \[ f \eqqcolon \sing_{z,w}(f) + \reg_{z,w}(f) \]
  be the decomposition induced by Lemma~\ref{lem:zw-decomposition}. In
  particular, given $a, b, c \in V$, recall that the weak
  associativity axiom implies $Y(a, z) Y(b, w) c = \iota_{w,z}^\sT
  F_{a,b,c}$ for an element $F_{a,b,c} \in V\pseries{1-z,
    1-w}[(1-z)^{-1}, (1-w)^{-1}, (1-zw)^{-1}]_\sT$. Define the {\it
    normal-ordered product}
  \[ \NO{Y(a, z) Y(b, w)}c \coloneqq \reg_{z,w} F_{a,b,c}. \]
  This terminology and notation comes from ordinary vertex algebras,
  cf. \cite[\S 2.2, \S 3.3]{Frenkel2004}. By construction, $\NO{Y(a,
    z) Y(b, w)}c$ has no poles of the form $1 - z^i w^j t^\mu$ for $i,
  j \neq 0$.
\end{definition}

\subsubsection{}

\begin{lemma}[Operator product expansion (OPE), cf. {\cite[(3.3.6)]{Frenkel2004}}] \label{lem:mVOA-OPE}
  Suppose the equivariant multiplicative vertex algebra is reduced.
  For $a, b, c \in V$:
  \begin{align*}
    Y(a, z) Y(b, w) c &= \iota_{w,z}^\sT Y\left(\sing_{z,w}(Y(a, z/w)b), w\right)c + \NO{Y(a, z) Y(b, w)} c, \\
    Y(b, w) Y(a, z) c &= \iota_{z,w}^\sT Y\left(\sing_{z,w}(Y(a, z/w)b), w\right)c + \NO{Y(b, w) Y(a, z)} c.
  \end{align*}
  Furthermore, $\NO{Y(a, z) Y(b, w)}c = \NO{Y(b, w) Y(a, z)}c$.
\end{lemma}

\begin{proof}
  Let $F_{a,b,c} \in V\pseries{1-z, 1-w}[(1-z)^{-1}, (1-w)^{-1},
    (1-zw)^{-1}]_\sT^{(1)}$ be the element whose expansions give the
  three elements $Y(a, z) Y(b, w) c$, $Y(b, w) Y(a, z) c$, and $Y(Y(a,
  z/w)b, w)c$, as in Corollary~\ref{cor:mVOA-cousin}. In particular,
  \[ Y(Y(a, z/w)b, w)c = \iota_{z/w,w}^\sT \sing_{z,w} F_{a,b,c} + \iota_{z/w,w}^\sT \reg_{w,z} F_{a,b,c}. \]
  The left hand side has poles of the form $1 - z^i w^j t^\mu$ for $i,
  j \neq 0$, coming from $Y(a, z/w)b \in V\lseries{1-z/w}_\sT^{(1)}$.
  Since $\iota_{z/w,w}^\sT$ acts trivially on $\sing_{z,w} F_{a,b,c}$,
  and $\iota_{z/w,w}^\sT \reg_{w,z} F_{a,b,c}$ has no such poles,
  \[ \sing_{z,w} F_{a,b,c} = Y(\sing_{z,w} Y(a, z/w)b, w) c. \]
  Plug this into the equations
  \begin{align*}
    Y(a, z) Y(b, w) c &= \iota_{w,z}^\sT \sing_{z,w} F_{a,b,c} + \iota_{w,z}^\sT \reg_{w,z} F_{a,b,c}, \\
    Y(b, w) Y(a, z) c &= \iota_{z,w}^\sT \sing_{z,w} F_{a,b,c} + \iota_{z,w}^\sT \reg_{w,z} F_{a,b,c},
  \end{align*}
  and note that $\iota_{w,z}^\sT$ and $\iota_{z,w}^\sT$ act trivially
  on $\reg_{w,z} F_{a,b,c} \in V\pseries{1-z, 1-w}[(1-z)^{-1},
    (1-w)^{-1}]_\sT^{(1)}$ to conclude.
\end{proof}

\subsubsection{}

\begin{example}
  In the non-equivariant (and reduced) case, given $a, b \in V$, one
  can explicitly define elements $a_{(j)} \cdot b \in V$ by
  \[ Y(a, z/w) b \eqqcolon \sum_{j=0}^{m-1} \frac{a_{(j)} \cdot b}{(1 - z/w)^{j+1}} + \cdots \]
  where $\cdots$ denotes terms with no poles in $1 - z/w$. Note the
  explicit formula
  \[ (\iota_{w,z} - \iota_{z,w}) \frac{1}{1 - z/w} = \frac{w}{1-z} \delta\left(\frac{1-z}{1-w}\right) \in \bZ\pseries*{(1-z)^\pm, (1-w)^\pm} \]
  where $\delta(x) \coloneqq \sum_{n \in \bZ} x^n$ is the formal delta
  distribution. So Lemma~\ref{lem:mVOA-OPE} implies
  \begin{equation} \label{eq:mVOA-nonequivariant-OPE}
    \left(Y(a, z)Y(b, w) - Y(b, w) Y(a, z)\right)c = \sum_{j=0}^{m-1} Y(a_{(j)} \cdot b, w) \frac{\di_z^j}{j!} \frac{w^{j+1}}{1-z}\delta\left(\frac{1-z}{1-w}\right),
  \end{equation}
  which is also sometimes referred to as the OPE. Note that $\di_z =
  -\di_{1-z}$.
\end{example}

\subsubsection{}
\label{sec:vertex-to-lie-algebra}

Let $\im(1-D(z)) \subset V$ denotes the $\bk_\sT$-submodule generated
by the coefficients of $(1 - D(z))a \in V\pseries{1-z}$ for all $a \in
V$, and define (with notation inspired by
Definition~\ref{def:rigidification})
\[ V^\pl \coloneqq V / \im(1 - D(z)). \]

\begin{theorem} \label{thm:vertex-to-lie-algebra}
  Suppose the equivariant multiplicative vertex algebra is reduced.
  Then
  \begin{align}
    V \otimes V &\to V \otimes \bk_\sT^\wedge \nonumber \\
    a \otimes b &\mapsto \rho_K\left(Y(a, z)b\right),\label{eq:vertex-lie-bracket}
  \end{align}
  where $\rho_K$ is the K-theoretic residue map
  (Definition~\ref{def:k-theoretic-residue-map}), induces a
  homomorphism
  \[ [-, -]\colon V^\pl \otimes V^\pl \to V^\pl \otimes \bk_\sT^\wedge, \]
  which becomes a Lie bracket on the $\bk_\sT^\wedge$-module $V^\pl
  \otimes \bk_\sT^\wedge$.
\end{theorem}

This should be compared to the analogous result for an ordinary vertex
algebra $V$: if $D(u)$ is the translation operator, $V/\im(D(u))$ has
a Lie bracket given by $[a, b] \coloneqq \Res_{u=0} Y(a, u)b$ \cite[\S
  4]{Borcherds1986}. The proof will follow the strategy of \cite[\S
  3.2, \S 3.3]{Frenkel2004}.

\subsubsection{}

\begin{definition} \label{def:global-vertex-algebra}
  The completion map $\bk_\sT \hookrightarrow \bk_\sT^\wedge$ is
  obviously injective. In fact $\bk_\sT$ is a pure sub-module, so the
  induced $V \to V \otimes \bk_\sT^\wedge$ is still injective. If
  \eqref{eq:vertex-lie-bracket} actually takes values in
  \[ V \subset V \otimes \bk_\sT^\wedge, \]
  then we say the vertex algebra is {\it global}. Then $[-, -]$
  becomes a Lie bracket on $V^\pl$ itself.
\end{definition}

\subsubsection{}

\begin{proof}[Proof of Theorem~\ref{thm:vertex-to-lie-algebra}.]
  First we verify that $[-, -]$ is well-defined. By the translation
  property (Lemma~\ref{lem:mVOA-translation-expansion}),
  \[ Y(a, z) D(w) b = \iota_{w^{-1}, z}^\sT D(w) Y(a, z/w) b = D(w) \iota_{w^{-1}, z} Y(a, z/w) b. \]
  Applying $\rho_{K,z}$ to both sides and using
  Lemma~\ref{lem:residue-map-of-diagonal-expansion},
  \[ \rho_{K,z} Y(a,z) D(w) b = D(w) \rho_{K,z} Y(a,z) b. \]
  Hence the operation \eqref{eq:vertex-lie-bracket} preserves $\im(1 -
  D(z))$ in its second factor. The same is true of its first factor by
  Lemma~\ref{lem:lie-bracket-skew-symmetry} below, which also shows
  that the induced $[-, -]$ is anti-symmetric.

\subsubsection{}

  \begin{lemma}[Anti-symmetry] \label{lem:lie-bracket-skew-symmetry}
    $\rho_K(Y(a,z)b) \in -\rho_K(Y(b, z)a) + \im(1 - D(z))$.
  \end{lemma}

  \begin{proof}
    By definition, $\rho_K(f(z^{-1})) = -\rho_K(f(z))$. So
    \begin{align*}
      \rho_K Y(a, z)b
      = \rho_K D(z) Y(b, z^{-1})a
      &= -\rho_K D(z^{-1}) Y(b, z) a \\
      &= -\rho_K Y(b,z)a + \rho_K (1 - D(z^{-1})) Y(b, z)a
    \end{align*}
    where the first equality is the skew symmetry axiom.
  \end{proof}

\subsubsection{}

  \begin{lemma}[Jacobi identity] \label{lem:mVOA-jacobi-identity}
    $[a, [b, c]] - [b, [a, c]] = [[a, b], c]$.
  \end{lemma}

  \begin{proof}
    Apply $\rho_{K,w}\rho_{K,z}$ to the OPE relations of
    Lemma~\ref{lem:mVOA-OPE} to get
    \begin{align*}
      [a, [b,c]] &= \rho_{K,w} Y\left(\rho_{K,z}\iota_{w,z}^\sT \sing_{w,z} Y(a, z/w)b, w\right) c + \rho_{K,w}\rho_{K,z} \NO{Y(a,z)Y(b,w)}c \\
      [b, [a,c]] &= \rho_{K,w} Y\left(\rho_{K,z}\iota_{z,w}^\sT \sing_{w,z} Y(a, z/w)b, w\right) c + \rho_{K,w}\rho_{K,z} \NO{Y(a,z)Y(b,w)}c.
    \end{align*}
    Applying Lemma~\ref{lem:residue-map-of-diagonal-expansion},
    \[ [a, [b,c]] - [b, [a,c]] = \rho_{K,w} Y\left(\rho_{K,z} Y(a, z)b, w\right) c = [[a, b], c]. \]
    Here we used that $\iota_{w,z}^\sT$ acts trivially on $\reg_{w,z}
    Y(a, z/w) b$, and that it has no contribution to $\rho_{K,z}$.
  \end{proof}

  This concludes the proof of Theorem~\ref{thm:vertex-to-lie-algebra}.
\end{proof}

\subsubsection{}

\begin{remark}
  In the non-equivariant (and reduced) setting, the Jacobi identity
  (Lemma~\ref{lem:mVOA-jacobi-identity}) for the Lie bracket is the
  equality of coefficients of the $(1 - z)^{-1} (1 - w)^{-1}$ term in
  \eqref{eq:mVOA-nonequivariant-OPE}. Putting the equality of all the
  other coefficients into an appropriate generating series yields the
  {\it Borcherds identity}; see e.g. \cite[\S 3.3.10]{Frenkel2004}.

  There appears to be no good equivariant analogue of the Borcherds
  identity. The problem is that $Y(a, z/w)b$ may have poles in $1 -
  (z/w)t^\mu$ for multiple $\mu \in \Char\sT$, and furthermore, which
  $\sT$-weights appear in these poles generally depends on the choice
  of $a, b \in V$, just like in \eqref{eq:mVOA-locality-as-equality}.
  So while it may be possible to extract identities like
  \eqref{eq:mVOA-nonequivariant-OPE} for any given $a, b, c \in V$,
  there is no universal formula for all $a, b, c \in V$, which is
  required for a Borcherds-type identity.
\end{remark}

\subsection{On K-homology of monoidal stacks}
\label{sec:mVOA-on-K-homology}

\subsubsection{}

\begin{definition} \label{def:graded-monoidal-stack}
  Let $\fM = \bigsqcup_\alpha \fM_\alpha \in \cat{Art}_\sT$, where
  $\alpha$ ranges over an additive monoid. We say $\fM$ is {\it graded
    monoidal} if:
  \begin{enumerate}
  \item (identity) $\fM_0 = *$ consists of a single point;
  \item (monoid) there exist morphisms in $\cat{Art}_\sT$
    \[ \Phi_{\alpha,\beta}\colon \fM_\alpha \times \fM_\beta \to \fM_{\alpha+\beta} \]
    such that $\Phi_{\alpha,\beta+\gamma} \circ (\id \times
    \Phi_{\beta,\gamma}) = \Phi_{\alpha+\beta,\gamma} \circ
    (\Phi_{\alpha,\beta} \times \id)$ and $\Phi_{0,\alpha} = \id =
    \Phi_{\alpha,0}$;
  \item (grading) there exist morphisms in $\cat{Art}_\sT$
    \[ \Psi_\alpha \colon [*/\bC^\times] \times \fM_\alpha \to \fM_\alpha \]
    such that $\Psi_{\alpha+\beta} \circ (\id \times
    \Phi_{\alpha,\beta}) = \Phi_{\alpha,\beta} \circ
    (\Psi_\alpha^{(12)}, \Psi_\alpha^{(13)})$ and $\Psi_\alpha \circ
    (\id \times \Psi_\alpha) = \Psi_\alpha \circ (\Omega \times \id)$.
    Here superscripts $(ij)$ mean to act on the $i$-th and $j$-th
    factors, and $\Omega\colon [*/\bC^\times]^2 \to [*/\bC^\times]$ is
    induced by the multiplication map $(\bC^\times)^2 \to \bC^\times$.
  \end{enumerate}
\end{definition}

\subsubsection{}
\label{sec:scalar-multiplication}

\begin{definition} \label{def:monoidal-stack-degree-map}
  The map $\Psi_\alpha$ induces a $\bZ$-grading on the K-theory of
  $\fM_\alpha$ as follows. Given a formal variable $z$ and any base $S
  \in \cat{Art}_\sT$, let $z^{\deg}$ be the {\it degree operator}
  \begin{equation} \label{eq:monoidal-stack-degree-map}
    z^{\deg}\colon K_{\sT}^\circ(\fM_\alpha \times S) \xrightarrow{(\Psi_\alpha \times \id)^*} K_{\sT}^\circ([*/\bC^\times] \times \fM_\alpha \times S) \cong K_\sT^\circ(\fM_\alpha \times S)[z^\pm],
  \end{equation}
  where we write $K([*/\bC^\times]) = \bZ[z^\pm]$. An element
  $\cE \in K_\sT^\circ(\fM_\alpha \times S)$ has {\it degree $k$} if
  $z^{\deg}\cE = z^k\cE$.

  For a product $\prod_i \fM^{(i)}_{\alpha_i}$, let $\Psi^{(i)}$
  be the $[*/\bC^\times]$-action on the $i$-th factor, and $\Psi$ be
  the induced diagonal $[*/\bC^\times]$-action. Write
  \[ z^{\deg_i}\colon K_{\sT}^\circ\Big(\prod_i \fM^{(i)}_{\alpha_i} \times S\Big) \to K_{\sT}^\circ\Big(\prod_i \fM^{(i)}_{\alpha_i} \times S\Big)[z^\pm] \]
  to mean $z^{\deg}$ applied to the $i$-th factor
  $\fM^{(i)}_{\alpha_i}$, i.e. applying only $\Psi_{\alpha_i}$ in
  \eqref{eq:monoidal-stack-degree-map}. For instance, $z^{\deg} =
  \prod_i z^{\deg_i}$.
\end{definition}

\subsubsection{}
\label{sec:bilinear-complex}

Throughout this subsection, $\fM \in \cat{Art}_\sT$ will be graded
monoidal with the extra data of {\it bilinear elements}
\[ \cE_{\alpha,\beta}^\bullet \in K_\sT^\circ(\fM_\alpha \times \fM_\beta) \]
of degree $\pm 1$ in the two factors, meaning that
\begin{equation} \label{eq:bilinear-complex-conditions}
\begin{alignedat}{3}
  (\Phi_{\alpha,\beta} \times \id)^*(\cE_{\alpha+\beta,\gamma}^\bullet) &= \cE_{\alpha,\gamma}^\bullet \oplus \cE_{\beta,\gamma}^\bullet \qquad &&(\Psi_\alpha \times \id)^*(\cE_{\alpha,\beta}^\bullet) &&= \cL^\vee \boxtimes \cE_{\alpha,\beta}^\bullet \\
  (\id \times \Phi_{\beta,\gamma})^*(\cE_{\alpha,\beta+\gamma}^\bullet) &= \cE_{\alpha,\beta}^\bullet \oplus \cE_{\alpha,\gamma}^\bullet \qquad &&(\id \times \Psi_\beta)^*(\cE_{\alpha,\beta}^\bullet) &&= \cL \boxtimes \cE_{\alpha,\beta}^\bullet
\end{alignedat}
\end{equation}
where $\cL_{[*/\bC^\times]} \in K_\sT^\circ([*/\bC^\times])$ is the
weight-$1$ representation and we omitted some obvious pullbacks along
projections to various factors. Its rank is therefore a bilinear
pairing
\[ \chi(\alpha, \beta) \coloneqq \rank \cE^\bullet_{\alpha,\beta}. \]

\subsubsection{}

\begin{remark} \label{rem:bilinear-complex-from-POT}
  If $\fM$ has a perfect obstruction theory given by a bilinear
  function $D(E, E)$ at a point $[E] \in \fM$, then the relative
  obstruction theory for $\Phi_{\alpha,\beta}$ is
  \begin{equation} \label{eq:bilinear-complex-as-normal-bundle}
    \cE_{\alpha,\beta}^\bullet \oplus \sigma^*\cE_{\beta,\alpha}^\bullet \,\text{ for some }\, \cE_{\alpha,\beta}^\bullet \in \cat{Perf}_\sT(\fM_\alpha \times \fM_\beta),
  \end{equation}
  where $\sigma\colon \fM_\beta \times \fM_\alpha \to \fM_\alpha
  \times \fM_\beta$ permutes the factors and
  $\cE_{\alpha,\beta}^\bullet$ is given by $D(E, F)$ at a point
  $([E],[F]) \in \fM_\alpha \times \fM_\beta$. Hence
  $\cE_{\alpha,\beta}^\bullet$ satisfies the bilinearity properties of
  \eqref{eq:bilinear-complex-conditions}.

  Monoidal stacks with such perfect obstruction theories play a
  crucial role in many other constructions in geometric representation
  theory, notably the quantum loop group action on the equivariant
  K-theory of Nakajima quiver varieties \cite{Maulik2019}. The
  multiplicative vertex algebra of
  Theorem~\ref{thm:mVOA-monoidal-stack} below can be viewed as a
  compatible, homological dual to these quantum loop groups
  \cite{Liu2022}.
\end{remark}

\subsubsection{}

\begin{theorem} \label{thm:mVOA-monoidal-stack}
  Let $\fM$ be a graded monoidal stack with a bilinear element
  (\S\ref{sec:bilinear-complex}). Let $\bfK^\sT_\circ(-)$ be a
  commutative K-homology theory
  (Definition~\ref{def:k-homology-theory}). If $\bfK^\sT_\circ([*/G])
  = \bK^\sT_\circ([*/G])$ for $G = 1$ and $G = \bC^\times$, then
  \[ \bfK^\sT_\circ(\fM) \coloneqq \bigoplus_\alpha \bfK^\sT_\circ(\fM_\alpha) \]
  has the structure of a global $\sT$-equivariant multiplicative
  vertex algebra.
\end{theorem}

The data of the multiplicative vertex algebra is given in
\S\ref{sec:mVOA-monoidal-stack-data}. Then we check the multiplicative
vertex algebra axioms. The primary difficulty, and also the reason for
finiteness axiom in the definition of $\bK^\sT_\circ(-)$, is to ensure
that the vertex product always lands in $V\lseries*{1-z}_\sT$; see
\S\ref{sec:monoidal-stack-vertex-product-well-defined}. All other
properties follow almost formally.

Theorem~\ref{thm:mVOA-monoidal-stack} is the K-theoretic version of
\cite[Theorem 4.6]{Joyce2021}, which puts an additive vertex algebra
structure on the homology group $H_*(\fM)$.

Theorem~\ref{thm:mVOA-monoidal-stack} continues to hold for localized
K-homology theories. Proposition~\ref{prop:actual-k-homology} says
that the ``concrete'' localized K-homology theory
$K_\circ^\sT(-)_{\loc}$ satisfies the hypotheses of the theorem.

\subsubsection{}
\label{sec:mVOA-monoidal-stack-data}

Here is the data of the $\sT$-equivariant multiplicative vertex
algebra.
\begin{itemize}
\item The vacuum element $\vac \in \bfK^\sT_\circ(\fM_0) \subset
  \bfK^\sT_\circ(\fM)$ is given by the identity maps
  \[ \vac_S\colon K^\circ_\sT(\fM_0 \times S) = K^\circ_\sT(S) \xrightarrow{\id} K^\circ_\sT(S), \]
  recalling that $\fM_0 = *$ is a single point by assumption.
  
\item The translation operator $D(z)
  \in \End(\bfK^\sT_\circ(\fM))\pseries{1-z}$ is given by
  \begin{equation} \label{eq:monoidal-stack-translation-operator}
    (D(z)\phi)_S(\cE) \coloneqq \phi_S(z^{\deg} \cE)
  \end{equation}
  (see
  Lemma~\ref{lem:monoidal-stack-translation-operator-well-defined}).
  Clearly $D(z)D(w) = D(zw)$. On the K-homology of a product, let
  $D(z)^{(i)}$ be the operator where $z^{\deg_i}$ is used instead of
  $z^{\deg}$, e.g. $D(z) = \prod_i D(z)^{(i)}$.
  
\item For $\phi \in \bfK^\sT_\circ(\fM_\alpha)$ and
  $\psi \in \bfK^\sT_\circ(\fM_\beta)$, the vertex product is
  \begin{equation} \label{eq:k-homology-vertex-operation}
    Y(\phi, z) \psi \coloneqq (\Phi_{\alpha,\beta})_* D(z)^{(1)} \left((\phi \boxtimes \psi) \cap \Theta_{\alpha,\beta}^\bullet(z)\right).
  \end{equation}
  Namely, its component over a base $S \in \cat{Art}_\sT$ is
  \begin{align}
    (Y(\phi, z) \psi)_S\colon K^\circ_\sT(\fM_{\alpha+\beta} \times S) &\to K^\circ_\sT(S)[(1-z)^\pm]_\sT^{(1)} \nonumber \\
    \cE &\mapsto (\phi \boxtimes \psi)_S\left( \Theta_{\alpha,\beta}^\bullet(z) \otimes z^{\deg_1} (\Phi_{\alpha,\beta} \times \id_S)^* \cE \right), \label{eq:monoidal-stack-vertex-product}
  \end{align}
  where (see \S\ref{sec:monoidal-stack-vertex-product-well-defined}
  for details, and cf. \eqref{eq:bilinear-complex-as-normal-bundle})
  \begin{equation} \label{eq:mVOA-theta}
    \Theta_{\alpha,\beta}^\bullet(z) \coloneqq \se\left(z^{-1} \cE_{\alpha,\beta}^\bullet + z \sigma^*\cE_{\beta,\alpha}^\bullet\right).
  \end{equation}
\end{itemize}

\subsubsection{}

\begin{lemma} \label{lem:monoidal-stack-translation-operator-well-defined}
  Write $\bfK_\circ^\sT([*/\bC^\times]) = \bZ[\xi]$ as in
  Proposition~\ref{prop:operational-k-homology-BGm}. Then the formula
  \eqref{eq:monoidal-stack-translation-operator} defines
  \[ D(z)\phi = \sum_{k \ge 0} (1 - z)^k \Psi_*\left(\xi^k \boxtimes \phi\right) \in \bfK_\circ^\sT(\fM)\pseries*{1-z}. \]
\end{lemma}

\begin{proof}
  This is by explicit calculation, by checking that the two sides are
  the same on any given $\cE \in K_\sT^\circ(\fM \times S)$. Write
  $K([*/\bC^\times]) = \bZ[s^\pm]$ and let $(\Psi \times \id)^*\cE =
  \bigoplus_{n \in \bZ} s^n \boxtimes \cE_n$. Then
  \begin{align*}
    \sum_{k \ge 0} (1 - z)^k \left(\Psi_*\left(\xi^k \boxtimes \phi\right)\right)_S(\cE)
    &= \sum_{k \ge 0} (1 - z)^k \sum_{n \in \bZ} \xi^k(s^n) \phi_S(\cE_n) \\
    &= \sum_{n \in \bZ} \phi_S(\cE_n) \sum_{k \ge 0} (z - 1)^k \binom{n}{k} \\
    &= \sum_{n \in \bZ} \phi_S(\cE_n) (1 + (z-1))^n
    = \phi_S(z^{\deg} \cE).
  \end{align*}
  Note that the sum in $n$ is always a finite sum.
\end{proof}

\subsubsection{}
\label{sec:monoidal-stack-vertex-product-well-defined}

We must clarify what \eqref{eq:monoidal-stack-vertex-product} and
\eqref{eq:mVOA-theta} mean. First, \eqref{eq:mVOA-theta} is not
well-defined because Definition~\ref{def:k-theoretic-euler-class} only
defines the K-theoretic Euler class $\se(-)$ on spaces with trivial
$\sT$-action. Second, it is not clear, a priori, that $Y(\phi, z)\psi$
lands in $\lseries*{1-z}^{(1)}_\sT$ nor that its components land in
$[(1-z)^\pm]_\sT^{(1)}$.

Let $\fix \coloneqq \fix_{\phi \boxtimes \psi}$ and $\fF \coloneqq
\fF_{\phi \boxtimes \psi}$ for short. If
$\Theta^\bullet_{\alpha,\beta}(z)$ were well-defined, then by the
localization axiom and push-pull,
\begin{equation} \label{eq:operational-k-homology-vertex-operation-Theta-cap}
  (\phi \boxtimes \psi) \cap \Theta^\bullet_{\alpha,\beta}(z) = \fix_*\left[(\phi \boxtimes \psi)^\sT \cap \fix^*\Theta^\bullet_{\alpha,\beta}(z)\right].
\end{equation}
We take the right hand side to be the {\it definition} of the left
hand side, using the following
Lemma~\ref{lem:operational-k-homology-vertex-operation-Theta-cap}.
Since $D(z) \in \End(\bfK_\circ^\sT(\fM))\pseries*{1-z}$, this addresses
both of the aforementioned issues.

\subsubsection{}

\begin{lemma} \label{lem:operational-k-homology-vertex-operation-Theta-cap}
  The operator $\fix^*\Theta_{\alpha,\beta}^\bullet(z) \otimes$ is
  well-defined and
  \[ \eqref{eq:operational-k-homology-vertex-operation-Theta-cap} \in \bfK_\circ^{\sT}(\fM_\alpha \times \fM_\beta)\left[(1 - z)^\pm\right]_{\sT}^{(1)}. \]
  Consequently, in \eqref{eq:k-homology-vertex-operation}, $Y(\phi,
  z)\psi \in \bfK^\sT_\circ(\fM_{\alpha+\beta})\lseries*{1 -
    z}_{\sT}^{(1)}$.
\end{lemma}

\begin{proof}
  By the localization axiom, $\sT$ acts trivially on $\fF$. Treat $z$
  as the weight of some extra $\bC^\times$ which also acts trivially
  on $\fF$. Then Definition~\ref{def:k-theoretic-euler-class} defines
  \[ \fix^*\Theta_{\alpha,\beta}^\bullet(z) \otimes \in \End(K^\circ_\sT(\fF))\lseries*{(1 - z)^{-1}}_{\sT}^{(1)}. \]
  By Lemma~\ref{lem:k-theory-inverse-euler-class}, the degree-$N$
  terms in this Laurent series are operators of multiplication by
  elements in $I^\circ(\fF)^{\otimes M(N)}$ for some function $M(N)$
  such that $\lim_{N \to \infty} M(N) = \infty$. But the finiteness
  axiom says
  \[ (\phi \boxtimes \psi)^\sT_S(I^\circ(\fF \times S)^{\otimes M}) = 0 \qquad \forall M \gg 0. \]
  Hence the Laurent series is actually a Laurent polynomial, i.e.
  \[ (\phi \boxtimes \psi)^\sT \cap \fix^*\Theta^\bullet_{\alpha,\beta}(z) \in \bfK_\circ^\sT(\fF)\left[(1 - z)^\pm\right]_{\sT}^{(1)}. \]
  The desired result follows after applying $\fix_*$.
\end{proof}

From the proof, observe that the vertex algebra is holomorphic unless
$\cE_{\alpha,\beta}^\bullet$ is truly a {\it virtual} element. This is
true of the original homological construction as well.

\subsubsection{}

\begin{proposition}[Vacuum axioms] \label{prop:operational-k-homology-vacuum-axiom}
  \begin{align*}
    Y(\vac, z)\phi &= \phi \\
    Y(\phi, z)\vac &= D(z)\phi.
  \end{align*}
\end{proposition}

\begin{proof}
  Under the isomorphism $\fM_0 \times \fM_\beta \cong \fM_\beta$, it
  is clear that $\Phi_{0,\beta}^*\cE = \cE$. Using that
  $\Theta_{0,\beta}^\bullet(z) = \cO_{\fM_\beta}$, what remains is an
  exercise in unrolling notation.
\end{proof}

\subsubsection{}

\begin{proposition}[Skew symmetry] \label{prop:operational-k-homology-skew-symmetry}
  $Y(\phi, z) \psi = D(z) Y(\psi, z^{-1}) \phi$.
\end{proposition}

\begin{proof}
  Since $z^{\deg} \Phi^* = \Phi^* z^{\deg}$, it follows that
  $D(z) \Phi_* = \Phi_* D(z)$. The right hand side becomes
  \[ (\Phi_{\beta,\alpha})_* D(z)^{(2)} \left((\psi \boxtimes \phi) \cap \Theta_{\beta,\alpha}^\bullet(z^{-1})\right). \]
  Using the symmetry $\Theta^\bullet_{\beta,\alpha}(z^{-1}) =
  \sigma^*\Theta_{\alpha,\beta}^\bullet(z)$ and the commutativity
  $\psi \boxtimes \phi = \phi \boxtimes \psi$ of the K-homology group,
  this is exactly $Y(\phi, z)\psi$.
\end{proof}

\subsubsection{}

\begin{lemma}[``Translation covariance''] \label{lem:operational-k-homology-translation-covariance}
  \begin{align*}
    (D(w)\phi \boxtimes \psi) \cap \Theta^\bullet(z) &\equiv D(w)^{(1)} \left((\phi \boxtimes \psi) \cap \Theta^\bullet(zw)\right) \\
    (\phi \boxtimes D(w)\psi) \cap \Theta^\bullet(z) &\equiv D(w)^{(2)} \left((\phi \boxtimes \psi) \cap \Theta^\bullet(z/w)\right)
  \end{align*}
\end{lemma}

\begin{proof}
  Unrolling the notation, the first claimed equality is
  \[ (\phi \boxtimes \psi)\left(w^{\deg_1} (\Theta^\bullet(z) \otimes \cE)\right) \equiv (\phi \boxtimes \psi)\left( \Theta^\bullet(zw) \otimes w^{\deg_1}\cE \right) \]
  for $\cE \in K^\sT_\circ(\fM \times \fM)$, and it is an easy
  algebraic exercise to verify that $w^{\deg_1}\Theta^\bullet(z) =
  \iota_{w,z}^\sT\Theta^\bullet(wz)$ where $\iota_{w,z}^\sT$ is the
  equivariant expansion of Definition~\ref{def:equivariant-expansion}.
  The second claimed equality is completely analogous.
\end{proof}

\subsubsection{}

\begin{proposition}[Weak associativity] \label{prop:operational-k-homology-weak-associativity}
  $Y(Y(\phi, z)\psi, w)\xi \equiv Y(\phi, zw) Y(\psi, w)\xi$.
\end{proposition}

\begin{proof}
  Written out fully, the left hand side is
  \[ (\Phi_{\alpha+\beta,\gamma})_* D(w)^{(1)} \left(\left[ (\Phi_{\alpha,\beta})_* D(z)^{(1)} ((\phi \boxtimes \psi) \cap \Theta_{\alpha,\beta}^\bullet(z)) \right] \boxtimes \xi \right) \cap \Theta_{\alpha+\beta,\gamma}^\bullet(w). \]
  Since $(\Phi_{\alpha,\beta})_*(\cdots) \otimes \xi = (\Phi_{\alpha,\beta} \times \id)_*(\cdots \otimes \xi)$, by push-pull this is equal to
  \[ (\Phi_{\alpha+\beta,\gamma})_* D(w)^{(1)} (\Phi_{\alpha,\beta}^{(12)})_* \left(\left[ D(z)^{(1)} ((\phi \boxtimes \psi) \cap \Theta_{\alpha,\beta}^\bullet(z)) \right] \boxtimes \xi \right) \cap \left(\Theta_{\alpha,\gamma}^\bullet(w) \otimes \Theta_{\beta,\gamma}^\bullet(w)\right) \]
  where $\Phi_{\alpha,\beta}^{(ij)}$ means to apply
  $\Phi_{\alpha,\beta}$ to the $i$-th and $j$-th factors. Then apply
  translation covariance to the operator $D(z)^{(1)}$ to get
  \[ (\Phi_{\alpha+\beta,\gamma})_* D(w)^{(1)} (\Phi_{\alpha,\beta}^{(12)})_* D(z)^{(1)} \left( (\phi \boxtimes \psi \boxtimes \xi) \cap (\Theta_{\alpha,\beta}^\bullet(z) \otimes \Theta_{\alpha,\gamma}^\bullet(zw) \otimes \Theta_{\beta,\gamma}^\bullet(w)) \right). \]
  Writing out the right hand side fully and applying the same
  procedure gives
  \begin{align*}
    &(\Phi_{\alpha,\beta+\gamma})_* D(zw)^{(1)} \left( \left( \phi \boxtimes \left[ (\Phi_{\beta,\gamma})_* D(w)^{(1)} ((\psi \boxtimes \xi) \cap \Theta_{\beta,\gamma}^\bullet(w)) \right]\right) \cap \Theta_{\alpha,\beta+\gamma}^\bullet(zw) \right) \\
    &\equiv (\Phi_{\alpha,\beta+\gamma})_* D(zw)^{(1)} (\Phi_{\beta,\gamma}^{(23)})_* D(w)^{(2)} \left( (\phi \boxtimes \psi \boxtimes \xi) \cap (\Theta_{\beta,\gamma}^\bullet(w) \otimes \Theta_{\alpha,\beta}^\bullet(z) \otimes \Theta_{\alpha,\gamma}^\bullet(zw)) \right)
  \end{align*}
  Finally,
  $D(w)^{(1)} (\Phi^{(12)})_* D(z)^{(1)} = (\Phi^{(12)})_* D(zw)^{(1)} D(w)^{(2)}$
  while $D(zw)^{(1)}$ commutes with $(\Phi^{(23)})_*$. We are done by
  the associativity of $\Phi$.
\end{proof}

This concludes the proof of Theorem~\ref{thm:mVOA-monoidal-stack}. \qed

\subsubsection{}

\begin{corollary} \label{cor:mVOA-monoidal-stack-lie-algebra}
  There is a Lie algebra structure on $\bfK_\circ^\sT(\fM)^\pl$ with Lie
  bracket given by
  \begin{equation} \label{eq:monoidal-stack-lie-bracket}
    [\phi, \psi]_S(\cE) = \rho_K\left[ (\phi \boxtimes \psi)_S\left(\Theta_{\alpha,\beta}(z) \otimes z^{\deg_1} (\Phi_{\alpha,\beta} \times \id_S)^*\cE\right)\right]
  \end{equation}
  for $\phi \in \bfK_\circ^\sT(\fM_\alpha)^\pl$ and $\psi \in
  \bfK_\circ^\sT(\fM_\beta)^\pl$, and $\cE \in
  K^\circ_\sT(\fM_{\alpha+\beta} \times S)$.
\end{corollary}

\begin{proof}
  Apply Theorem~\ref{thm:vertex-to-lie-algebra}, noting that
  \eqref{eq:monoidal-stack-lie-bracket} is clearly valued in $\bk_\sT
  \subset \bk_\sT^\wedge$ because the input to $\rho_K$ is a rational
  function in $z$.
\end{proof}

\subsubsection{}

\begin{definition} \label{def:rigidification}
  Let $\fM^{\pl}$ be the {\it $\bC^\times$-rigidification}
  \cite{Abramovich2008} of $\fM$ with respect to the group
  $\bC^\times$ of scaling automorphisms. Roughly, this means to
  quotient away this $\bC^\times$ from all stabilizer groups. There is
  a canonical map
  \[ \Pi_\alpha^\pl\colon \fM_\alpha \to \fM_\alpha^{\pl} \]
  which is a principal $[*/\bC^\times]$-bundle for all $\alpha \neq
  0$. The notation $\pl$ stands for {\it projective linear},
  exemplified by $[*/\GL(n)]^{\pl} = [*/\PGL(n)]$.
  
  Group actions on stacks can be complicated because the group can act
  non-trivially on stabilizers. We assume that the $\sT$-action on
  $\fM$ commutes with scaling automorphisms, and therefore descends to
  a $\sT$-action on $\fM^{\pl}$.
\end{definition}

\subsubsection{}

By construction, for $\alpha \neq 0$, the image of the pullback
$(\Pi^{\pl}_\alpha)^*\colon K_{\sT}^\circ(\fM_\alpha^{\pl}) \to
K_{\sT}^\circ(\fM_\alpha)$ consists of degree-$0$ elements. For
instance, for $n > 0$,
\[ K_\sT^\circ([*/\PGL(n)]) = \bk_\sT[s_i/s_j : i \neq j]^{S_n} \subset \bk_\sT[s_1^\pm, \ldots, s_n^\pm]^{S_n} = K_\sT^\circ([*/\GL(n)]) \]
where the superscript $S_n$ means to take permutation invariants, and
$\deg s_i = 1$ for all $1 \le i \le n$. Conversely, a coherent sheaf
on $\fM_\alpha$ descends to $\fM_\alpha^\pl$ if and only if it has
degree zero (Definition~\ref{def:monoidal-stack-degree-map}), and a
morphism descends to a morphism of rigidified stacks if and only if it
preserves degree.

\subsubsection{}

\begin{lemma} \label{lem:k-homology-pl}
  The morphism $\Pi_\alpha^\pl\colon \fM_\alpha \to \fM_\alpha^\pl$ induces
  \[ (\Pi_\alpha^\pl)_*\colon \bK_\circ^\sT(\fM_\alpha)^\pl \to \bK_\circ^\sT(\fM_\alpha^\pl). \]
  If $\Pi_\alpha^\pl$ admits a section $I_\alpha$, then this is an
  isomorphism of $\bk_\sT$-modules with inverse $(I_\alpha)_*$.
\end{lemma}

Note that $\Pi_\alpha^\pl$ admits a section if and only if it is a
trivial $[*/\bC^\times]$-bundle, i.e. $\fM_\alpha = \fM_\alpha^\pl
\times [*/\bC^\times]$ \cite[Lemma 3.21]{Laumon2000}.

\begin{proof}
  A priori, $(\Pi_\alpha^\pl)_*$ is a morphism
  $\bK_\circ^\sT(\fM_\alpha) \to \bK_\circ^\sT(\fM_\alpha^\pl)$, but
  \[ \left[(\Pi_\alpha^\pl)_*(1 - D(z)) \phi\right]_S(\cE) = \phi_S\left((\Pi_\alpha^\pl \times \id)^*\cE\right) - \phi_S\left(z^{\deg} (\Pi_\alpha^\pl \times \id)^*\cE\right) = 0 \]
  for any $\phi \in \bK_\circ^\sT(\fM_\alpha)$ and $\cE \in
  K_\sT^\circ(\fM_\alpha \times S)$. So indeed $(\Pi_\alpha^\pl)_*$
  vanishes on $\im(1 - D(z))$. The remaining claim is clear.
\end{proof}

\subsection{On K-theory of quiver moduli}
\label{sec:mVOA-on-K-theory}

\subsubsection{}

The analogue in K-theory (as opposed to K-homology) of the vertex
algebra construction of \S\ref{sec:mVOA-on-K-homology} is not
well-defined in general: $K_{\sT}(-)$ only admits proper pushforwards,
while the vertex operation \eqref{eq:k-homology-vertex-operation}
requires a pushforward along $\Phi_{\alpha,\beta}\colon \fM_\alpha
\times \fM_\beta \to \fM_{\alpha+\beta}$. This is already not proper
in the simplest setting of $\fM_\alpha = [*/\GL(\alpha)]$, for $\alpha
\in \bZ_{\ge 0}$, where the fibers are $\GL(\alpha+\beta)/(\GL(\alpha)
\times \GL(\beta))$. Nonetheless, in this subsection, we equip
$K_\sT(\fM)$, for special $\fM$, with a {\it holomorphic} equivariant
multiplicative vertex algebra structure.

\subsubsection{}

\begin{theorem} \label{thm:k-theory-vertex-operation}
  Let $\fM$ be a graded monoidal stack with bilinear elements
  $\cE_{\alpha,\beta} \in K_\sT^\circ(\fM_\alpha \times \fM_\beta)$ as
  in \S\ref{sec:bilinear-complex}. Suppose that, for every $\alpha$
  and $\beta$, there exists a factorization
  \[ \Phi_{\alpha,\beta}\colon \fM_\alpha \times \fM_\beta \xhookrightarrow{\iota_{\alpha,\beta}} \fN_{\alpha,\beta} \xrightarrow{\phi_{\alpha,\beta}} \fM_{\alpha+\beta}, \]
  in the homotopy category of $\cat{Art}_\sT$, such that:
  \begin{enumerate}
  \item $\pi_{\alpha,\beta}\colon \fN_{\alpha,\beta} \to \fM_\alpha
    \times \fM_\beta$ is the total space of a $\sT$-equivariant vector
    bundle $\cV_{\alpha,\beta}$, and $\iota_{\alpha,\beta}$ is the the
    zero section;
  \item the K-theoretic Euler class $\theta_{\alpha,\beta}^\bullet
    \coloneqq \se(\cE_{\alpha,\beta}^\bullet +
    \sigma^*\cE_{\beta,\alpha}^\bullet + \cV_{\alpha,\beta})$ is
    well-defined (cf. \eqref{eq:mVOA-theta});
  \item there is a well-defined pushforward
    $(\phi_{\alpha,\beta})_*\colon \pi_{\alpha,\beta}^*(\theta_{\alpha,\beta}^\bullet) \otimes K_{\sT}(\fN_{\alpha,\beta}) \to K_{\sT}(\fM_{\alpha+\beta})$.
  \end{enumerate}
  Then, for $\cF_\alpha \in K_\sT(\fM_\alpha)$ and $\cF_\beta \in
  K_\sT(\fM_\beta)$, the vertex operation
  \begin{equation} \label{eq:k-theory-vertex-operation}
    Y(\cF_\alpha, z)\cF_\beta \coloneqq (\phi_{\alpha,\beta})_*\pi_{\alpha,\beta}^*\left((z^{-\deg}\cF_\alpha \boxtimes \cF_\beta) \otimes \theta_{\alpha,\beta}^\bullet\right)
  \end{equation}
  makes $K_{\sT}(\fM)$ into a holomorphic equivariant multiplicative
  vertex algebra.
\end{theorem}

\begin{proof}
  We explain how \eqref{eq:k-theory-vertex-operation} compares to the
  formula \eqref{eq:k-homology-vertex-operation} for $Y(-,z)$ in
  K-homology. We drop subscripts $\alpha$ and $\beta$ to avoid
  notational clutter whenever there is no ambiguity.

  Since $\pi$ is a vector bundle, $\pi^*\colon K_\sT(\fM_\alpha \times
  \fM_\beta) \to K_\sT(\fN_{\alpha,\beta})$ is an isomorphism
  \cite[Theorem 5.4.17]{Chriss1997}, and therefore its inverse is
  $\iota^*$. Hence we may factor
  \[ \Phi_* = \phi_* \iota_* = \phi_* (\pi^*\iota^*)\iota_* = \phi_* \pi^* (\se(\cV) \otimes -) \]
  where the last equality is the self-intersection formula for
  $\iota$. Using this, and supposing that $\Theta^\bullet(z) \coloneqq
  \se(z^{-1}\cE_{\alpha,\beta}^\bullet +
  z\sigma^*\cE_{\beta,\alpha}^\bullet)$, i.e. \eqref{eq:mVOA-theta},
  were well-defined, we have
  \begin{align}
    &\Phi_* z^{-\deg_1}\left((\cF_\alpha \boxtimes \cF_\beta) \otimes \Theta^\bullet(z)\right) \label{eq:k-theory-vertex-operation-naive} \\
    &= \phi_* \pi^*\left(\se(\cV) \otimes z^{-\deg_1} \left((\cF_\alpha \boxtimes \cF_\beta) \otimes \Theta^\bullet(z)\right)\right) \nonumber \\
    &= \phi_* \pi^*\left(z^{-\deg_1}(\cF_\alpha \boxtimes \cF_\beta) \otimes z^{-\deg_1}(\se(z^{\deg_1}\cV) \otimes \Theta^\bullet(z))\right). \nonumber
  \end{align}
  But by definition $\theta^\bullet = z^{-\deg_1}(\se(z^{\deg_1}\cV)
  \otimes \Theta^\bullet(z))$. So the last line is exactly
  \eqref{eq:k-theory-vertex-operation}.

  We check vertex algebra axioms. The identity element is $\vac
  \coloneqq \cO_{\fM_0}$, and by construction $D(z) \coloneqq
  z^{-\deg}$, so \eqref{eq:k-theory-vertex-operation-naive} is
  formally identical to \eqref{eq:k-homology-vertex-operation}. The
  proof that the latter makes $\bK_\circ^\sT(\fM)$ into an equivariant
  multiplicative vertex algebra can therefore be repeated for
  $K_{\sT}(\fM)$ verbatim. Note that the minus sign in $z^{-\deg}$ is
  necessary for
  Lemma~\ref{lem:operational-k-homology-translation-covariance} to
  hold.

  Finally, clearly \eqref{eq:k-theory-vertex-operation} is a Laurent
  polynomial in $z$. It therefore lives in
  $K_\sT(\fM_{\alpha,\beta})[z^\pm] \subset
  K_\sT(\fM_{\alpha,\beta})\pseries*{1-z}$, so the vertex algebra is
  holomorphic.
\end{proof}

\subsubsection{}

The hypotheses of Theorem~\ref{thm:k-theory-vertex-operation} hold in
(and were abstracted from) the important setting of quiver
representations, which will be the focus of the remainder of this
subsection.

A holomorphic additive vertex algebra is equivalent to an algebra with
derivation; analogously, a holomorphic multiplicative vertex algebra
where $D(z) = z^D$ for a grading operator $D$ is equivalent to a
$\bZ$-graded algebra, with graded product given by $Y(-, 1)$. In what
follows, we will see $\bZ$-graded algebra structures on $K_\sT(\fM)$
which should be compared to K-theoretic Hall algebras. In particular,
\eqref{eq:quiver-moduli-vertex-operation} is a sort of shuffle product
with trivial kernel, and can be compared to the cohomological Hall
algebra computations of \cite[\S 2]{Kontsevich2011}.

\subsubsection{}

\begin{definition}
  To fix notation, let $Q$ be a quiver with vertices indexed by $i \in
  I$, and
  \[ \fM_\alpha = \bigg[\prod_{i \to j} \Hom(\bC^{\alpha_i}, \bC^{\alpha_j}) \Big/ \prod_i \GL(\alpha_i)\bigg] \]
  be its associated moduli stack of representations with finite
  dimension vector $\alpha = (\alpha_i)_i \in \bZ_{\ge 0}^{|I|}$. Note
  that $\fM_\alpha$ is a smooth Artin stack. Let $\cV_\alpha =
  \bigoplus_i \cV_{\alpha,i}$ be the universal bundle on $\fM_\alpha$,
  with $\cV_{\alpha,i}$ corresponding to the $i$-th vertex. On
  $\fM_\alpha \times \fM_\beta$, following the philosophy of
  Remark~\ref{rem:bilinear-complex-from-POT}, let
  \begin{equation} \label{eq:quiver-moduli-deformation-theory}
    \cE_{\alpha,\beta}^\bullet \coloneqq \sum_i \cV_{\alpha,i}^\vee \boxtimes \cV_{\beta,i} - \sum_{i \to j} \cV_{\alpha,i}^\vee \boxtimes \cV_{\beta,j}
  \end{equation}
  This is the ``bilinear'' version of the tangent complex of
  $\fM_\alpha$; see \cite{Kirillov2016}. Write $\cE_{\alpha,\beta}^0$
  and $\cE_{\alpha,\beta}^1$ for the first and second terms in
  \eqref{eq:quiver-moduli-deformation-theory} respectively.

  Let a torus $\sT$ act on the pre-quotient $\prod_{i \to j}
  \Hom(\bC^{\alpha_i}, \bC^{\alpha_j})$ and assume it commutes with
  the action of the stabilizer $\GL(\alpha) \coloneqq \prod_i
  \GL(\alpha_i)$. Let $\sT_{\alpha} \subset \GL(\alpha)$ be the
  maximal torus and $S_{\alpha} \coloneqq \prod_i S_{\alpha_i}$ be the
  Weyl group. Then
  \[ K_{\sT}(\fM_\alpha) \cong K_{\sT}([*/\GL(\alpha)]) = \bk_{\sT \times \GL(\alpha)} = \bk_{\sT \times \sT_\alpha}^{S_\alpha} = \bk_{\sT}[\{\vec s_i^\pm\}_{i \in I}]^{S_{\alpha}} \]
  where $\vec s_i = (s_{i,j})_{j=1}^{\alpha_i}$ is the set of
  variables permuted by $S_{\alpha_i}$. Each variable $s_{i,j}$ has
  degree one with respect to the $[*/\bC^\times]$-action induced by
  the diagonal $\bC^\times \subset \GL(\alpha)$, i.e. $z^{\deg}
  s_{i,j} = z s_{i,j}$
  (Definition~\ref{def:monoidal-stack-degree-map}).
\end{definition}

\subsubsection{}

\begin{theorem} \label{thm:quiver-moduli-vertex-algebra}
  The Artin stack $\fM = \bigsqcup_\alpha \fM_\alpha$ is graded
  monoidal, with bilinear elements $\cE_{\alpha,\beta}$ as in
  \S\ref{sec:bilinear-complex}. It satisfies the conditions of
  Theorem~\ref{thm:k-theory-vertex-operation} by setting
  \[ \fN_{\alpha,\beta} \coloneqq \bigg[\prod_{i \to j} \Hom(\bC^{\alpha_i} \oplus \bC^{\beta_i}, \bC^{\alpha_j} \oplus \bC^{\beta_j}) \Big/ \GL(\alpha) \times \GL(\beta)\bigg]. \]
  For elements $f \in K_{\sT}(\fM_\alpha) = \bk_\sT[\{\vec
    s_i^\pm\}]^{S_\alpha}$ and $g \in K_{\sT}(\fM_\beta) = \bk_\sT[\{\vec
    t_j^\pm\}^{S_\beta}]$, the resulting vertex operation is
  \begin{equation} \label{eq:quiver-moduli-vertex-operation}
    Y(f, z) g = \frac{1}{\alpha! \beta!} \sum_{\sigma \in S_{\alpha+\beta}} \sigma \cdot (f\big|_{\vec s_i\mapsto z\vec s_i} g)
  \end{equation}
  where $S_{\alpha_i+\beta_i}$ permutes
  $\{\vec s_i\} \cup \{\vec t_i\}$ and
  $\alpha! \coloneqq \prod_i \alpha_i!$ and likewise for $\beta!$.
\end{theorem}

\begin{proof}
  It is clear that $\fM$ is graded monoidal: on quiver
  representations, $\Phi$ and $\Psi$ are direct sum and the action of
  scaling automorphisms respectively. Also the
  $\cE_{\alpha,\beta}^\bullet$ are clearly bilinear.
  
  We verify the conditions of
  Theorem~\ref{thm:k-theory-vertex-operation}. The natural projection
  $\pi_{\alpha,\beta}\colon \fN_{\alpha,\beta} \to \fM_\alpha \times
  \fM_\beta$ is the vector bundle whose fibers are $\prod_{i \to j}
  \Hom(\bC^{\beta_i}, \bC^{\alpha_j}) \oplus \Hom(\bC^{\alpha_i},
  \bC^{\beta_j})$ and may therefore be identified with the total space
  of
  \[ \cV_{\alpha,\beta} \coloneqq \cE^1_{\alpha,\beta} \oplus \sigma^*\cE^1_{\beta,\alpha}. \]
  Then the K-theoretic Euler class of
  \[ \cE_{\alpha,\beta}^\bullet + \sigma^*\cE_{\beta,\alpha}^\bullet + \cV_{\alpha,\beta} = \cE_{\alpha,\beta}^0 + \sigma^*\cE_{\beta,\alpha}^0 \]
  is well-defined because it is a vector bundle. Finally, the natural
  projection $\phi_{\alpha,\beta}\colon \fN_{\alpha,\beta} \to
  \fM_{\alpha+\beta}$ is a $\GL(\alpha+\beta) / (\GL(\alpha) \times
  \GL(\beta))$-fibration, so the pushforward $(\phi_{\alpha,\beta})_*$
  may be defined and computed via $(\sT \times
  \sT_{\alpha+\beta})$-equivariant localization (cf. the case of flag
  varieties $G/B$, which produces the Weyl character formula). Namely,
  fixed points are indexed by $\sigma \in S_{\alpha+\beta}/(S_\alpha
  \times S_\beta)$, and the tangent space at $\sigma$ is
  \[ \sigma \cdot \left(\cE^0_{\alpha+\beta,\alpha+\beta} \ominus \cE^0_{\alpha,\alpha} \ominus \cE^0_{\beta,\beta}\right) = \sigma \cdot (\cE^0_{\alpha,\beta} \oplus \cE^0_{\beta,\alpha}) = \sigma \cdot \theta_{\alpha,\beta}^\bullet. \]
  Identifying $K_\sT(\fM_{\alpha+\beta}) \cong \bk_{\sT \times
    \sT_\alpha \times \sT_\beta}^{S_{\alpha+\beta}}$ and
  $K_\sT(\fM_{\alpha} \times \fM_\beta) \cong \bk_{\sT \times
    \sT_\alpha \times \sT_\beta}^{S_\alpha \times S_\beta}$, the
  pushforward is
  \begin{align*}
    (\phi_{\alpha,\beta})_*\colon \bk_{\sT \times \sT_\alpha \times \sT_\beta,\loc}^{S_\alpha \times S_\beta} &\to \bk_{\sT \times \sT_\alpha \times \sT_\beta,\loc}^{S_{\alpha+\beta}} \\
    \cF &\mapsto \sum_{\sigma \in S_{\alpha+\beta}/S_\alpha \times S_\beta} \sigma \cdot \frac{\cF}{\theta_{\alpha,\beta}^\bullet}.
  \end{align*}
  A factor of $\theta_{\alpha,\beta}^\bullet$ in $\cF$ therefore
  cancels out the denominator, and the result lands in the sub-module
  $\bk_{\sT \times \sT_\alpha \times \sT_\beta}^{S_{\alpha+\beta}}
  \subset \bk_{\sT \times \sT_\alpha \times
    \sT_\beta,\loc}^{S_{\alpha+\beta}}$, as desired.
  
  To obtain \eqref{eq:quiver-moduli-vertex-operation}, note that the
  input $\cF = f \boxtimes g \in K_{\sT}(\fM_\alpha \times \fM_\beta)$
  is already symmetric with respect to $S_\alpha \times S_\beta$,
  hence for such $\cF$ we have $\sum_{\sigma \in
    S_{\alpha+\beta}/S_\alpha \times S_\beta} = \frac{1}{\alpha!\beta!}
  \sum_{\sigma \in S_{\alpha+\beta}}$.
\end{proof}

\section{Wall-crossing}
\label{sec:wall-crossing}

\subsection{The moduli stacks}
\label{sec:moduli-stacks}

\subsubsection{}
\label{sec:moduli-stack}

Fix a Noetherian $\bC$-linear abelian category $\cat{A}$ with a moduli
stack $\fM = \bigsqcup_\alpha \fM_\alpha$ which is Artin and locally
of finite type.\footnote{With appropriate care, $\cat{A}$ may be
replaced by an exact subcategory $\cat{B} \subset \cat{A}$ closed
under isomorphisms in $\cat{A}$ (i.e. if $E \cong F$ in $\cat{A}$ with
$E \in \cat{B}$ then $F \in \cat{B}$) and direct summands in $\cat{A}$
(i.e. if $E, F \in \cat{A}$ with $E \oplus F \in \cat{B}$, then $E, F
\in \cat{B}$).} Here $\alpha$ ranges over (some quotient of) the
Grothendieck group $K_0(\cat{A})$, and $\fM_\alpha$ is the component
parameterizing objects of class $\alpha$, such that $\fM_0 = \{[0]\}$
contains only the zero object. Direct sum and composition with scaling
automorphisms induce maps
\begin{align*}
  \Phi_{\alpha,\beta} &\colon \fM_\alpha \times \fM_\beta \to \fM_{\alpha+\beta} \\
  \Psi_\alpha &\colon [*/\bC^\times] \times \fM_\alpha \to \fM_\alpha
\end{align*}
respectively, which make $\fM$ into a graded monoidal stack
(Definition~\ref{def:graded-monoidal-stack}). Assume also that $\fM$
has an obstruction theory which depends bilinearly on the universal
family of objects in $\fM$, as in
Remark~\ref{rem:bilinear-complex-from-POT}, and let
$\cE_{\alpha,\beta}^\bullet$ be the resulting bilinear elements.

In what follows, we will assume everything is equivariant for the
action of a torus $\sT$. 

\subsubsection{}

\begin{example} \label{ex:moduli-of-sheaves}
  Let $X$ be a smooth quasi-projective variety of dimension $\le 2$.
  Let $\cat{A}$ be the abelian category of coherent sheaves on $X$,
  or, more generally, some suitable heart in $D^b\cat{Coh}(X)$. A
  moduli stack $\fM$ for $\cat{A}$ is known to exist
  \cite{Lieblich2006}.
  \begin{enumerate}
  \item The maps $\Phi$ and $\Psi$ are given, respectively, by direct
    sum (of sheaves, or of quiver representations) and by composition
    with the subgroup $\bC^\times$ of scaling automorphisms;
  \item The moduli stack $\fM$ has a perfect obstruction theory given
    by $\Ext(E, E)$, which is bilinear as a function of $[E] \in \fM$;
  \end{enumerate}
  Note that (i) holds for $X$ of arbitrary dimension, but (ii) fails
  for dimension $> 2$. If a torus $\sT$ acts on $X$, then everything
  is equivariant for the induced $\sT$-action on $\fM$.
\end{example}

\subsubsection{}

\begin{lemma} \label{lem:obstruction-theory-pl}
  Let $\phi\colon \bE_{\fM} \to \bL_{\fM}$ be an obstruction theory on
  $\fM$. If $\bE_{\fM}$ has degree zero, i.e. is
  $\bC^\times$-invariant, then there is a natural obstruction theory
  $\phi^\pl\colon \bE_{\fM^\pl} \to \bL_{\fM^\pl}$ compatible under
  $\Pi^\pl$ with $\phi$.
\end{lemma}

\begin{proof}
  Consider the compatibility diagram
  \eqref{eq:obstruction-theories-compatibility} for the morphism
  $\Pi^\pl\colon \fM \to \fM^\pl$. In the given situation, we get an
  induced $\phi^\pl\colon \bF \to (\Pi^\pl)^*\bL_{\fM^\pl}$ where $\bF
  = \cone(\bE_\fM \to \bL_{\Pi^\pl})[-1]$. Both terms here are
  $\bC^\times$-invariant, hence so is $\bF$, which therefore has the
  form $\bF = (\Pi^\pl)^*\bE_{\fM^\pl}$ for some $\bE_{\fM^\pl}$. One
  can check via the long exact sequence that $\phi^\pl$ is still an
  obstruction theory.
\end{proof}

\subsubsection{}
\label{sec:pairs-stack}

\begin{definition} \label{def:pairs-stack}
  Let $\cat{Vect}$ be the category of vector spaces. A {\it framing
    functor} for $\cat{A}$ is a $\bC$-linear exact functor
  \[ \Fr\colon \cat{A}^{\Fr} \to \cat{Vect} \]
  on a full exact $\sT$-invariant subcategory $\cat{A}^{\Fr} \subset
  \cat{A}$, closed under isomorphisms in $\cat{A}$ and direct summands
  in $\cat{A}$ (i.e. if $E, F \in \cat{A}$ with $E \oplus F \in
  \cat{A}^{\Fr}$, then $E, F \in \cat{A}^{\Fr}$), such that:
  \begin{enumerate}
  \item the moduli substack $\fM^{\Fr}_\alpha \subset \fM_\alpha$, of
    objects in $\cat{A}^{\Fr}$ of class $\alpha$, is open;
  \item $\Hom(E, E) \to \Hom(\Fr(E), \Fr(E))$ is injective for all $E
    \in \cat{A}^{\Fr}$, so $\Fr(E) \neq 0$ for $E \neq 0$;
  \item $\fr(E) \coloneqq \dim \Fr(E)$ depends only on the class of
    $E$.
  \end{enumerate}
  Following \cite{Joyce2021}, wall-crossing in $\cat{A}$ is studied
  using the {\it auxiliary category of pairs}
  \[ \tilde{\cat{A}}^{\Fr} \coloneqq \{\rho\colon V \to \Fr(E) : V \in \cat{Vect}, \, E \in \cat{A}, \, \rho \text{ is linear}\} \]
  associated to a framing functor $\Fr$. A morphism from $\rho\colon V
  \to \Fr(E)$ to $\rho'\colon V' \to \Fr(E')$ in
  $\tilde{\cat{A}}^{\Fr}$ is given by a linear map $V \to V'$ and a
  morphism $E \to E'$ in $\cat{A}$ making the obvious square commute.
  A scaling automorphism $\lambda \in \bC^\times$ acts on $\rho\colon
  V \to \Fr(E)$ as $(V, E) \mapsto (\lambda V, \lambda E)$.
\end{definition}

\subsubsection{}

Let $\tilde\fM^{\Fr} = \bigsqcup_{\alpha,d}
\tilde\fM^{\Fr}_{\alpha,d}$ be the moduli stack of objects in
$\tilde{\cat{A}}^{\Fr}$ with $E \in \cat{A}^{\Fr}$, and let $\alpha$
be the class of $E$ and $d \coloneqq \dim V$. Clearly the $\sT$-action
lifts from $\fM$ to $\tilde\fM^{\Fr}$. Let
\[ \tilde\Pi^{\Fr,\pl}_{\alpha,d}\colon \tilde\fM^{\Fr}_{\alpha,d} \to \tilde\fM_{\alpha,d}^{\Fr,\pl} \]
denote the rigidification maps, and let
\[ \pi_{\fM^{\Fr}_\alpha}\colon \tilde\fM^{\Fr}_{\alpha,d} \to \fM_\alpha^{\Fr}, \qquad \pi_{\fM^{\Fr,\pl}_\alpha}\colon \tilde\fM^{\Fr,\pl}_{\alpha,d} \to \fM_\alpha^{\Fr,\pl} \]
be the forgetful maps. If $\cV$ and $\cFr(\cE)$ are the universal
bundles for $V$ and $\Fr(E)$ respectively,
\[ \tilde\fM^{\Fr}_{\alpha,d} = \tot\left(\cV^\vee \otimes \cFr(\cE)\right) \to \fM^{\Fr}_{\alpha} \times [*/\GL(d)]. \]
So $\pi_{\fM^{\Fr}_\alpha}$ is smooth, with relative cotangent
complex $\bL_{\pi_{\fM^{\Fr}_\alpha}} = \cFr(\cE)^\vee \otimes \cV -
\cV^\vee \otimes \cV \in K_\sT^\circ(\tilde\fM_{\alpha,d}^{\Fr})$.

\subsubsection{}

We only use the case $d = 1$, where the rigidification map
$\tilde\Pi^\pl_{\alpha,1}$ can be described non-canonically as fixing
an isomorphism $\bC \xrightarrow{\sim} V$, by identifying the
$[*/\bC^\times]$ fiber with the moduli stack of $1$-dimensional vector
spaces $V$. Let
\[ \tilde I_\alpha\colon \tilde\fM^{\Fr,\pl}_{\alpha,1} \to \tilde\fM^{\Fr}_{\alpha,1} \]
be the map which forgets this isomorphism. This is a non-zero section
of $\tilde\Pi_{\alpha,1}^\pl$, which is therefore a trivial
$[*/\bC^\times]$-bundle. (In particular, the second part of
Lemma~\ref{lem:k-homology-pl} applies, though we will not need it.) As
in \cite[Equation (5.26)]{Joyce2021}, there is a commutative diagram
\begin{equation} \label{eq:auxiliary-stack-rigidified-projection}
  \begin{tikzcd}[column sep=huge]
    \tilde\fM^{\Fr}_{\alpha,1} \ar[bend left]{d}{\tilde\Pi^\pl_{\alpha,1}} \ar{r}{\pi_{\fM^{\Fr}_\alpha}} & \fM^{\Fr}_\alpha \ar{d}{\Pi^\pl_\alpha} \\
    \tilde\fM^{\Fr,\pl}_{\alpha,1} \ar[bend left]{u}{\tilde I_\alpha} \ar[dashed]{ur}[swap]{\pi^\pl_{\fM^{\Fr}_\alpha}} \ar{r}[swap]{\pi_{\fM^{\Fr,\pl}_\alpha}} & \fM^{\Fr,\pl}_\alpha
  \end{tikzcd}
\end{equation}
which we use to define the dashed diagonal map
$\pi^\pl_{\fM^{\Fr}_\alpha}$ for (important!) later use in
\S\ref{sec:semistable-invariants}.

\subsubsection{}

\begin{definition} \label{def:stability-condition}
  A {\it stability condition} \cite{Joyce2021} on an abelian category
  $\cat{A}$ is a function $\tau$ from classes $\alpha$ into some
  totally-ordered set, such that, for a short exact sequence $0 \to A
  \to B \to C \to 0$ in $\cat{A}$, either
  \[ \tau(A) > \tau(B) > \tau(C) \;\text{ or }\; \tau(A) = \tau(B) = \tau(C) \;\text{ or }\; \tau(A) < \tau(B) < \tau(C). \]
  An object $E \in \cat{A}$ is {\it $\tau$-stable} (resp. {\it
    $\tau$-semistable}) if $\tau(E') < \tau(E/E')$ (resp. $\tau(E')
  \le \tau(E/E')$) for all sub-objects $E' \subset E$ with $E' \neq 0,
  E$. We say $E$ is {\it strictly $\tau$-semistable} if it is
  $\tau$-semistable but not $\tau$-stable. Let
  \[ \fM_\alpha^\st(\tau) \subset \fM_\alpha^{\sst}(\tau) \subset \fM_\alpha^\pl \]
  be the (open) moduli substacks of $\tau$-stable and
  $\tau$-semistable objects respectively, and similarly for the
  substack $\fM_\alpha^{\Fr} \subset \fM_\alpha$. Standard arguments
  show that if $E$ is stable, then $\Hom(E, E) = \bC$, so
  $\fM_\alpha^\st(\tau)$ is always an algebraic space.

  Given a stability condition $\tau$ on $\cat{A}$, and assuming
  \ref{it:stability-rank-function}, Joyce constructs \cite[Example
    5.6]{Joyce2021} a stability condition $\tilde\tau$ on
  $\tilde{\cat{A}}^{\Fr}$, with no strictly semistable objects, such
  that $\rho\colon V \to \Fr(E)$ is $\tilde\tau$-semistable if and
  only if $E$ is $\tau$-semistable, $\rho \neq 0$, and if $\rho(V)
  \subset \Fr(E')$ for some $0 \neq E' \subsetneq E$ then $\tau(E') <
  \tau(E/E')$. Consequently
  \[ \tilde\fM_\alpha^{\Fr,\st}(\tilde\tau) = \tilde\fM_\alpha^{\Fr,\sst}(\tilde\tau) \subset \tilde\fM_\alpha^{\Fr,\pl} \]
  and the first two are algebraic spaces.
\end{definition}

\subsubsection{}

\begin{example}
  Let $X$ be a smooth quasi-projective variety with an ample line
  bundle $\cL$, and let $\cat{A} = \cat{Coh}(X)$. Let $\cat{A}_k
  \subset \cat{A}$ be the exact sub-category of {\it $k$-regular}
  objects, namely $E$ such that $H^i(E \otimes \cL^{k-i}) = 0$ for all
  $i > 0$. Recall that $\cat{A}_k \subset \cat{A}_{k+1} \subset
  \cdots$ \cite[Lemma 1.7.2]{Huybrechts2010}. One can check that
  \[ \Fr_k\coloneqq E \mapsto H^0(E \otimes \cL^k) \]
  is therefore a framing functor on $\cat{A}$ which is exact on
  $\cat{A}_k$, i.e. satisfying the conditions in
  Definition~\ref{def:pairs-stack}. Then $\tilde\fM^{\Fr_k}$ becomes a
  moduli stack of pairs in the style of Joyce--Song \cite[Definition
    12.2]{Joyce2012}.

  In $\cat{A}$, one can formulate Gieseker stability in terms of a
  stability condition $\tau$ \cite[Definition 7.7]{Joyce2021}. The
  auxiliary stability condition $\tilde\tau$ from
  Definition~\ref{def:stability-condition} becomes the notion of
  stability for Joyce--Song pairs.
\end{example}

\subsection{Semistable invariants}
\label{sec:semistable-invariants}

\subsubsection{}
\label{sec:semistable-invariants-assumptions}

Throughout this subsection, fix a stability condition $\tau$ on
$\cat{A}$. We make the following assumptions (cf. \cite[Assumptions
  5.1, 5.2]{Joyce2021}).
\begin{enumerate}
\item\label{it:proper-stable-locus} If there are no strictly
  $\tau$-semistable objects of class $\alpha$, then
  $\fM_\alpha^{\sst}(\tau) = \fM_\alpha^{\st}(\tau)$ have $\sT$-fixed
  loci which are proper schemes in $\cat{C}$ (e.g. quasi-projective
  schemes; see \S\ref{sec:k-homology-support-category}), onto which
  the restriction of the obstruction theory on $\fM^{\pl}$
  (Lemma~\ref{lem:obstruction-theory-pl}) is perfect.

\item For any given class $\alpha$, there exists a framing functor
  $\Fr$ such that $\fM_\alpha^{\sst}(\tau) \subset
  \fM_\alpha^{\Fr,\pl}$.

\item\label{it:stability-rank-function} There exists a ``rank''
  function $r(\alpha) \in \bZ_{> 0}$ such that if $\tau(\alpha) =
  \tau(\beta)$ then $r(\alpha + \beta) = r(\alpha) + r(\beta)$. This
  is used in inductive proofs and to construct stability conditions on
  $\tilde{\cat{A}}^{\Fr}$ and the master space of
  \S\ref{sec:semistable-invariants-master-space}.
\end{enumerate}
Furthermore, the following assumptions, for all $\alpha$, enable us to
do enumerative geometry on
$\tilde\fM_{\alpha,1}^{\Fr,\sst}(\tilde\tau)$, and also on the master
space $\bM_{\alpha, \vec 1}^\sst(\pmb\tau)$ of
\S\ref{sec:semistable-invariants-master-space} which carries an extra
$\bC^\times$-action.
\begin{enumerate}[resume]
\item\label{it:auxiliary-proper-stable-locus} The (semi)stable loci
  $\tilde\fM_{\alpha,1}^{\Fr,\sst}(\tilde\tau)$ have $\sT$-fixed loci
  which are proper schemes in $\cat{C}$.
\item\label{it:master-space-proper-stable-locus} For any $(\sT \times
  \bC^\times)$-weight $w$ with non-trivial $\bC^\times$-component, the
  (semi)stable loci $\bM_{\alpha,\vec 1}^{\sst}(\pmb\tau)$ have
  $\sT_w$-fixed loci which are proper schemes in $\cat{C}$, where
  $\sT_w \subset \ker(w)$ is the maximal torus.
\item\label{it:auxiliary-obstruction-theories} There is a perfect
  obstruction theory on $\tilde\fM_{\alpha,1}^{\Fr,\sst}(\tilde\tau)$
  and on $\bM_{\alpha,\vec 1}^{\sst}(\pmb\tau)$, compatible
  (Definition~\ref{def:obstruction-theories}) under the forgetful maps
  $\pi_{\fM^{\Fr,\pl}_\alpha}$ and $\pi_{\fM_\alpha^{\Fr_1 \cap
      \Fr_2,\pl}}$ with the one on $\fM^{\pl}$.
\end{enumerate}

\subsubsection{}

\begin{remark}
  We comment on how these assumptions may be satisfied in practice.
  The stable (=semistable) loci in \ref{it:proper-stable-locus},
  \ref{it:auxiliary-proper-stable-locus} and
  \ref{it:master-space-proper-stable-locus} are often
  (quasi-)projective schemes via a GIT construction. Even if this is
  not the case, properness can typically be verified by semistable
  reduction, e.g. \cite[Appendix 2.B]{Huybrechts2010}. If we disregard
  the resolution property (see
  \S\ref{sec:k-homology-support-category}), being a ``proper scheme in
  $\cat{C}$'' is equivalent to being a ``proper scheme''.

  The assumption \ref{it:auxiliary-obstruction-theories} can be
  satisfied by smooth pullback, along the smooth forgetful maps, of
  the obstruction theory on $\fM$, as long as there is some mechanism
  to ensure the resulting pulled-back obstruction theories are
  perfect. This is automatic if in fact the obstruction theory on
  $\fM$ is already perfect, e.g. as in
  Example~\ref{ex:moduli-of-sheaves}. In more general settings, extra
  work is required; see Remark~\ref{rem:symmetrized-framework}.

  Note that even when a smooth pullback does not obviously exist, one
  may use {\it almost-perfect} obstruction theories \cite{Kiem2020},
  for which smooth pullbacks always exist \cite[\S 2]{Kuhn2023}, and
  which suffice for enumerative geometry \cite{Kiem2020a}. For a
  fairly non-trivial example of this, see \cite[\S 5, \S 6]{Kuhn2023}.
\end{remark}

\subsubsection{}

\begin{definition} \label{def:universal-enumerative-invariant}
  Suppose there are no $\tau$-semistable objects of class $\alpha$.
  Using assumption~\ref{it:proper-stable-locus}, define the {\it
    universal enumerative invariant} $\sZ_\alpha \in
  K_\circ^\sT(\fM_\alpha^\pl)_\loc$ by
  \begin{align*}
    \sZ_\alpha(\tau)
    &\coloneqq \chi\left(\fM^{\sst}_\alpha(\tau), \cO^\vir_{\fM^{\sst}_\alpha(\tau)} \otimes -\right) \\
    &\coloneqq \chi\left(\fM^{\sst}_\alpha(\tau)^{\sT}, \frac{\cO^\vir_{\fM^{\sst}_\alpha(\tau)^\sT}}{\se(\cN_\iota^\vir)} \otimes (j \circ \iota)^*(-)\right)
  \end{align*} 
  where $\fM^{\sst}_\alpha(\tau)^\sT \xrightarrow{\iota}
  \fM^{\sst}_\alpha(\tau) \xrightarrow{j} \fM^\pl_\alpha$ are the
  natural inclusions. The first line is purely suggestively notation;
  the second line is the actual definition, using the shorthand of
  \S\ref{sec:universal-invariants-shorthand}.

  Similarly, for any $\alpha$, since
  $\tilde\fM_{\alpha,1}^{\Fr,\sst}(\tilde\tau)$ has no strictly
  $\tilde\tau$-semistable objects, let
  \[ \tilde\cO^\vir \coloneqq \cO_{\tilde\fM_{\alpha,1}^{\Fr,\sst}(\tilde\tau)}^\vir \otimes c_{\text{rank}}^K(\bL_{\pi_{\fM_\alpha^{\Fr}}}^\vee) \]
  and define the {\it framed universal enumerative invariant}
  \begin{equation} \label{eq:universal-enumerative-invariant-pairs}
    \tilde\sZ_\alpha^{\Fr}(\tilde\tau) \coloneqq \chi\left(\tilde\fM_{\alpha,1}^{\Fr,\sst}(\tilde\tau), \tilde\cO^{\vir} \otimes (\pi^\pl_{\fM^{\Fr}_\alpha})^*(-)\right) \in K^\sT_\circ(\fM_\alpha)_{\loc},
  \end{equation}
  using the map $\pi^\pl_{\fM^{\Fr}_\alpha}$ from
  \eqref{eq:auxiliary-stack-rigidified-projection}. This is
  equivalently the universal enumerative invariant of
  $\tilde\fM_{\alpha,1}^{\Fr,\sst}(\tilde\tau) \subset
  \tilde\fM_{\alpha,1}^{\Fr,\pl}(\tilde\tau)$, capped with
  $c_{\text{rank}}^K(\bL_{\pi_{\fM_\alpha^{\Fr}}}^\vee)$, and pushed forward
  via $\pi^\pl_{\fM^{\Fr}_\alpha}$.
\end{definition}

\subsubsection{}
\label{sec:semistable-invariants-characterization}

For every class $\alpha$, even if there are strictly $\tau$-semistable
objects in $\fM_\alpha$, the goal is to define a {\it semistable
  invariant}
\[ \sZ_\alpha(\tau) \in K^\sT_\circ(\fM_\alpha^\pl)_{\loc,\bQ} \coloneqq K^\sT_\circ(\fM_\alpha^\pl)_{\loc} \otimes_{\bZ} \bQ \]
which represents ``pairing with the virtual cycle'' of the
corresponding open substack
$\fM^{\sst}_\alpha(\tau) \subset \fM_\alpha^\pl$ of $\tau$-semistable
objects. These invariants will be characterized by the following
properties.
\begin{enumerate}
\item \label{item:vss-no-strictly-semistables} If there are no
  strictly $\tau$-semistable objects of class $\alpha$, then
  $\sZ_\alpha(\tau)$ is the universal enumerative invariant of
  Definition~\ref{def:universal-enumerative-invariant}.
\item \label{item:vss-isomorphic-moduli} If $\tau, \tau'$ are two
  stability conditions with $\fM_\alpha^{\sst}(\tau) =
  \fM_\alpha^{\sst}(\tau')$, then $\sZ_\alpha(\tau) =
  \sZ_\alpha(\tau')$.
\item \label{item:vss-pairs-normalization} For any framing functor
  $\Fr$ such that $\fM^\sst_\alpha(\tau) \subset
  \fM_\alpha^{\Fr,\pl}$,
  \begin{equation} \label{eq:semistable-invariants-pairs-formula}
    \tilde\sZ_\alpha^{\Fr}(\tilde\tau) = \sum_{\substack{n>0\\\alpha_1 + \cdots + \alpha_n = \alpha\\\forall i: \,\tau(\alpha_i) = \tau(\alpha),\\\;\;\fM_{\alpha_i}^{\sst}(\tau) \neq \emptyset}} \frac{1}{n!} \Big[\sz_{\alpha_n}(\tau), \Big[\cdots \Big[\sz_{\alpha_3}(\tau), \Big[\sz_{\alpha_2}(\tau), \fr(\alpha_1) \sz_{\alpha_1}(\tau)\Big]\Big] \cdots\Big]\Big]
  \end{equation}
  for elements $\sz_\beta(\tau) \in K^\sT_\circ(\fM_\beta)_\bQ^\pl$
  such that $\sZ_\beta(\tau) = (\Pi_\beta^\pl)_* \sz_\beta(\tau)$.
  Here we are using the Lie bracket from
  Corollary~\ref{cor:mVOA-monoidal-stack-lie-algebra}.
\end{enumerate}
To be clear, the sum in \eqref{eq:semistable-invariants-pairs-formula}
is a finite sum \cite[Lemma 9.1]{Joyce2021}. Also, throughout this
subsection, we use the same symbol to denote a class in
$K_\circ^\sT(\fM_\alpha)$ and its image in
$K_\circ^\sT(\fM_\alpha)^\pl$.

\subsubsection{}

\begin{theorem}[Semistable invariants] \label{thm:semistable-invariants}
  Suppose the assumptions of
  \S\ref{sec:semistable-invariants-assumptions} hold. For classes
  $\alpha$ and framing functors $\Fr$ such that $\fM_\alpha^\sst(\tau)
  \subset \fM_\alpha^{\Fr,\pl}$, the equations
  \begin{equation} \label{eq:semistable-invariants-pairs-formula-naive}
    \tilde\sZ_\alpha^{\Fr}(\tilde\tau) = \sum_{\substack{n>0\\\alpha_1 + \cdots + \alpha_n = \alpha\\\forall i: \,\tau(\alpha_i) = \tau(\alpha),\\\;\;\fM_{\alpha_i}^{\sst}(\tau) \neq \emptyset}} \frac{1}{n!} \Big[\sz^{\Fr}_{\alpha_n}(\tau), \Big[\cdots \Big[\sz_{\alpha_3}^{\Fr}(\tau), \Big[\sz_{\alpha_2}^{\Fr}(\tau), \fr(\alpha_1) \sz_{\alpha_1}^{\Fr}(\tau)\Big]\Big] \cdots\Big]\Big]
  \end{equation}
  uniquely define elements $\sz^{\Fr}_\alpha(\tau) \in
  K^\sT_\circ(\fM_\alpha)_{\loc,\bQ}^\pl$ that are independent of
  $\Fr$, and the resulting $\sZ_\alpha(\tau) \coloneqq
  (\Pi_\alpha^\pl)_* \sz_\alpha^{\Fr}(\tau)$ satisfy
  \ref{item:vss-no-strictly-semistables} and
  \ref{item:vss-isomorphic-moduli} (and
  \ref{item:vss-pairs-normalization}, trivially).
\end{theorem}

This is a direct restatement, in our K-theoretic framework, of Joyce's
construction \cite[Theorem 5.7]{Joyce2021} of semistable invariants in
homology.

\begin{proof}
  If $\fM_\alpha^{\sst}(\tau) = \emptyset$, we define
  $\sz_\alpha^{\Fr}(\tau) \coloneqq 0$. In general,
  $\fM_\alpha^{\sst}(\tau) \subset \fM_\alpha^{\Fr,\pl}$, so if
  $\alpha_i$ appears in the decompositions in
  \eqref{eq:semistable-invariants-pairs-formula-naive} then
  $\fM_{\alpha_i}^{\sst}(\tau) \subset \fM_{\alpha_i}^{\Fr,\pl}$ as
  well, because $\cat{A}^{\Fr}$ is closed under direct summands. Thus
  the combinatorial
  Lemma~\ref{lem:semistable-invariants-combinatorics} (cf. the
  brute-force approach of \cite[\S 9.3]{Joyce2021}) and geometric
  wall-crossing Theorem~\ref{thm:semistable-invariants-WCF} below can
  be used to conclude that the elements $\sz_\alpha^{\Fr}(\tau)$ exist
  uniquely and are independent of $\Fr$.

  We verify \ref{item:vss-no-strictly-semistables}. If terms with $n >
  1$ exist in the sum in
  \eqref{eq:semistable-invariants-pairs-formula-naive}, then a choice
  of $[E_i] \in \fM_{\alpha_i}^\sst(\tau)$ for $i = 1, \ldots, n$
  gives $[E_1 \oplus \cdots \oplus E_n] \in \fM_\alpha^\sst(\tau)$,
  which is strictly $\tau$-semistable. So
  \eqref{eq:semistable-invariants-pairs-formula-naive} becomes
  \[ \tilde\sZ_\alpha^{\Fr}(\tilde\tau) = \fr(\alpha) \sz_{\alpha}^{\Fr}(\tau). \]
  Apply $(\Pi_\alpha^\pl)_*$ to both sides. Since $\Pi_\alpha^\pl
  \circ \pi^\pl_{\fM^{\Fr}_\alpha} = \pi_{\fM^{\Fr,\pl}_\alpha}$, and
  its restriction $\pi_{\fM^{\Fr,\pl}_\alpha}\colon
  \tilde\fM^{\Fr,\sst}_{\alpha,1}(\tilde\tau) \to
  \fM^\sst_\alpha(\tau)$ is a $\bP^{\fr(\alpha)-1}$-bundle by the
  description of $\tilde\tau$ in
  Definition~\ref{def:stability-condition}, we get
  \[ \chi\left(\tilde\fM_{\alpha,1}^{\Fr,\sst}(\tilde\tau), \cO^\vir \otimes \wedge_{-1}^\bullet(\bL_{\pi_{\fM_\alpha^{\Fr}}}) \otimes (\pi_{\fM^{\Fr,\pl}_\alpha})^*(-)\right) = \fr(\alpha) \chi\left(\fM_\alpha^\sst(\tau), \cO^\vir \otimes -\right) \]
  by the projection formula and the virtual projective bundle formula
  (Lemma~\ref{lem:projective-bundle-formula}). Note that
  $\bL_{\pi_{\fM_\alpha^{\Fr}}}$ is a vector bundle, so
  Lemma~\ref{lem:k-theoretic-chern-class} allows us to replace
  $c_{\text{rank}}^K$ by $\wedge_{-1}^\bullet(-)^\vee$. We are done
  because the right hand side is $\fr(\alpha) \sZ_\alpha(\tau)$ by
  definition.
  
  We verify \ref{item:vss-isomorphic-moduli}. If $\fM_\alpha^{\sst}$
  is equal for $\tau$ and $\tau'$, then the same is true for
  $\tilde\fM_{\alpha,1}^{\Fr,\sst}$ for $\tilde\tau$ and
  $\tilde\tau'$, and therefore also for $\tilde\sZ_\alpha^{\Fr}$.
  Hence the same is true of $\sz_\alpha$, and therefore of
  $\sZ_\alpha$.
\end{proof}

\subsubsection{}

Let $A$ be a monoid and $L = \bigoplus_{\alpha \in A} L_\alpha$ be an
$A$-graded Lie algebra over $\bQ$ with Lie bracket $[-, -]\colon
L_\alpha \otimes L_\beta \to L_{\alpha+\beta}$ for $\alpha, \beta \in
A$. Given two homomorphisms $\fr_1, \fr_2\colon A \to \bZ$, extend the
Lie bracket to $L \oplus \bQ \xi^{(1)} \oplus \bQ \xi^{(2)}$ by
$[\xi^{(i)}, \xi^{(j)}] = 0$ for all $i, j$ and
\[ [\xi^{(i)}, \phi] \coloneqq -[\phi, \xi^{(i)}] \coloneqq \fr_i(\alpha) \phi \; \text{ for } \phi \in L_\alpha. \]
It is easy to check that this makes $L \oplus \bQ \xi^{(1)} \oplus \bQ
\xi^{(2)}$ into a Lie algebra. We assume there is some homomorphism
$\tau\colon A \to \bZ$ such that the sums
\eqref{eq:semistable-invariants-relation} and
\eqref{eq:semistable-invariants-abstract-pairs-formula-naive} are
finite, and that there is a ``rank'' function $r(\alpha)$ as in
\ref{it:stability-rank-function}.

\begin{lemma} \label{lem:semistable-invariants-combinatorics}
  Suppose that, for $i = 1, 2$ and every $\alpha \in A$, there are
  elements $\tilde x^{(i)}_\alpha \in L_\alpha$ such that
  \begin{equation} \label{eq:semistable-invariants-relation}
    \left[\tilde x^{(1)}_\alpha, \xi^{(2)}\right] + \left[\xi^{(1)}, \tilde x^{(2)}_\alpha\right] - \sum_{\substack{\alpha_1+\alpha_2=\alpha\\\forall j: \tau(\alpha_j)=\tau(\alpha)}} \left[\tilde x^{(1)}_{\alpha_1}, \tilde x^{(2)}_{\alpha_2}\right] = 0, \qquad \forall \alpha \in A,
  \end{equation}
  holds in $L \oplus \bQ \xi_1 \oplus \bQ \xi_2$. Then the elements
  $x^{(i)}_\alpha \in L_\alpha$ uniquely defined by the formulas
  \begin{equation} \label{eq:semistable-invariants-abstract-pairs-formula-naive}
    \tilde x_\alpha^{(i)} = -\sum_{\substack{n>0\\\alpha_1 + \cdots + \alpha_n = \alpha\\\forall j: \tau(\alpha_j)=\tau(\alpha)}} \frac{1}{n!} \left[x^{(i)}_{\alpha_n}, \left[\cdots \left[ x^{(i)}_{\alpha_2}, [x^{(i)}_{\alpha_1}, \xi^{(i)}]\right] \cdots \right]\right]
  \end{equation}
  are independent of $i$, namely $x^{(1)}_\alpha = x^{(2)}_\alpha$
  for all $\alpha \in A$.
\end{lemma}

\begin{proof}
  We induct on $r(\alpha)$. If $r(\alpha) = 1$, then there are no
  splittings $\alpha = \sum_i \alpha_i$ with $\tau(\alpha_i) =
  \tau(\alpha)$ for all $i$, so
  \eqref{eq:semistable-invariants-relation} becomes $[\xi^{(1)},
    \tilde x_\alpha^{(1)}] = [\xi^{(2)}, \tilde x_\alpha^{(2)}]$ and
  \eqref{eq:semistable-invariants-abstract-pairs-formula-naive}
  becomes $\tilde x_\alpha^{(i)} = [x_\alpha^{(i)}, \xi^{(i)}]$, so
  the base case holds. By the induction hypothesis,
  \begin{equation} \label{eq:semistable-invariants-inductive}
    \tilde x^{(i)}_\alpha = -[x_{\alpha}^{(i)}, \xi^{(i)}] - \sum_{\substack{n>1\\\alpha_1 + \cdots + \alpha_n = \alpha\\\forall j: \tau(\alpha_j)=\tau(\alpha)}} \frac{1}{n!} \left[x^{(1)}_{\alpha_n}, \left[\cdots \left[ x^{(1)}_{\alpha_2}, [x^{(1)}_{\alpha_1}, \xi^{(i)}]\right] \cdots \right]\right]
  \end{equation}
  and the sum lives in $L_\alpha$. Since $\fr(\alpha)$ is invertible,
  this uniquely defines $x_\alpha^{(i)} \in L_\alpha$. To show
  $x_\alpha^{(1)} = x_\alpha^{(2)}$, plug
  \eqref{eq:semistable-invariants-inductive} into
  \eqref{eq:semistable-invariants-relation} to get
  \begin{equation} \label{eq:semistable-invariants-comparison}
    \fr_1(\alpha) \fr_2(\alpha) \left(x_\alpha^{(2)} - x_\alpha^{(1)}\right) - \sum_{\substack{n>1\\\alpha_1+\cdots+\alpha_n=\alpha\\\forall j:\tau(\alpha_j)=\tau(\alpha)}} \!\!C_{\alpha_1,\ldots,\alpha_n} = 0,
  \end{equation}
  where the first two terms are $n=1$ case of the sum
  \eqref{eq:semistable-invariants-inductive}, and
  $C_{\alpha_1,\ldots,\alpha_n}$ is
  \[ \sum_{m=0}^n \bigg[\frac{1}{m!} \Big[x_{\alpha_m}^{(1)}, \big[\cdots[x_{\alpha_2}^{(1)}, [x_{\alpha_1}^{(1)}, \xi^{(1)}]]\cdots\big]\Big], \frac{1}{(n-m)!} \Big[x_{\alpha_n}^{(1)}, \big[\cdots[x^{(1)}_{\alpha_{m+2}}, [x_{\alpha_{m+1}}^{(1)}, \xi^{(2)}]]\cdots\big]\Big]\bigg]. \]
  To complete the inductive step, it remains to show that the sum in
  \eqref{eq:semistable-invariants-comparison} vanishes. In the
  completion of the universal enveloping algebra with respect to the
  grading, consider the operator $\ad_x(-) \coloneqq [x, -]$ where $x
  \coloneqq \sum_{\beta : \tau(\beta) = \tau(\alpha)} x_\beta^{(1)}$.
  Then
  \[ \sum_{\substack{n\ge 0\\\alpha_1+\cdots+\alpha_n=\alpha\\\forall j: \tau(\alpha_j)=\tau(\alpha)}} C_{\alpha_1,\ldots,\alpha_n} = \alpha\text{-weight piece of } \left[e^{\ad_x} \xi^{(1)}, e^{\ad_x} \xi^{(2)}\right]. \]
  But a standard combinatorial result in Lie theory says $e^{\ad_u} v
  = e^u v e^{-u}$ for any $u, v$, so
  \[ \left[e^{\ad_x} \xi^{(1)}, e^{\ad_x} \xi^{(2)}\right] = [e^x \xi^{(1)} e^{-x}, e^x \xi^{(2)} e^{-x}] = e^x [\xi^{(1)}, \xi^{(2)}] e^{-x} = 0. \qedhere \]
\end{proof}

\subsubsection{}

\begin{theorem} \label{thm:semistable-invariants-WCF}
  Let $\Fr_1, \Fr_2$ be two framing functors with
  $\fM_\alpha^\sst(\tau) \subset \fM_\alpha^{\Fr_1,\pl} \cap
  \fM_\alpha^{\Fr_2,\pl}$. Then
  \begin{equation} \label{eq:semistable-invariants-WCF}
    \fr_2(\alpha) \tilde \sZ_\alpha^{\Fr_1}(\tilde\tau) - \fr_1(\alpha) \tilde \sZ_\alpha^{\Fr_2}(\tilde\tau) + \sum_{\substack{\alpha_1+\alpha_2=\alpha\\\forall j: \tau(\alpha_j)=\tau(\alpha)}} \left[\tilde \sZ_{\alpha_1}^{\Fr_1}(\tilde\tau), \tilde \sZ_{\alpha_2}^{\Fr_2}(\tilde\tau)\right] = 0.
  \end{equation}
\end{theorem}

This is the K-theoretic version of \cite[Corollary 9.10]{Joyce2021},
and we will give essentially the same geometric, wall-crossing proof.
This will occupy the remainder of this subsection. From now on, the
class $\alpha$ is fixed once and for all.

\subsubsection{}
\label{sec:semistable-invariants-master-space}

Joyce considers the abelian category of objects $E \in \cat{A}$,
vector spaces $V_1, V_2, V_3$, and morphisms
\begin{equation} \label{eq:master-space-quiver}
  \begin{tikzcd}[row sep=tiny, column sep=huge]
    & V_1 \ar{r}{\rho_1} & \Fr_1(E) \\
    V_3 \ar{ur}{\rho_3} \ar{dr}[swap]{\rho_4} \\
    & V_2 \ar{r}{\rho_2} & \Fr_2(E).
  \end{tikzcd}
\end{equation}
Objects have class $(\beta, \vec d)$ where $\beta$ is the class of
$E$ and $\vec d = (\dim V_1, \dim V_2, \dim V_3)$. Joyce defines the
stability condition
\[ \vec\tau(\beta, \vec d) \coloneqq \begin{cases} \left(\tau(\beta), (\epsilon d_1 + \epsilon d_2 + d_3) / r(\beta)\right) & \beta \neq 0 \\ (\infty, (\epsilon d_1 + \epsilon d_2 + d_3) / (d_1 + d_2 + d_3)) & \beta = 0 \end{cases} \]
for a choice of $0 < \epsilon < \frac{1}{2r(\alpha)}$. The upper bound
ensures that there are no strictly $\vec\tau$-semistable objects for
$(\alpha, \vec 1)$ where $\vec 1 \coloneqq (1, 1, 1)$.

Let $\bM = \bigsqcup_{\beta,\vec d} \bM_{\beta,\vec d}$ denote the
moduli stack of objects \eqref{eq:master-space-quiver} with $E \in
\cat{A}^{\Fr_1} \cap \cat{A}^{\Fr_2}$. Clearly the $\sT$-action lifts
from $\fM$ to $\bM$. Let
\[ \pi_{\fM_\beta^{\Fr_1 \cap \Fr_2}}\colon \bM_{\beta,\vec d} \to \fM_\beta^{\Fr_1} \cap \fM_\beta^{\Fr_2}, \qquad \pi_{\fM^{\Fr_1 \cap \Fr_2,\pl}}\colon \bM_{\beta,\vec d}^\pl \to \fM_\beta^{\Fr_1,\pl} \cap \fM_\beta^{\Fr_2,\pl} \]
be the forgetful maps. Let $\cV_i$ and $\cFr_j(\cE)$ denote the
universal bundles for $V_i$ and $\Fr_j(E)$ respectively.

\subsubsection{}

\begin{proposition}[{\cite[Propositions 9.5, 9.6]{Joyce2021}}] \label{prop:semistable-invariants-master-space}
  $M \coloneqq \bM_{\alpha,\vec 1}^{\sst}(\vec\tau)$ has a
  $\bC^\times$-action given by scaling $\rho_4$ with weight denoted
  $z$. Its fixed locus is the disjoint union of the following pieces.
  \begin{enumerate}
  \item Let $Z_{\rho_4=0} \coloneqq \{\rho_4=0\} \subset M$, with
    virtual normal bundle $z \cV_3^\vee \otimes \cV_2$. By
    $\vec\tau$-stability, $\rho_1, \rho_2, \rho_3 \neq 0$. The
    forgetful map
    \[ \pi_{\rho_4=0}\colon Z_{\rho_4=0} \to \fM_{\alpha,1}^{\Fr_1,\sst}(\tilde\tau), \]
    which remembers only $\rho_1\colon V_1 \to \Fr_1(E)$, is a
    $\bP^{\fr_2(\alpha)-1}$-bundle.
    
  \item Let $Z_{\rho_3=0} \coloneqq \{\rho_3=0\} \subset M$, with
    virtual normal bundle $z^{-1} \cV_3^\vee \otimes \cV_1$. By
    $\vec\tau$-stability, $\rho_1, \rho_2, \rho_4 \neq 0$. The
    forgetful map
    \[ \pi_{\rho_3=0}\colon Z_{\rho_3=0} \to \fM_{\alpha,1}^{\Fr_2,\sst}(\tilde\tau), \]
    which remembers only $\rho_2\colon V_2 \to \Fr_2(E)$, is a
    $\bP^{\fr_1(\alpha)-1}$-bundle.

  \item For each splitting $\alpha = \alpha_1 + \alpha_2$ with
    $\tau(\alpha_1) = \tau(\alpha_2)$, let
    \[ \iota_{\alpha_1,\alpha_2}\colon Z_{\alpha_1,\alpha_2} \coloneqq \{E = E_1 \oplus E_2, \; \rho_i\colon V_i \to \Fr_i(E_i) \subset \Fr_i(E) \text{ for } i = 1, 2\} \hookrightarrow M, \]
    where $\bC^\times$ scales $V_3$, $V_1$, and $E_1$ with weight $z$.
    By $\vec\tau$-stability, $\rho_1, \ldots, \rho_4 \neq 0$. The
    forgetful map
    \[ \pi_{\alpha_1,\alpha_2}\colon Z_{\alpha_1,\alpha_2} \to \fM_{\alpha_1,1}^{\Fr_1,\sst}(\tilde\tau) \times \fM_{\alpha_2,1}^{\Fr_2,\sst}(\tilde\tau), \]
    which remembers only $\rho_i\colon V_i \to \Fr_i(E_i)$ for $i = 1,
    2$, is an isomorphism. Under this isomorphism, the virtual normal
    bundle of $Z_{\alpha_1,\alpha_2} \subset M$ is
    \[ z^{-1} \cV_1^\vee \otimes \cFr_1(\cE_2) + z \cV_2^\vee \otimes \cFr_2(\cE_1) - (\pi^\pl_{\fM_{\alpha_1}^{\Fr_1}} \times \pi^\pl_{\fM_{\alpha_2}^{\Fr_2}})^*\left(z^{-1} \cE_{\alpha_1,\alpha_2}^\bullet + z \sigma^*\cE_{\alpha_2, \alpha_1}^\bullet\right) \]
    where $\cFr_i(\cE_j)$ is the universal bundle of $\Fr_i(E_j)$, and
    $\cE_{\alpha_i,\alpha_j}^\bullet$ are the bilinear elements
    associated to the obstruction theory on $\fM$
    (Remark~\ref{rem:bilinear-complex-from-POT}).
  \end{enumerate}
\end{proposition}
All morphisms of stacks $f\colon \fX \to \fY$ in this proposition are
compatible with virtual structure sheaves, i.e. $\cO^\vir_{\fX} =
f^*\cO^\vir_{\fY}$.

Our convention is to write $(\bC^\times \times \sT)$-equivariant
sheaves on $\bC^\times$-fixed loci as a product of an explicit
$\bC^\times$-weight and a $\sT$-equivariant sheaf. Obvious pullbacks
will be omitted.

\subsubsection{}

The standard geometric wall-crossing strategy may now be executed
using Proposition~\ref{prop:semistable-invariants-master-space}: write
the $(\bC^\times \times \sT)$-equivariant localization formula on $M$,
with an appropriately-chosen integrand, and then apply the K-theoretic
residue map $\rho_K$ (Definition~\ref{def:k-theoretic-residue-map}) to
get the desired formula \eqref{eq:semistable-invariants-WCF}. Here we
are using assumptions \ref{it:auxiliary-proper-stable-locus} and
\ref{it:auxiliary-obstruction-theories}. Let $\cO^\vir_M$ be the
virtual structure sheaf on $M$.

\subsubsection{}
\label{sec:semistable-invariants-master-space-integrand}

The integrand of interest on $M$ is the K-theoretic Chern class
$c_{\text{rank}}^K(\cG)$ where
\[ \cG \coloneqq \bL_{\pi_{\fM^{\Fr_1 \cap \Fr_2, \pl}_\alpha}}^\vee - (\cV_3^\vee \otimes \cV_1) \otimes (\cV_3^\vee \otimes \cV_2). \]
\begin{enumerate}
\item On $Z_{\rho_4=0}$, the first term in $\cG$ splits as
  \[ \bL_{\pi_{\fM^{\Fr_1 \cap \Fr_2, \pl}_\alpha}}^\vee\Big|_{Z_{\rho_4=0}} = z\cV_3^\vee \otimes \cV_2 + \bL_{\pi_{\rho_4=0}}^\vee + \pi_{\rho_4=0}^* \bL_{\pi_{\fM_\alpha^{\Fr_1}}}^\vee. \]
  The second term in $\cG$ becomes $z\cV_3^\vee \otimes \cV_2$
  because the line bundle $\cV_3^\vee \otimes \cV_1$ carries the
  non-zero section $\rho_3$ and is therefore
  $\bC^\times$-equivariantly trivial. Hence
  \[ c_{\text{rank}}^K(\cG)\Big|_{Z_{\rho_4=0}} = \wedge_{-1}^\bullet(\bL_{\pi_{\rho_4=0}}) \otimes \pi_{\rho_4=0}^*c_{\text{rank}}^K(\bL_{\pi_{\fM_\alpha^{\Fr_1}}}^\vee), \]
  using Lemma~\ref{lem:k-theoretic-chern-class} to replace
  $c_{\text{rank}}^K$ with $\wedge_{-1}^\bullet$ for the vector bundle
  $\bL_{\pi_{\rho_4=0}}$.
\item On $Z_{\rho_3=0}$, the same calculation yields
  \[ c_{\text{rank}}^K(\cG)\Big|_{Z_{\rho_3=0}} = \wedge_{-1}^\bullet(\bL_{\pi_{\rho_3=0}}) \otimes \pi_{\rho_3=0}^*c_{\text{rank}}^K(\bL_{\pi_{\fM_\alpha^{\Fr_2}}}^\vee). \]
\item On $Z_{\alpha_1,\alpha_2}$, under the isomorphism
  $\pi_{\alpha_1,\alpha_2}$, the first term in $\cG$ restricts to
  \begin{align*}
    \Big(\cV_3^\vee\otimes \cV_1 + \cV_3^\vee \otimes \cV_2
    &+ \cV_1^\vee \otimes \cFr_1(\cE_1) + z^{-1} \cV_1^\vee \otimes \cFr_1(\cE_2) \\
    &+ z\cV_2^\vee \otimes \cFr_2(\cE_1) + \cV_2^\vee \otimes \cFr_2(\cE_2)\Big) - \sum_{i=1}^3 \cV_i^\vee \otimes \cV_i
  \end{align*}
  by writing $\bM_{\alpha,\vec 1}$ as a vector bundle like in
  \S\ref{sec:semistable-invariants-master-space}. The second term in
  $\cG$ becomes $\bC^\times$-equivariantly trivial because the line
  bundles $\cV_3^\vee \otimes \cV_1$ and $\cV_3^\vee \otimes \cV_2$
  carry the non-zero sections $\rho_3$ and $\rho_4$ respectively.
  Identifying $\cV_3^\vee \otimes \cV_1 = \cO = \cV_i^\vee \otimes
  \cV_i$, the result is that
  \[ c_{\text{rank}}^K(\cG)\Big|_{Z_{\alpha_1,\alpha_2}} = \se\left(z^{-1} \cV_1^\vee \otimes \cFr_1(\cE_2) + z \cV_2^\vee \otimes \cFr_2(\cE_1)\right)\otimes c_{\text{rank}}^K(\bL_{\pi_{\fM_{\alpha_1}^{\Fr_1}}}) \otimes c_{\text{rank}}^K(\bL_{\pi_{\fM_{\alpha_2}^{\Fr_2}}}). \]
\end{enumerate}

\subsubsection{}

Putting together
Proposition~\ref{prop:semistable-invariants-master-space} and the
calculations in
\S\ref{sec:semistable-invariants-master-space-integrand}, by
$(\bC^\times \times \sT)$-equivariant localization, the projection
formula, and the virtual projective bundle formula
(Lemma~\ref{lem:projective-bundle-formula}),
\begin{equation} \label{eq:semistable-invariants-master-space-localization}
  \begin{aligned}
    &\chi\left(M, \cO_M^{\vir} \otimes c_{\text{rank}}^K(\cG) \otimes (\pi^\pl_{\fM_\alpha})^*(-)\right) \\
    &= \fr_2(\alpha) \chi\left(\tilde\fM^{\Fr_1,\sst}_{\alpha,1}(\tilde\tau), \frac{\tilde\cO^{\vir}}{1 - z^{-1} \cV_3 \otimes \cV_2^\vee} \otimes (\pi^\pl_{\fM_\alpha})^*(-)\right) \\
    &\quad+ \fr_1(\alpha) \chi\left(\tilde\fM^{\Fr_2,\sst}_{\alpha,1}(\tilde\tau), \frac{\tilde\cO^{\vir}}{1 - z \cV_3 \otimes \cV_1^\vee} \otimes (\pi^\pl_{\fM_\alpha})^*(-)\right) \\
    &\quad+ \!\!\!\!\!\!\sum_{\substack{\alpha=\alpha_1+\alpha_2\\\forall j:\,\tau(\alpha_j)=\tau(\alpha)}} \!\!\!\!\chi\left(\tilde\fM_{\alpha_1,1}^{\Fr_1,\sst}(\tilde\tau) \times \tilde\fM_{\alpha_2,1}^{\Fr_2,\sst}(\tilde\tau), \left(\tilde\cO^{\vir} \boxtimes \tilde\cO^{\vir}\right) \otimes \tilde\Theta_{\alpha_1,\alpha_2}^\bullet(z) \otimes \iota_{\alpha_1,\alpha_2}^* (\pi^\pl_{\fM_\alpha})^*(-) \right) 
  \end{aligned}
\end{equation}
where $\tilde\Theta_{\alpha_1,\alpha_2}^\bullet(z) \coloneqq
(\pi^\pl_{\fM_{\alpha_1}^{\Fr_1}} \times
\pi^\pl_{\fM_{\alpha_2}^{\Fr_2}})^*
\Theta_{\alpha_1,\alpha_2}^\bullet(z)$ with
$\Theta_{\alpha_1,\alpha_2}^\bullet(z)$ from \eqref{eq:mVOA-theta}.
Note that all $\tilde\cO^\vir$ are clearly $\bC^\times$-invariant,
i.e. have no $z$-dependence.

To be clear, recall that each term of the form $\chi(X, \cF)$ on the
right hand side is purely {\it notational shorthand} for
\[ \chi\left(X^\sT, \frac{\cF\big|_{X^\sT}}{\wedge_{-1}^\bullet (\cN^\vir_{X^\sT/X})^\vee}\right). \]
This is important because $(\bC^\times \times \sT)$-equivariant
classes like $1 - z^{-1} \cV_3 \otimes \cV_2^\vee$ on $X$ do not actually
have inverses per Lemma~\ref{lem:k-theory-inverse-euler-class}; they
must be further restricted to $X^\sT$ before the lemma applies. Also,
the notation $\pi^\pl_{\fM_\alpha}$ means the appropriate forgetful
map composed with the open inclusion into $\fM_\alpha$.

\subsubsection{}

\begin{lemma}[Pole cancellation] \label{lem:master-space-pole-cancellation}
  Assumption~\ref{it:master-space-proper-stable-locus} implies that
  \[ \chi(M, \cF) \in \bk_{\sT,\loc} \otimes_{\bZ} \bk_{\bC^\times} \]
  for any $\cF \in K_{\sT \times \bC^\times}(M)$. Consequently $\rho_K
  \chi(M, \cF) = 0$.
\end{lemma}

\begin{proof}
  This is a geometric observation from \cite[Proposition
    3.2]{Arbesfeld2021}. A priori, $\chi(M, \cF) \in \bk_{\sT \times
    \bC^\times,\loc}$ may have poles at $w = 1$ for $(\sT \times
  \bC^\times)$-weights $w$ with non-trivial $\bC^\times$-component.
  But properness of $M^{\sT_w}$ means $\sT_w$-equivariant localization
  may be used to compute the right hand side of
  \[ \chi(M, \cF)\Big|_{w=1} = \chi\left(M, \cF\Big|_{w=1}\right). \]
  The result is a well-defined element of $\bk_{\sT_w,\loc}$. In
  particular $\chi(M, \cF)$ has no pole at $w=1$ for any such $w$.
\end{proof}

\subsubsection{}

Apply the residue map $\rho_K$ to both sides of
\eqref{eq:semistable-invariants-master-space-localization}. The left
hand side vanishes by Lemma~\ref{lem:master-space-pole-cancellation}.
The first two terms on the right hand side become
\[ \fr_2(\alpha) \tilde\sZ^{\Fr_1}_\alpha(\tilde\tau) - \fr_1(\alpha) \tilde\sZ^{\Fr_2}_\alpha(\tilde\tau) \]
because $\cV_3 \otimes \cV_i^\vee$, for $i = 1, 2$, are line bundles.
Finally, using the commutative diagram
\[ \begin{tikzcd}[column sep=huge]
  \tilde\fM_{\alpha_1,1}^{\Fr_1,\sst}(\tilde\tau) \times \tilde\fM_{\alpha_2,1}^{\Fr_2,\sst}(\tilde\tau) \ar[hookrightarrow]{r}{\iota_{\alpha_1,\alpha_2}} \ar{d}[swap]{\pi^\pl_{\fM_{\alpha_1}} \times \pi^\pl_{\fM_{\alpha_2}}} & M \ar{d}{\pi^\pl_{\fM_\alpha}} \\
  \fM_{\alpha_1} \times \fM_{\alpha_2} \ar{r}{\Phi_{\alpha_1,\alpha_2}} & \fM_\alpha
\end{tikzcd} \]
and \eqref{eq:universal-invariants-tensor-product} for exterior tensor
product, the third term becomes
\[ \rho_K \left(Y(\tilde\sZ^{\Fr_1}_{\alpha_1}(\tilde\tau), z)\tilde\sZ^{\Fr_2}_{\alpha_2}(\tilde\tau)\right) = \left[\tilde\sZ^{\Fr_1}_{\alpha_1}(\tilde\tau), \tilde\sZ^{\Fr_2}_{\alpha_2}(\tilde\tau)\right]  \]
where the equality is the definition of the Lie bracket. Note that the
stacky morphism $\Phi_{\alpha_1,\alpha_2}\colon \fM_{\alpha_1} \times
\fM_{\alpha_2} \to \fM_\alpha$ has non-trivial
$\bC^\times$-equivariance, namely the scaling automorphism of the
first factor $\fM_{\alpha_1}$ has weight $z$. This is the source of
the operator $z^{\deg_1}$ in the definition
\eqref{eq:monoidal-stack-vertex-product} of the vertex product $Y(-,
z)$, in which $\Phi$ denoted the {\it non-$\bC^\times$-equivariant}
version of the map.

This concludes the proof of Theorem~\ref{thm:semistable-invariants-WCF}. \qed

\subsubsection{}

\begin{remark} \label{rem:symmetrized-framework}
  If $\fM$ has a {\it symmetric} obstruction theory, it is common and
  productive to replace every virtual structure sheaf $\cO^\vir$ and
  Euler class $\se(-)$ with their symmetrized versions
  \[ \hat\cO^\vir \coloneqq \cO^\vir \otimes \cK_{\vir}^{\frac{1}{2}}, \qquad \hat\se(-) \coloneqq \se(-) \otimes \det(-)^{\frac{1}{2}} \]
  whenever the square roots exist; see e.g. various flavors of
  K-theoretic Donaldson--Thomas theory \cite{Nekrasov2016,
    Thomas2020}. This replacement must occur not only throughout this
  section but also in the vertex algebra construction of
  \S\ref{sec:mVOA-on-K-homology}, in particular \eqref{eq:mVOA-theta}. 

  In future work \cite{Kuhn2025}, we will explain why all the main
  results of this paper, and more generally Joyce's wall-crossing
  machine, continue to hold in this symmetrized setting --- more
  precisely, where $\cat{A}$ is an equivariantly 3-Calabi--Yau
  category. The results of \cite{Liu2023} and \cite{Kuhn2023} may be
  viewed as special cases. The main technical difficulty is the
  construction of a {\it symmetrized pullback} of symmetric
  obstruction theories, in order to satisfy the
  assumption~\ref{it:auxiliary-obstruction-theories}.

  In principle, $\cO^\vir$ can be twisted by {\it any} line bundle as
  long as Euler classes and related quantities are modified
  appropriately. But twisting by $\cK_{\vir}^{1/2}$ is often
  preferred, since rigidity arguments such as in \cite[\S
    8]{Nekrasov2016} do not work as well for other twists.
\end{remark}

\appendix

\section{Appendix: Residue maps}
\label{sec:residue-maps}

\subsubsection{}

The passage from vertex algebra to Lie algebra in
\S\ref{sec:vertex-to-lie-algebra} is controlled by the residue map
$\rho_K$ of Definition~\ref{def:k-theoretic-residue-map}. In this
appendix, we define residue maps associated to a given
torus-equivariant cohomology theory, mildly generalizing the
definition in \cite{Metzler2002}, and then show the following.

\begin{proposition} \label{prop:k-theoretic-residue-map}
  $\rho_K$ is the unique residue map for the K-theory of
  Deligne--Mumford stacks.
\end{proposition}

\subsubsection{}

Let $\sT = (\bC^\times)^r$ be a torus. Let $E_{\sT}(-)$ be a $\sT$-equivariant
(complex-oriented) cohomology theory for a given class of spaces (e.g.
schemes, DM stacks, or Artin stacks), such that $E_{\sT}(-)$ is
free and split over the non-equivariant cohomology theory $E(-)$
\cite{May1996}. In particular, this means:
\begin{itemize}
\item the base ring $\bk_{\sT} \coloneqq E_{\sT}(*)$ is
  an algebra over the non-equivariant base ring $\bk \coloneqq E(*)$;
\item there is a notion of $\sT$-equivariant Euler
  class\footnote{the restriction of the $\sT$-equivariant Thom
  class to the zero section} $e(-)$ of $\sT$-equivariant line
  bundles;
\item setting $E_{\sT}(-)_{\loc} \coloneqq \bk_{\loc}
  \otimes_{\bk} E_{\sT}(-)$ for the localized base ring
  \[ \bk_{\sT,\loc} \coloneqq \bk_{\sT}[e(V)^{-1} : 1 \neq V \in \Pic_{\sT}(*)], \]
  there is an equivariant localization isomorphism
  $E_{\sT}(X^{\sT})_{\loc} \xrightarrow{\sim}
  E_{\sT}(X)_{\loc}$;
\item if $\sT$ acts trivially on $X$, then there is a
  natural isomorphism
  \[ E_{\sT}(X) = E(X) \otimes_{\bk} \bk_{\sT}, \]
  and $e(\cL)^{-1}$ exists in $E_{\sT}(X)_{\loc}$ for any $\cL
  \in \Pic_{\sT}(X)$ with non-trivial $\sT$-weight.
\end{itemize}
Examples of such $E$ include equivariant cohomology and equivariant
K-theory. 

\subsubsection{}

\begin{definition} \label{def:residue-map}
  A {\it $\sT$-equivariant residue map} for $E$ is a
  $\bk_{\sT,\loc}$-module homomorphism
  \[ \rho\colon \bk_{\sT \times \bC^\times,\loc} \to \bk_{\sT,\loc}, \]
  satisfying the properties:
  \begin{enumerate}
  \item (zero on non-localized classes) if $\iota\colon \bk_{\sT
    \times \bC^\times} \to \bk_{\sT \times \bC^\times,\loc}$ is the
    natural map, then $\rho \circ \iota = 0$;
  \item (normalization) if $X$ has trivial $(\sT \times
    \bC^\times)$-action, then
    \[ \rho(e(z \otimes \cL)^{-1}) = 1 \]
    for any $z \otimes \cL \in \Pic_{\sT \times \bC^\times}(X) =
    \Pic_{\bC^\times}(*) \otimes_{\bZ} \Pic_{\sT}(X)$ of
    $\bC^\times$-weight $1$. (To be clear, throughout, $z$ denotes the
    $\bC^\times$-representation of weight $1$.)
  \end{enumerate}
  Here, and throughout, we continue to use $\rho$ to denote the
  induced maps
  \[ \bk_{\sT \times \bC^\times,\loc} \otimes_{\bk} E(X^{\sT \times \bC^\times}) \xrightarrow{\rho \otimes \id} \bk_{\sT,\loc} \otimes_{\bk} E(X^{\sT \times \bC^\times}), \]
  which, using the equivariant localization isomorphisms, is
  equivalent to a map
  \[ \rho\colon E_{\sT \times \bC^\times}(X)_{\loc} \to E_{\sT}(X^{\bC^\times})_{\loc}. \]
\end{definition}

\subsubsection{}

For the rest of this appendix, $\sT$ will be trivial, i.e. we consider
the $\bC^\times$-equivariant cohomology theory $E_{\bC^\times}(-)$,
and we simply refer to $\rho$ as a residue map. We proceed with the
proof of Proposition~\ref{prop:k-theoretic-residue-map} in this
setting, although, with more care, the same arguments show that
$\rho_K$ is also the unique $\sT$-equivariant residue map for the
K-theory of DM stacks.

\subsubsection{}

\begin{example} \label{ex:cohomology-residue-map}
  In cohomology, $e(-)$ is the ordinary Euler class and
  \begin{equation} \label{eq:localized-ordinary-cohomology}
    \bk_{\bC^\times,\loc} = \bk_{\bC^\times}[u^{-1}], \qquad u \coloneqq e(z).
  \end{equation}
  When $X = \bP^N$, let $h \coloneqq e(\cO(1)) \in H^*(\bP^N)$ be the
  hyperplane class, so that
  \begin{equation} \label{eq:euler-class-PN-ordinary-cohomology}
    e\left(z \otimes \cO_{\bP^N}(1)\right)^{-1} = \frac{1}{u + h} = \frac{1}{u} + \sum_{k=1}^N \frac{h^k}{u^{k+1}} \in H^*_{\bC^\times}(\bP^N)_{\loc}.
  \end{equation}
  Since $N > 0$ was arbitrary and $\{h^k : 0 \le k \le N\}$ are linearly
  independent, the normalization condition for a residue map $\rho$
  therefore requires $\rho(u^{-1}) = 1$ and $\rho(u^{<-1}) = 0$. Along
  with \eqref{eq:localized-ordinary-cohomology}, this uniquely
  specifies the only possible cohomological residue map
  \[ \rho_{\text{Coh}}(f) \coloneqq \Res_{u=0}(f(u) \, du). \]
\end{example}

\subsubsection{}

Slightly more subtle is the case of $\bC^\times$-equivariant K-theory,
where $e(-) \coloneqq \wedge^\bullet_{-1}(-)^\vee$ and
\[ \bk_{\bC^\times,\loc} = \bk_{\bC^\times}\left[\frac{1}{1 - z^i} : i \in \bZ \setminus \{0\}\right] = \bk_{\bC^\times}' \oplus \bigoplus_\gamma \bigoplus_{m > 0} \frac{1}{(1 - \gamma z)^m} \bk_{\bC^\times}' \]
where the sum is over roots of unity $\gamma$ and $\bk_{\bC^\times}'
\supset \bk_{\bC^\times}$ is an extension to include roots of unity.
The analogue of \eqref{eq:euler-class-PN-ordinary-cohomology} is the
expansion
\[ \se(z \otimes \cO(1))^{-1} = \frac{1}{1 - z^{-1}} + \sum_{k=1}^N \frac{z^{-k} (\cO(-1) - 1)^k}{(1 - z^{-1})^{k+1}} \in K_{\bC^\times}(\bP^N)_{\loc}, \]
see e.g. \eqref{eq:k-theoretic-inverse-chern-root}. The normalization
condition here yields only the constraint
\begin{equation} \label{eq:residue-map-K-theory-scheme-constraint}
  \rho\left(\frac{1}{1 - z^{-1}}\right) = 1, \qquad \rho\left(\frac{z^{-k}}{(1-z^{-1})^{k+1}}\right) = 0 \qquad \forall k > 0
\end{equation}
for a residue map $\rho$, with no constraints on poles in $z$ at
non-trivial roots of unity. This will always be the case for the
K-theory of a {\it scheme} $X$, where the action of all line bundles
$\cL$ on $K(X)$ is unipotent and therefore the expansion of $e(z
\otimes \cL)^{-1}$ has poles only at $z=1$.

\subsubsection{}

\begin{remark}
  For the K-theory of schemes, the freedom to choose how a residue map
  $\rho$ behaves at non-trivial roots of unity was already observed in
  \cite{Metzler2002}, where both
  \begin{equation} \label{eq:residue-map-K-theory-naive}
    \rho_K^{\text{naive}}(f) \coloneqq \Res_{z=1}(f(z) \, z^{-1} dz)
  \end{equation}
  and $\rho_K$ are observed to satisfy
  \eqref{eq:residue-map-K-theory-scheme-constraint}. The former is the
  unique such residue map which is zero at all other poles at
  non-trivial roots of unity, see
  Lemma~\ref{lem:residue-map-K-theory-uniqueness} and
  Remark~\ref{rem:residue-map-K-theory-comparison}.
\end{remark}

\subsubsection{}

When $\fX$ is instead a DM stack, the action of line bundles $\cL$ on
$K(\fX)$ is in general only {\it quasi}-unipotent. This leads to poles
appearing at non-trivial roots of unity. Let $X = \bP(n, n, n, \ldots,
n)$ be an $N$-dimensional weighted projective space, so that
\[ K(X) = \bZ[\cL^\pm]/\inner*{(\cL^n - 1)^{N+1}} \]
by excision long exact sequence or otherwise. The correct
K-theoretic expansion is
\begin{align*}
  e(z \otimes \cL)^{-1}
  &= \frac{1 + z^{-1}\cL^{-1} + \cdots + z^{-n+1}\cL^{-n+1}}{1 - z^{-n} \cL^{-n}} \\
  &= \sum_{k=0}^N \frac{z^{-nk} (\cL^{-n} - 1)^k}{(1 - z^{-n})^{k+1}} (1 + z^{-1}\cL^{-1} + \cdots + z^{-n+1}\cL^{-n+1}) \in K_{\bC^\times}(X)_{\loc}.
\end{align*}
This now has poles at all $n$-th roots of unity. As $n, N > 0$ are
arbitrary, any residue map $\rho$ must satisfy
\begin{equation} \label{eq:residue-map-K-theory-stack-constraint}
  \rho\left(\frac{z^{-nk-a}}{(1 - z^{-n})^{k+1}}\right) = \begin{cases} 1 & k = a = 0 \\ 0 & \text{otherwise} \end{cases}
\end{equation}
for $0 \le a < n$.

\subsubsection{}

\begin{lemma} \label{lem:residue-map-K-theory-uniqueness}
  The unique residue map satisfying
  \eqref{eq:residue-map-K-theory-stack-constraint} is $\rho_K$.
\end{lemma}

This lemma concludes the proof of
Proposition~\ref{prop:k-theoretic-residue-map}.

\begin{proof}
  Let $\rho$ be a residue map and $\gamma_n$ be an $n$-th root of
  unity. The leading-order poles in
  \begin{align*}
    \frac{z^{-nk} z^{-a}}{(1 - z^{-n})^{k+1}}
    &= \sum_{i=0}^k (-1)^i \binom{k}{i} \frac{z^{-a}}{(1 - z^{-n})^{k+1-i}} \\
    &= \frac{1}{n^{k+1}} \sum_{j=0}^{n-1} \frac{\gamma_n^{-ja}}{(1 - z^{-1} \gamma_n^j)^{k+1}} + O\left(\frac{1}{(1 - z^{-n})^k}\right),
  \end{align*}
  for $0 \le a < n$, have coefficients (proportional to)
  $\gamma_n^{-ja}$ which form a Vandermonde matrix of non-zero
  determinant. Applying $\rho$ to both sides, by induction on the
  order of poles the only unknowns are
  $\rho((1 - z^{-1} \gamma_n^i)^{-k-1})$ for $0 \le i < n$, which are
  therefore uniquely determined.

  It is straightforward to verify from the definition that $\rho_K$
  satisfies \eqref{eq:residue-map-K-theory-stack-constraint}.
\end{proof}

\subsubsection{}

\begin{remark} \label{rem:residue-map-K-theory-comparison}
  The connection between $\rho_K$ and $\rho_K^{\text{naive}}$ from
  \eqref{eq:residue-map-K-theory-naive} is as follows. By definition,
  \[ \rho_K(f) = -(\Res_{z=0} + \Res_{z=\infty})\left(f(z) \, z^{-1} dz\right). \]
  A function $f \in \bk_{\bC^\times,\loc}$ has poles only at $0$, $\infty$, and
  roots of unity, so by the residue theorem
  \[ \rho_K(f) = \rho_K^{\text{naive}} + \sum_{\gamma \neq 1} \Res_{z=\gamma} \left(f(z) \, z^{-1} dz\right) \]
  where the sum is over all non-trivial roots of unity $\gamma$.
  Explicitly, for all $m > 0$ and roots of unity $\gamma$,
  \[ \rho_K\left(\frac{1}{(1 - z^{-1}\gamma)^m}\right) = 1, \qquad \rho_K^{\text{naive}}\left(\frac{1}{(1 - z^{-1}\gamma)^m}\right) = \begin{cases} 1 & \gamma = 1 \\ 0 & \text{otherwise.} \end{cases} \]
\end{remark}





\phantomsection
\addcontentsline{toc}{section}{References}

\begin{small}
\bibliographystyle{alpha}
\bibliography{mVOA}
\end{small}

\end{document}